\documentclass{article}

\usepackage{amsmath}
\usepackage{amssymb}
\usepackage{amsthm}
\usepackage{hyperref}
\usepackage{indentfirst}
\usepackage{latexsym}
\usepackage{mathrsfs}
\usepackage[a4paper,left=3cm,right=3cm,top=3cm,bottom=3cm]{geometry}
\numberwithin{equation}{section}
\allowdisplaybreaks[1]

\newtheorem{theorem}{Theorem}[section]
\newtheorem{lemma}[theorem]{Lemma}

\newtheorem{thm}[theorem]{Theorem}
\newtheorem{prop}[theorem]{Proposition}
\newtheorem{cor}[theorem]{Corollary}

\theoremstyle{remark}
\newtheorem{rem}[theorem]{Remark}


\begin{document}
\title{Existence of blowup solutions to Boussinesq equations on $\mathbb{R}^3$ with dissipative temperature}
\author{Chen Gao\footnote{Beijing International Center for Mathematical Research, Peking University, Beijing 100871, China. E-mail: gaochen@amss.ac.cn.},~~Liqun Zhang\footnote{Hua Loo-Keng Key Laboratory of Mathematics, Institute of Mathematics, AMSS, and School of Mathematical Sciences, UCAS, Beijing 100190, P. R. China. E-mail: lqzhang@math.ac.cn.},~~Xianliang Zhang\footnote{Yanqi Lake Beijing Institute of Mathematical Sciences and Applications, Beijing 101408, China and  Yau Mathematical Sciences Center, Tsinghua University, Beijing 100084, China, E-mail: zhangxianliang16@mails.ucas.ac.cn.}}
\maketitle

\begin{abstract}
The three-dimensional incompressible Boussinesq system is one of the important equations in fluid dynamics. The system describes the motion of temperature-dependent incompressible flows. And	the temperature naturally has diffusion. Recently, Elgindi, Ghoul and Masmoudi constructed a $C^{1,\alpha}$ finite time blow-up solutions for Euler systems with finite energy. Inspired by their works, we constructed $C^{1,\alpha}$ finite time blow-up solution for Boussinesq equations where the temperature has diffusion and finite energy. Generally speaking, the diffusion of temperature smooths the solution of the system which is against the formations of singularity.  The main difficulty is that the Laplace operator of the temperature equation is not coercive under the Sobolev weighted norm introduced by Elgindi. We introduced a new time depending scaling formulation and new weighted Sobolev norms, under which we obtain the nonlinear estimate. The new norm is well-coupled with the original norm, which enables us to finish the proof.
\end{abstract}

\tableofcontents
\section{Introduction}
We are concerned with the finite time blowup solutions of the three-dimensional incompressible Boussinesq systems. These systems are widely used to model the dynamics of the ocean or the atmosphere, we refer to \cite{majda2} for a rigorous justification. In this paper we shall assume that the fluid is inviscid flow with heat-conducting diffusion, hence the system reads
\begin{equation}
	\begin{cases}\label{eq}
		&\partial_{t}u+u\cdot\nabla u +\nabla p=\theta e_3,\quad (t,x)\in \mathbb{R}_+\times \mathbb{R}^3,\\
		&\partial_{t}\theta+u\cdot\nabla\theta=\Delta \theta,\\
		&\textnormal{div}\,  u=0,\\
		&u|_{ t=0}=u_{0},   \quad \theta|_{ t=0}=\theta_{0},
	\end{cases}
\end{equation}
where the velocity $u$ is a vector field with free divergence, the scalar function $\theta$ denotes the density or the temperature and $p$ is the pressure of the fluid.

Note that when the initial density $\theta_0$ is identically zero, the above system reduces to the incompressible Euler equation:
\begin{equation*}
	\partial_{t}u+u\cdot\nabla u =-\nabla p,\quad\quad \textnormal{div}\, u=0.
\end{equation*}

For incompressible three-dimensional Euler equations, whether the solution of the equation with smooth initial data of finite energy can develop a finite time
singularity has been one of the most outstanding open questions in nonlinear partial
differential equations. Relative progress could be found in the book\cite{2001Vorticity} and excellent survey papers \cite{Constantin,titi,gibbon,Kiselev,hou,hou09}. A well-known criterion for the existence
of global smooth solutions is the Beale-Kato-Majda criterion \cite{BKM}. It states that the control of the vorticity of the fluid $\omega = \nabla\times u$ in $L^1_{loc}(\mathbb{R}_+,L^{\infty})$, is sufficient to get global well-posedness. In space dimension two, it is not difficult to obtain the global well-posedness following from the Beale-Kato-Majda criterion.
The three-dimensional case is much more difficult since vorticity $\omega$ is governed by:
\begin{equation}\label{omegaeq}
	\omega_t+(u\cdot\nabla)\omega=(\omega\cdot\nabla) u.
\end{equation}

The term $(\omega\cdot\nabla) u$ on the right-hand side is referred to the vertex stretching term, which is absent in the 2D case. Note that the vortex stretching term formally has the same scaling as that of $\omega^2$. If such nonlinear alignment persists in time, the 3D Euler equations may develop a finite-time singularity, which is also the main difficulty to the global well-posedness.

The possible breakdown of solution of incompressible three-dimensional Euler equations has been studied widely in numerical simulation. There is convincing numerical evidence that the 3D axis-symmetric Euler equations may develop a potential finite time singularity (see \cite{2014Toward,Thomas2014,Thomas2107}). Where the initial data is smooth with finite energy, satisfying certain symmetric properties. Such evidence has generated great interest, see the survey article
\cite{Kiselev2}.

For mathematical justification, in \cite{Elgindi1904} Elgindi made a breakthrough by constructing self-similar blowup solutions to the 3D axis-symmetric solutions of Euler equations with $C^{1,\alpha}$ velocity and without swirl. Under these conditions, the vorticity is only non-vanishing in one direction, and the vorticity equation reduces to one equation:
\begin{equation*}
	\partial_t\omega  +u_r\partial_r\omega+ u_3\partial_{3} \omega = \frac{\omega u_r}{r},
\end{equation*}
the right-hand side is referred to the stretching term under axis-symmetric and non-swirl conditions.

Elgindi established an approximation (\ref{decomposition}) of Biot-Savart law in a variant spherical coordinate:
	\begin{equation}\label{decomposition}
		\Phi_F=\frac{1}{4\alpha}\sin(2\beta) L_{12}(F)(z)+\bar\Phi_F.
	\end{equation}

This approximation separated the solution into singular and regular part, and regular part can be controlled by (\ref{thm_biot}) under weighted Sobolev norm $\mathcal{H}^k$(see the definition in (\ref{HsNorm})).
	\begin{equation}\label{thm_biot}
		|\partial_{\beta\beta}\bar\Phi_F|_{\mathcal{H}^k}+\alpha|D_R\partial_\beta\bar\Phi_F|_{\mathcal{H}^k}+\alpha^2|D_R^2\bar\Phi_F|_{\mathcal{H}^k}\leq C|F|_{\mathcal{H}^k}.
	\end{equation}

Elgindi observes that the advection terms are relatively small compared with the nonlinear vortex stretching term when one works with $C^\alpha$ solution with small $\alpha$ for vorticity. So he can drop the advection term, and solves the solution explicitly, with following self-similar type:
$$F_*=\frac{\Gamma(\beta)}{c}\frac{2\alpha z}{(1+z)^2},$$where $\Gamma(\beta)=(\sin\beta\cos^{2}\beta)^{\alpha/3}$. Note that $F_*$ has a variables-separable form, for angular variable, the function that author choose, $\Gamma(\beta)$ is singular in~$\beta=0, \pi/2$, which benefits to obtain that the advection term is relatively small in vorticity equation.

And he performed linearization around the fundamental model and verified the coercivity of linearized operator in weighted Sobolev space $\mathcal{H}^k$. At last with elliptic estimates and energy estimates, Elgindi established the formation of singular solutions to the Euler equation,
$$\omega(x)=\frac{1}{1-t}F\left(\frac{|x|^\alpha}{(1-t)^{1+\delta}},\beta\right),$$ where~$F=F_*+g,$ where $g$ is remainder term, satisfied $|g|_{\mathcal{H}^4}\leq C\alpha^2$.

Then in \cite{Masmoudi,elgindi1906} Elgindi, Ghoul, Masmoudi overcame the lack of finite energy, they showed that the stability of the self-similar solution constructed in \cite{Elgindi1904} allows a compact supported initial vorticity to be selected. In fact, the perturbations are allowed to have a non-trivial swirl.

In \cite{Chen2021} Jiajie Chen and Thomas Y. Hou studied 2D Boussinesq equations without dissipative temperature. With a similar approximation of Biot-Savart law for 2D case, they prove finite time blowup of the Boussinesq equations with some $C^{1,\alpha}$ initial data for the velocity and the temperature. And for Boussinesq equations with dissipative temperature, Chae proved the global well-posedness for 2D Boussinesq equations in \cite{2006chae}. For the 3D case, in \cite{2012Global} Hmidi and Rousset proved the global well-posedness for Boussinesq equations with dissipative temperature and axi-symmetric initial data without swirl and some strong regularity condition.

\subsection{Main result.}
In this article, inspired by Elgindi's work \cite{Masmoudi}, we prove the existence of finite time blowup solutions to Boussinesq equations (\ref{eq}).
\begin{thm}\label{mainthm}
	There exist $\alpha_0 >0$ and the initial data $u_0\in C^{1,\alpha}(\mathbb{R}^3)$ with compact supported initial vorticity $\omega_0\in C_0^{\alpha}(\mathbb{R}^3)$, $0\not\equiv\theta_0\in C_0^{\infty}(\mathbb{R}^3)$ with $0<\alpha<\alpha_0$, such that the unique solution of the 3D Boussinesq equations $(\ref{eq})$ with the given initial data develops an asymptotically self-similar singularity in finite time, and satisfies
	$$
	\lim\limits_{t \rightarrow 1}|\omega(t)|_{L^{\infty}}=+\infty,
	$$
	$$(u,\theta)\in \left(L^2\bigcap C^{1,\alpha}_{t,x}([0,1)\times \mathbb{R}^3)\right)\times C^{2,\alpha ;3,\alpha}([0,1)\times\mathbb{R}^3)).
	$$
	
\end{thm}
\begin{rem}
The profile of the vorticity we constructed is separated into two parts under self-similar coordinate: $W=F+\varepsilon$ where $F$ is the main term introduced in \cite{Elgindi1904}. As a remainder term, the $L^\infty$ norm of~$\varepsilon$ is a lower order term compared with F, so the blowup rate which is determined by F, is $\frac{1}{1-t}$, and satisfies:
	\begin{equation}
		\lim\limits_{t\rightarrow 1}\int^t_0|\omega(s)|_{L^{\infty}}ds=+\infty.
	\end{equation}
	Note that before the blowup time, velocity field $u$ is $C^{1,\alpha}$ smooth, since the weighted Sobolev space $\mathcal{H}^k$ is embedded in H\"older continuous space $C^\alpha$. By the classical result for parabolic equations, temperature $\theta$ satisfies that $\partial^2_{tt}\theta , \partial_t\nabla^2 \theta, \nabla^3\theta$ are $C^{\alpha}$ smooth before blowup time. So the solution we constructed is strong solution.

\end{rem}
\begin{rem}
	Compared with the global well-posedness result in \cite{2012Global}, Hmidi and Rousset's result requires high regularity of velocity, which the solutions we construct following aren't satisfied.
\end{rem}

Next, we explain our main ideas. We start with 3D Boussinesq equation (\ref{eq}) under axi-symmetric and vanishing swirl case, and it is convenient to consider vorticity $\omega=\partial_3 u_r-\partial_r u_3$, rather than velocity $u$. The equation can be written as follows:
\begin{equation}\label{eq_omega_theta}
	\begin{cases}
		&\partial_t\omega+(u_r\partial_r+u_3\partial_3)\omega=\frac{u_r\omega}{r}+\partial_r\theta,\\
		&\partial_{t}\theta+(u_r\partial_r+u_3\partial_3)\theta=\Delta \theta.
	\end{cases}
\end{equation}

 For the blowup profile of the vorticity, we follow the construction in \cite{Elgindi1904} and \cite{Masmoudi}. As for the temperature, note that the parabolic equation that temperature $\theta$ is governed by has a smoothing effect and is weakly coupled with the vorticity equation since the temperature term $\partial_r\theta$ on the right-hand side of vorticity equation is linear. In the blowup mechanism of the solution we construct, we regard the temperature $\theta$ as a perturbation term that is not involved in the blowup.

 In order to obtain the stability estimate, we need to consider the effect of the temperature on the vorticity, establish the estimate of $\partial_r\theta$ which is on the right-hand side of the vorticity equation, that is, we need the estimate of the gradient of the temperature. The main difficulty we face is that the gradient of temperature is not well coupled with the weighted Sobolev norm $\mathcal{H}^k$ introduced in \cite{Elgindi1904}.

To overcome this difficulty, we introduce a different scaling formation for temperature equation and new weighted Sobolev norms $\mathcal{W}^k_1, \mathcal{W}^k_2, \mathcal{W}^k_3$(see the definition in (\ref{normdef}). More specifically, our main aim in the following sections is to establish the following estimates of the gradient of temperature:
\begin{equation}
	\frac {d}{ds}X(s)\leq -c X(s)+\frac{C}{\alpha^{3/2}}|\varepsilon(s)|_{\mathcal{H}^k}X(s)-C Y(s),
\end{equation}
Where constant $c,C>0$, X(s) is the sum of the following weighted norms for the gradient of temperature that we want to control, and $Y(s)$ is a sum of high order derivatives that come from Laplace term in temperature equation:
\begin{equation}
	\begin{array}{l}
			X(s)=|\xi(s)|_{\mathcal{W}^k_1}^2+|\xi(s)|_{\mathcal{W}^k_2}^2+|\phi(s)|_{\mathcal{W}^k_3}^2,\\
		Y(s)=|\frac{1}{\bar\rho}\partial_\beta\xi(s)|_{\mathcal{W}^{k}_1}^{2}+|\frac{1}{\bar\rho}D_{\bar\rho}\xi(s)|_{\mathcal{W}^{k}_1}^{2}+
		|\frac{1}{\bar\rho}\partial_\beta\xi(s)|_{\mathcal{W}^{k}_2}^{2}+|\frac{1}{\bar\rho}D_{\bar\rho}\xi(s)|_{\mathcal{W}^{k}_2}^{2}+
		|\frac{1}{\bar\rho}\partial_\beta\phi(s)|_{\mathcal{W}^{k}_3}^{2}+|\frac{1}{\bar\rho}D_{\bar\rho}\phi(s)|_{\mathcal{W}^{k}_3}^{2}
	\end{array}
\end{equation}
where $\xi,\phi$ are profiles of gradient of temperature $\partial_r\theta, \partial_{3}\theta$ under the self-similar coordinate.

The strategy is that
 we obtain $\mathcal{H}^k$ norm estimates of $\xi$ from $X(s)$, since we can estimate $\mathcal{H}^k$ norm by $\mathcal{W}_1^k, \mathcal{W}_2^k$ norm. And the scaling formation for temperature we establish in section  \ref{section3} ensure that $\partial_r \theta$ term from vorticity equation is eliminated in stability estimate by $Y(s)$ from the temperature equation. Then combined with the stable estimate for the remainder term of vorticity~$\varepsilon$, i.e. we set $\mathcal{E}(s)=C\alpha^{-2k+1}X(s)+|\varepsilon(s)|_{\mathcal{H}^k}^2$, then  $\mathcal{E}(s)$ controls both remainder term~$\varepsilon$ and gradient of temperature. We can establish following energy estimation formula:
\begin{equation}
	\frac {d}{ds}\mathcal{E}(s)\leq -c\mathcal{E}(s)+\frac{C}{\alpha^{3/2}}\mathcal{E}(s)^{3/2},
\end{equation}
which leads to our main Theorem \ref{mainthm}.

\section{Preliminaries}
In this section, we introduce some notations and lemma inherited from \cite{Elgindi1904,Masmoudi}, and explain the strategy we set.
\subsection{Notations}
In the following argument $r$ will denote the two dimensional radial variable:
\begin{equation}
r=\sqrt{x_1^2+x_2^2},\qquad x=(x_1,x_2,x_3).
\end{equation}
$\rho$ and $\beta $ will denote three dimensional radial variable and angle variable between plane $x_3 = 0$ and the $x_3$ axis :
\begin{equation}\label{beta}
\beta=\arctan(\frac{x_3}{r}),\quad\quad \rho=|x|=\sqrt{r^2+x_3^2}.
\end{equation}
And write $R$ :
\begin{equation}\label{R}
R=\rho^\alpha,
\end{equation}
where $\alpha>0$ is a small constant which is chosen in~\cite{Elgindi1904}. Since the vortex we will consider following is axis-symmetric and is odd in $x_3$, the ranges of $\beta$ and $R$ are $[0,\pi/2]$ and $[0,\infty)$.
The main parameters and functions used here :
\[\eta=\frac{99}{100}, \qquad \alpha>0,\qquad \gamma=1+\frac{\alpha}{10},\]
\[\Gamma(\beta)=(\sin\beta\cos^2\beta)^{\alpha/3},\] and \[K(\beta)=3\sin\beta\cos^2\beta.\]
Besides, define~$L_{12}(f)(y)$:
\begin{equation}\label{L_12}
	L_{12}(f)(y)=\int_y^\infty \int_0^{\pi/2} f(z,\beta)K(\beta)\frac{dzd\beta}{z}.
\end{equation}
For simplicity, set:
$$D_\beta(f)=\sin(2\beta)\partial_\beta f, \quad D_R(f)=R\partial_R f, \quad and \quad D_\rho(f)=\rho\partial_\rho f.$$
Define weighted Sobolev norm $\mathcal{H}^k([0,\infty)\times [0,\pi/2])$ :
\begin{equation}\label{HsNorm}
	|f|_{\mathcal{H}^k}^2=\sum_{i=0}^k\left|(D_R)^i f \frac{(1+R)^2}{R^2}\frac{1}{\sin^{\eta/2}(2\beta)}\right|_{L^2}^2+\sum_{\substack{i\geq 1\\1\leq i+j\leq k} }\left|(D_\beta)^iD_R^jf\frac{(1+R)^2}{R^2}\frac{1}{\sin^{\gamma/2}(2\beta)}\right|^2_{L^2}.\end{equation}
Recall from \cite{Elgindi1904} that $\mathcal{H}^k$ embeds continuously in $L^\infty$ and H\"older space, and in fact there exists a universal constant $C>0$ that independent of $\alpha$, such that
\begin{equation}\label{space}
		|g|_{L^{\infty}}\leq \frac{C}{\sqrt{\alpha}}|g|_{\mathcal{H}^2}\quad\quad
		|g|_{C^{\mu}}\leq \frac{C}{\sqrt{\alpha}}|g|_{\mathcal{H}^2},\quad for~any~ 0<\mu<\alpha.
\end{equation}

We introduce several new weighted Sobolev spaces, define inner products:
\begin{small}
	\begin{equation}\label{normdef}
		\begin{aligned}
			\left<f,g\right>_{\mathcal{W}_1^k}=&\int D_{\rho}^kf\cdot D_{\rho}^kg \cdot\rho^2 \sin^{2-\eta}(2\beta)d\beta d \rho+\sum_{\substack{(i,j)\neq(0,k),\\0\leq i+j\leq k}}\int D_\beta^iD_{\rho}^jf\cdot D_\beta^iD_{\rho}^j g\cdot\rho^2 \sin^{-\eta}(2\beta)d\beta d \rho,\\
			\left<f,g\right>_{\mathcal{W}_2^k}=&\int D_{\rho}^k f\cdot D_{\rho}^k g\cdot \rho^{\eta}\sin^{2-\eta}(2\beta)d\beta d \rho+\sum_{\substack{(i,j)\neq(0,k),\\0\leq i+j\leq k}}\int D_\beta^i D_{\rho}^jf\cdot D_\beta^i D_{\rho}^jg\cdot \rho^{\eta}\sin^{-\eta}(2\beta)d\beta d \rho,\\
			\left<f,g\right>_{\mathcal{W}_3^k}=&\sum_{0\leq i+j\leq k}\int D_{\rho}^i D_\beta^j f\cdot D_{\rho}^i D_\beta^j g \cdot\rho^2 \cos^{2-\eta}\beta d\beta d \rho+\\
			&+\sum_{0\leq i+j\leq k-1}\int D_\beta^{i}D_{\rho}^j \partial_\beta f\cdot D_\beta^{i}D_{\rho}^j \partial_\beta g\cdot \rho^2\cos^{2-\eta}\beta d\beta d \rho.
		\end{aligned}
	\end{equation}
\end{small}
And the corresponding norms are~$|f|_{\mathcal{W}_i^k}^2=\left<f,f\right>_{\mathcal{W}_i^k}, i=1,2,3.$

\begin{rem}
In our choice of the norm, the $\rho$ weight is the same as the $\rho$ weight in polar coordinate, for the weight of $\mathcal{W}_3^k$ norm, when we involve a $\partial_\beta$ derivative, there is one less $\sin(2\beta)$ weight than the other terms. And note that the variables in the integrals of inner products above depend on the coordinate we consider, when we work under $(\bar{y},\beta)$ coordinate, we replace $\rho$ by $\bar\rho=\bar{y}^{1/\alpha}$.
%Under the ~ $(\bar\rho,\beta) coordinate$? The radial weights in the norms $\mathcal{W}_1^k,\mathcal{W}_2^k,\mathcal{W}_3^k$ are $\bar\rho^2,\bar\rho^\eta,\bar\rho^2$
%%    ???????~$\mathcal{W}_1^k,\mathcal{W}_2^k,\mathcal{W}_3^k$?????????????~$(\bar\rho,\beta)$???????~$\bar\rho^2,\bar\rho^\eta,\bar\rho^2$??~$\bar\rho^2$????????????????????????????????~$\mathcal{H}^k$ ????~R ???????????????????????~$\bar\rho$?????????$D_{\bar\rho}=\alpha D_R$???~$\alpha$????
\end{rem}

We consider the 3D incompressible Boussinesq equations in the case that the velocity is axi-symmetric with vanishing swirl, then velocity $u$ has form: $u(t,x)=u_r(t, r, x_3)e_r+u_3(t, r, x_3)e_3$, it is convenient to consider the blowup behavior in terms of vorticity:
\begin{equation}
	\begin{aligned}
		\begin{cases}
			&\partial_t\omega+u_r\partial_r\omega+ u_3\partial_{3} \omega = \frac{\omega u_r}{r}+ \partial_{r}\theta,\\
			&\partial_t \theta+u_r\partial_r\theta+ u_3\partial_{3} \theta = \Delta \theta.
		\end{cases}
	\end{aligned}
\end{equation}
By Biot-Savart law, we have stream function $\tilde{\Phi}$ that satisfies:
\begin{equation}
	\begin{cases}
		\partial_r(\frac{1}{r}\partial_r \tilde{\Phi})+\frac{1}{r}\partial_{33}\tilde{\Phi}=-\omega,\\
		u_r=\frac{1}{r}\partial_3\tilde{\Phi},\ \ \ \ u_3=-\frac{1}{r}\partial_r\tilde{\Phi}.
	\end{cases}
\end{equation}
For temperature function $\theta$, we consider
the gradient in the directions of $r,x_3$ denoted by $\theta_r, \theta_3$, and set $r\Phi=\tilde{\Phi}$, the system becomes:
%?????$r,x_3$ $\theta_r, \theta_{3},$
%
%%??????? $\theta_r, \theta_{3},$ ?????????????? $\theta,$
%turns out to be
%??Boussinesq ??????
\begin{equation}\label{cyeq}
	\begin{cases}
		\partial_t\omega+(u_r\partial_r+u_3\partial_3)\omega=\frac{u_r\omega}{r}+\partial_r\theta,\\
		-\partial_{rr}\Phi-\partial_{33}\Phi-\frac{1}{r}\partial_r\Phi+\frac{\Phi}{r^2}=\omega,\\
		\partial_t\theta_r+(u_r\partial_r+u_3\partial_3)\theta_r+(\partial_r u_r)\theta_r+(\partial_r u_3 )\theta_3=\Delta\theta_r-\frac{\theta_r}{r^2},\\
		\partial_t\theta_3+(u_r\partial_r+u_3\partial_3)\theta_3+(\partial_3 u_r)\theta_r+(\partial_3 u_3) \theta_3=\Delta\theta_3.
	\end{cases}
\end{equation}
We choose that the initial value of $\omega_0,\theta_0$ is odd in $x_3$, so because of the equation $(\ref{eq_omega_theta})$,the vorticity $\omega$ streaming function $\Phi$ and temperature $\theta$ still keep this symmetry:
%?????????????~$\omega_0,\theta_0$?? $x_3$ ????????????~$(\ref{eq_omega_theta})$????~$\omega$ ???~$\theta$ ?????????
\begin{equation}
	\omega(r,-x_3)=-\omega(r,x_3),\quad\theta(r,-x_3)=-\theta(r,x_3),\quad \Phi(r,x_3)=-\Phi(r,-x_3)
\end{equation}
As for the gradient of $\theta$, we have:
\begin{equation}\label{boundary_condition}
		\partial_r\theta(r,-x_3)=-\partial_r\theta(r,x_3),\quad \partial_{3}\theta(r,-x_3)=\partial_{3}\theta(r,x_3),
\end{equation}
and boundary condition
\begin{equation}\label{boundary_condition2}
	\begin{aligned}
	&\partial_r\theta(r,0)=\partial_r\theta(0,x_3)=0,\\
	&\partial_{33}\theta(r, 0)=\partial_{r3}\theta(0, x_3)=0,
		\end{aligned}
\end{equation}
since $\partial_{3}\theta$ is axisymmetric and odd in $x_3$. Note that such boundary condition offers the vanishing order of $\theta$, and allow us to integrate by parts .

\subsection{Coordinate transformation}
We inherit the notations from ~\cite{Masmoudi} and pass to the spherical coordinate $(R, \beta)$ we define in (\ref{beta}, \ref{R}), set
%?????? $(R, \beta)$ ????????~$\cite{Masmoudi}$??????????????
\begin{align}
	&\omega(r,x_3)=\Omega(R,\beta) \quad\quad \Phi(r,x_3)=\rho^2\Phi_\Omega(R,\beta),\\
	&\theta_r(r,x_3)=\tilde{\xi}(R,\beta)\quad\quad \theta_3(r,x_3)=\tilde{\phi}(R,\beta).
\end{align}

Then we have
\begin{equation}\label{prp3}
	\partial_r=\frac{\cos\beta}{\rho}\alpha D_R-\frac{\sin\beta}{\rho}\partial_\beta, \ \ \ \
	\partial_{3}=\frac{\sin\beta}{\rho}\alpha D_R+\frac{\cos\beta}{\rho}\partial_\beta,
\end{equation}
and
\begin{equation}\label{uru3}
	\begin{aligned}
		&u_r=\rho(2\sin\beta\Phi_\Omega+\alpha\sin\beta D_R\Phi_\Omega+\cos\beta\partial_\beta\Phi_\Omega),\\
		&u_3=\rho(-\frac{1}{\cos\beta}\Phi_\Omega-2\cos\beta\Phi_\Omega-\alpha\cos\beta D_R\Phi_\Omega+\sin\beta\partial_\beta\Phi_\Omega).
	\end{aligned}
\end{equation}
By such transformation, and recall that $R=\rho^\alpha$, the original equation can be written in coordinate $(R, \beta)$ as follows:
\begin{equation}\label{OmegaEvolution}
	\begin{cases}
	&\partial_t\Omega+U(\Phi_\Omega)\partial_\beta\Omega+V(\Phi_\Omega)\alpha D_R\Omega=\mathcal{R}(\Phi_\Omega)\Omega+\tilde{\xi},\\
	&\partial_t\tilde{\xi}+U(\Phi_\Omega)\partial_\beta\tilde{\xi}+V(\Phi_\Omega)\alpha D_R\tilde{\xi}+\Lambda_1(\Phi_\Omega)\tilde{\xi}+\Lambda_2(\Phi_\Omega)\tilde{\phi}=\Delta\tilde{\xi}-\frac{1}{R^{\frac{2}{\alpha}}\cos^2\beta}\tilde{\xi}
	,\\
	&\partial_t\tilde{\phi}+U(\Phi_\Omega)\partial_\beta\tilde{\phi}+V(\Phi_\Omega)\alpha D_R\tilde{\phi}+\Lambda_3(\Phi_\Omega)\tilde{\phi}+\Lambda_4(\Phi_\Omega)\tilde{\xi}=\Delta\tilde{\phi},
\end{cases}
\end{equation}
where
\begin{equation}\label{U_V}
	\begin{aligned}
		&U(\Phi_\Omega)=\frac{1}{\rho}u_\beta=-3\Phi_\Omega-\alpha D_R\Phi_\Omega,\quad
		V(\Phi_\Omega)=\frac{1}{\rho}u_\rho=\partial_\beta\Phi_\Omega-\tan\beta\Phi_\Omega, \\
		&\mathcal{R}(\Phi_\Omega)=\frac{1}{r}u_r=\frac{1}{\cos\beta}(2\sin\beta\Phi_\Omega+\alpha\sin\beta D_R\Phi_\Omega+\cos\beta \partial_\beta\Phi_\Omega),
	\end{aligned}
\end{equation}
and
\begin{equation}\label{Lambda}
	\begin{aligned}
		&\Lambda_1(\Phi_\Omega)=\partial_r u_r(r,x_3)=\sin\beta\cos\beta(\alpha D_R)^2\Phi_\Omega+(\cos^2\beta-\sin^2\beta)\alpha D_R\partial_\beta\Phi_\Omega\\
		&-\sin\beta\cos\beta\partial_{\beta\beta}\Phi_\Omega+\sin(2\beta)\alpha D_R\Phi_\Omega+\cos(2\beta)\partial_\beta\Phi_\Omega,\\
		&\\
		&\Lambda_2(\Phi_\Omega)=\partial_r u_{3}(r,x_3)=-\cos^2\beta(\alpha D_R)^2\Phi_\Omega+\sin(2\beta)\alpha D_R\partial_\beta\Phi_\Omega\\
		&-\sin^2\beta\partial_{\beta\beta}\Phi_\Omega-2(1+\cos^2\beta)\alpha D_R\Phi_\Omega+(\tan\beta+\sin(2\beta))\partial_\beta\Phi_\Omega+(\tan^2\beta-3)\Phi_\Omega,\\
		&\\
		&\Lambda_3(\Phi_\Omega)=\partial_3 u_r(r,x_3)=\sin^2\beta(\alpha D_R)^2\Phi_\Omega+\sin(2\beta)\alpha D_R\partial_\beta\Phi_\Omega\\
		&+\cos^2\beta\partial_{\beta\beta}\Phi_\Omega+(1+2\sin^2\beta)\alpha D_R\Phi_\Omega+\sin(2\beta)\partial_\beta\Phi_\Omega+2\Phi_\Omega,\\
		&\\
		&\Lambda_4(\Phi_\Omega)=\partial_3 u_{3}(r,x_3)=-\sin\beta\cos\beta(\alpha D_R)^2\Phi_\Omega-(\cos^2\beta-\sin^2\beta)\alpha D_R\partial_\beta\Phi_\Omega\\
		&+\sin\beta\cos\beta\partial_{\beta\beta}\Phi_\Omega-(\tan\beta+\sin(2\beta))\alpha D_R\Phi_\Omega-2\cos^2\beta\partial_\beta\Phi_\Omega-2\tan\beta\Phi_\Omega.
	\end{aligned}
\end{equation}
\subsection{The Setup}\label{setup}

First, we explain the main ideas, in the equation of (\ref{eq}) the temperature $\theta$ is obtained by a linear transport-diffusion equation coupled with velocity $u$, and heuristically speaking, the influence of $\theta$ on $u$ can be controlled by the smoothing effect of the transport-diffusion equation for $\theta$. In this way, we construct a solution such that $\theta$ is regarded as a small perturbation to the incompressible fluid.

More specifically, the solution we construct is nearly self-similar, which can be regarded as a perturbation to a self-similar profile constructed by Elgindi in~\cite{Elgindi1904}, which is self-similar and non-swirl, the vorticity has the form:

%?????????????????????????????????Elgindi?~$\cite{Elgindi1904}$????????????Euler ?????????????
\begin{equation*}
	\Omega=\frac{1}{T-t}F\Big(\frac{R}{(T-t)^{1+\delta}},\beta\Big),
\end{equation*}
where $\delta$ is a constant depending on $\alpha$, and $F=F_*+\alpha^2 g$ satisfying
\begin{equation}\label{F}
	F_*=\frac{\Gamma(\beta)}{c}\frac{4\alpha z}{(1+z)^2},\qquad |g|_{\mathcal{H}^k}\leq C,
\end{equation}
where $z=\frac{R}{(T-t)^{1+\delta}}, c=\int_{0}^{2\pi}3\sin\beta\cos^2\beta\Gamma(\beta)d\beta$.

And we write the equation of $F$
\begin{equation}\label{eq_F}
	F+(1+\delta)z\partial_{z}F+U(\Phi_{F})\partial_{\beta}F+V(\Phi_{F})\alpha z\partial_{z}F=R(\Phi_{F})F,
\end{equation}
then following the idea in \cite{Masmoudi}, we reformulate the equation by using the dynamic rescaling, set time variable $s$ and function~$\lambda(s),\mu(s)$ that satisfy:
\begin{align}\label{changeofvariable}
	y&=\frac{\mu(s) R}{\lambda^{1+\delta}(s)},\quad\quad\quad\frac{ds}{dt}=\frac{1}{\lambda(s)},\nonumber\\
	\Omega(R,\beta,t)&=\frac{1}{\lambda(s)}W(y,\beta,s ),~~ \Phi_\Omega(R,\beta,t)=\frac{1}{\lambda(s)}\Phi_W(y,\beta,s ).
\end{align}

For $\tilde\xi, \tilde\phi$ we give a different self-similar transformations: first we introduce parameters $l_1(s),l_2(s)$, and variable~$\bar{y}=\bar\rho^{\alpha}=l_2(s)y$, then we have a new self-similar coordinates~$(\bar{y},\beta)$, and we set:
\begin{equation*}
	\begin{aligned}
		&\tilde\xi(R,\beta,t)=\frac{l_1(s)}{\lambda^2(s)}\xi(l_2(s)y,\beta,s)=\frac{l_1(s)}{\lambda^2(s)}\xi(\bar y,\beta,s),\\
		&\tilde\phi(R,\beta,t)=\frac{l_1(s)}{\lambda^2(s)}\phi(l_2(s)y,\beta,s)=\frac{l_1(s)}{\lambda^2(s)}\phi(\bar y,\beta,s),
	\end{aligned}
\end{equation*}
\begin{rem}\label{remark}
	By the commutativity of the second derivative of $\theta$ :~$\partial_r \partial_{3}\theta=\partial_{3}\partial_{r}\theta$, we have compatibility conditions~: $\partial_{3}\xi=\partial_r\phi$ in~$(\bar y,\beta)$ coordinates:
	\begin{equation}\label{exchangeorder}
		\sin\beta \alpha D_{\bar y}\xi+\cos\beta\partial_\beta \xi=\cos\beta \alpha D_{\bar y}\phi-\sin\beta\partial_\beta\phi.
	\end{equation}
And by boundary condition (\ref{boundary_condition2}), we have boundary condition of $\phi$

\begin{equation}\label{boundary_condition3}
	\left.\partial_{\beta} \phi\right|_{\beta=0}=\left.\partial_{\beta} \phi\right|_{\beta=\pi / 2}=0.
\end{equation}
\end{rem}
From the above transformation, we can start to consider the solution we will construct. For $\xi,\phi$ that satisfy the commutativity condition in Remark $\ref{remark}$, we can solve the corresponding original function $\theta$, So we consider the solution~$(W,\xi,\phi)$, as a disturbance near $(F,0,0)$:

\begin{align*}
	&W(y,\beta,s)=F(y,\beta)+\varepsilon(y,\beta,s),~~\Phi_W(y,\beta,s)=\Phi_F(y,\beta)+\Phi_\varepsilon(y,\beta,s),\\
	&\xi(y,\beta,s)=0+\xi(y,\beta,s),~~\phi(y,\beta,s)=0+\phi(y,\beta,s).
\end{align*}

We write the equation of $W$:
\begin{equation}
W_{s}+\frac{\mu_{s}}{u} y \partial_{y} W-\frac{\lambda_{s}}{\lambda} (W+(1+\delta)y\partial_y W)+U\left(\Phi_{W}\right) \partial_{\theta} W+V\left(\Phi_{W}\right) \alpha y \partial_{y} W=\mathcal{R}\left(\Phi_{W}\right) W+l_1 (s)\xi(l_2(s)y,\beta,s),
\end{equation}
Combined with the equation of $F$ (\ref{eq_F}), we obtain the equation of $(\varepsilon,\xi,\phi)$:
\begin{align}\label{Equationepsilon}
	\begin{cases}
		&\partial_s \varepsilon+\frac{\mu_s}{\mu}y\partial_y \varepsilon-\Big(\frac{\lambda_s}{\lambda}+1\Big)(\varepsilon+(1+\delta)y\partial_y\varepsilon)+\mathcal{M}_F\varepsilon=E+N_2(\varepsilon)+l_1(s)\xi(l_2(s)y),\\
		&\partial_s\xi+\frac{\mu_s}{\mu}\bar{y}\partial_{\bar{y}}\xi-\Big(\frac{\lambda_s}{\lambda}+1\Big)(2\xi+(1+\delta)\bar{y}\partial_{\bar{y}}\xi)+\Big(2+\frac{l'_{1}}{l_1}\Big)\xi+\Big(1+\delta+\frac{l'_{2}}{l_2}\Big)\bar{y}\partial_{\bar{y}}\xi\\
		&+\mathcal{M}_\xi(\xi,\phi)=\Big(\frac{\mu l_2}{\lambda^{1+\delta}}\Big)^{\frac{2}{\alpha}}\lambda\bar{y}^{-\frac{2}{\alpha}}\big(\tilde{\Delta}\xi-\frac{1}{\cos^2\beta}\xi\big),\\
		&\partial_s\phi+\frac{\mu_s}{\mu}\bar{y}\partial_{\bar{y}}\phi-\Big(\frac{\lambda_s}{\lambda}+1\Big)(2\phi+(1+\delta)\bar{y}\partial_{\bar{y}}\phi)+\Big(2+\frac{l'_{1}}{l_1}\Big)\phi+\Big(1+\delta+\frac{l'_{2}}{l_2}\Big)\bar{y}\partial_{\bar{y}}\phi\\
		&+\mathcal{M}_\phi(\xi,\phi)=\Big(\frac{\mu l_2}{\lambda^{1+\delta}}\Big)^{\frac{2}{\alpha}}\lambda\bar{y}^{-\frac{2}{\alpha}}\big(\tilde{\Delta}\phi\big),\\
		&-\alpha^2y^2\partial_{yy}\Phi_\varepsilon-\alpha(5+\alpha)y\partial_y\Phi_\varepsilon-\partial_{\beta\beta}\Phi_\varepsilon+\partial_\beta\big(\tan\beta
		\Phi_\varepsilon\big)-6\Phi_\varepsilon=\varepsilon.
	\end{cases}
\end{align}

First in the equation of $\varepsilon$, $\mathcal{M}_{F}$ is the linearization operator, $N_2(\varepsilon)$ is the non-linear terms, $E$ is the error from the rescaling formation, inherited from \cite{Masmoudi}:
\begin{equation}
	\begin{aligned}
		&\mathcal{M}_{F} \varepsilon=\varepsilon+(1+\delta)y\partial_y\varepsilon+U\left(\Phi_{F}\right) \partial_{\beta} \varepsilon+V\left(\Phi_{F}\right) \alpha y \partial_{y} \varepsilon+U\left(\Phi_{\varepsilon}\right) \partial_{\beta} F+\\
		\quad\quad&+V\left(\Phi_{\varepsilon}\right) \alpha y \partial_{y} F-\mathcal{R}\left(\Phi_{F}\right) \varepsilon-\mathcal{R}\left(\Phi_{\varepsilon}\right) F\\
		&E=-\frac{\mu_{s}}{\mu} y \partial_{y} F+\left(\frac{\lambda_{s}}{\lambda}+1\right) (F+(1+\delta)y\partial_y F)\\
		&N_{2}(\varepsilon)=-U\left(\Phi_{\varepsilon}\right) \partial_{\beta} \varepsilon-\alpha V\left(\Phi_{\varepsilon}\right) y \partial_{y} \varepsilon+\mathcal{R}\left(\Phi_{\varepsilon}\right) \varepsilon
	\end{aligned}
\end{equation}

By the explicit form of $F$ (\ref{F}) and elliptic estimate of $\bar\Phi_F$, the remainder term of stream function $\Phi_F$ (\ref{thm_biot}), in \cite{Masmoudi} author rewrites the linearized operator $\mathcal{M}_{F} \varepsilon$ as:
\begin{equation}
	\mathcal{M}_{F}\varepsilon=\mathcal{L}_{F_{*}}^{T}\varepsilon+\frac{2 y^{2} \Gamma}{c(1+y)^{3}} L_{12}\left(\frac{3}{1+y} \sin (2 \beta) \partial_{\beta} \varepsilon\right)(0)+\sqrt{\alpha} \tilde{L},
\end{equation}
where
\begin{equation}
	\tilde{L} \varepsilon=-\frac{1}{\sqrt{\alpha}}\left\{\alpha V\left(F_{*}\right) y \partial_{y} \varepsilon+U\left(\Phi_{\varepsilon}\right) \partial_{\beta} F_{*}+\alpha V\left(\Phi_{\varepsilon}\right) y \partial_{y} F_{*}+\text{l.o.t.}\right\},
\end{equation}
where l.o.t.=lower order terms, and obtains the following estimate
	\begin{equation}\label{elgindiest2}
	\begin{aligned}
		\left|\left<\tilde{L}g, g\right>_{\mathcal{H}^{k}}\right| &\leq C|g|_{\mathcal{H}^{k}}^{2},\\
		\left<\mathcal{M}_F \varepsilon, \varepsilon\right>_{\mathcal{H}^{k}}&\geq c|\varepsilon|_{\mathcal{H}^{k}}^{2},
	\end{aligned}
\end{equation}
for all $g, \varepsilon\in\mathcal{H}^k$, satisfies $L_{12}(\varepsilon)(0)=0.$

In the equations of $\xi,\phi$, $\mathcal{M}_\xi, \mathcal{M}_\phi$ are the gradient of the transport term:
\begin{align*}
	&\mathcal{M}_\xi(\xi,\phi)=U(\Phi_W)\partial_\beta\xi+V(\Phi_W)\alpha\bar{y}\partial_{\bar{y}}\xi+\Lambda_1(\Phi_W)\xi+\Lambda_2(\Phi_W)\phi,\\
	&\mathcal{M}_\phi(\xi,\phi)=U(\Phi_W)\partial_\beta\phi+V(\Phi_W)\alpha\bar{y}\partial_{\bar{y}}\phi+\Lambda_3(\Phi_W)\xi+\Lambda_4(\Phi_W)\phi,
\end{align*}
as for Laplace term, for convenient we set
\begin{equation}\label{D_rho}
		\tilde{\Delta}f=\bar{y}^{2/\alpha}\Delta f=\alpha^2 D_{\bar{y}}^2f+\alpha D_{\bar{y}}f+\frac{1}{\cos\beta}\partial_\beta(\cos\beta\partial_\beta f)
\end{equation}

In order to make the solution we constructed take the formula $(\ref{F})$ as the main term, we have to give the estimate of the stability of $\varepsilon$ at least in $\mathcal{H}^4$. For any $k\geq4$, we have the energy estimate of the equation of $\varepsilon$:

%???????????????~$(\ref{F})$?????????$\varepsilon$???$\mathcal{H}^4$ ???????????$k\geq4$, ???$\varepsilon$?????????
\begin{equation}\label{varepsilonenergy}
	\begin{aligned}
&\frac{1}{2} \frac{d}{d s}\left<\varepsilon, \varepsilon\right>_{\mathcal{H}^{k}}\leq-\left<\mathcal{M}_{F} \varepsilon, \varepsilon\right>_{\mathcal{H}^{k}}+\left<E, \varepsilon\right>_{\mathcal{H}^{k}}+\left|\frac{\mu_{s}}{\mu}\right|\left|\left<y \partial_{y}\varepsilon,\varepsilon\right>_{\mathcal{H}^{k}}\right|+\left|\frac{\lambda_{s}}{\lambda}+1\right|\left|\left<\varepsilon+(1+\delta)y\partial_y\varepsilon, \varepsilon\right>_{\mathcal{H}^{k}}\right|\\
		&+\left<N_{2}(\varepsilon), \varepsilon\right>_{\mathcal{H}^{k}}+l_1(s)\left<\xi(l_2(s)y, \beta, s), \varepsilon(y, \beta, s)\right>_{\mathcal{H}^{k}}\\
		&\leq-c|\varepsilon(s)|_{\mathcal {H}^k}^2+\left|\frac{\lambda_{s}}{\lambda}+1\right||\varepsilon(s)|^{2}_{\mathcal{H}^k}+\frac{C}{\alpha^{3/2}}|\varepsilon(s)|_{\mathcal {H}^k}^3+l_1(s)\left<\xi(l_2(s)y, \beta, s),\varepsilon(y, \beta, s)\right>_{\mathcal{H}^k},
	\end{aligned}
\end{equation}

where every terms appear in~\cite{Masmoudi} except the last term, in fact in~\cite{Masmoudi}, the authors give the following estimate in the case that temperature $\theta$ is vanishing:

\begin{equation}\label{elgindiest}
	\frac{d}{ds}\left<\varepsilon,\varepsilon\right>_{\mathcal {H}^k}\leq -c|\varepsilon|_{\mathcal {H}^k}^2+\frac{C}{\alpha^{3/2}}|\varepsilon|_{\mathcal {H}^k}^3.
\end{equation}
In our case, difference is that we need the estimate of the last term about $\xi$ and different choice of dynamic scaling function $\lambda(s), \mu(s), l_1(s), l_2(s)$. Under dynamic scaling chosen in \ref{section3}, and using Cauchy-Schwartz inequality to the last term of (\ref{varepsilonenergy}), we also have 
\begin{equation}
	\frac{d}{ds}\left<\varepsilon,\varepsilon\right>_{\mathcal {H}^k}\leq -c|\varepsilon(s)|_{\mathcal {H}^k}^2+\frac{C}{\alpha^{3/2}}|\varepsilon(s)|_{\mathcal {H}^k}^3+C'l_1^2(s)|\xi(l_2(s)y,\beta,s)|_{\mathcal{H}^k}^2,
\end{equation}
where $C'>0$ depends only on $c$.

Since $\xi$ is corresponding to $\partial_r \theta$, we need estimate of the gradient of the temperature function. 
First, the equations that $\xi,\phi$ satisfy are parabolic equation, the self-similar scaling is different from the scaling of the velocity field. Therefore, we select a new self-similar rate for the temperature function, with suitable time depending function $l_1(s), l_2(s)$, which are determined in Section \ref{section3}. We obtain energy estimate of $\xi,\phi$ with norm we introduced in (\ref{normdef}),  for equation of $\xi$, we consider the energy estimate under the norms of $\mathcal{W}^k_1$ and $\mathcal{W}_2^k$, and energy estimate of $\phi$ under $\mathcal{W}_3^k$ norm:

\begin{equation}\label{energy}
	\begin{aligned}
		&\frac{1}{2}\frac{d}{ds}\langle\xi,\xi\rangle_{\mathcal{W}_1^k}+\frac{\mu_s}{\mu}\langle\frac{1}{\alpha}\bar\rho\partial_{\bar\rho}\xi,\xi\rangle_{\mathcal{W}_1^k}-
		\Big(\frac{\lambda_s}{\lambda}+1\Big)\big\langle2\xi+\frac{(1+\delta)}{\alpha}\bar\rho\partial_{\bar\rho}\xi,\xi\big\rangle_{\mathcal{W}_1^k}+\Big(2+\frac{l'_{1}}{l_1}\Big)\langle\xi,\xi\rangle_{\mathcal{W}_1^k}\\
		&+\Big(1+\delta+\frac{l'_{2}}{l_2}\Big)\Big\langle\frac{1}{\alpha}\bar\rho\partial_{\bar\rho}\xi,\xi\Big\rangle_{\mathcal{W}_1^k}+\langle\mathcal{M}_\xi(\xi,\phi),\xi\rangle_{\mathcal{W}_1^k}=\Big(\frac{\mu l_2}{\lambda^{1+\delta}}\Big)^{\frac{2}{\alpha}}\lambda\Big\langle\bar\rho^{-2} \big(\tilde{\Delta}\xi
		-\frac{1}{\cos^2\beta}\xi\big),\xi\Big\rangle_{\mathcal{W}_1^k},\\
		&\frac{1}{2}\frac{d}{ds}\langle\xi,\xi\rangle_{\mathcal{W}_2^k}+\frac{\mu_s}{\mu}\langle\frac{1}{\alpha}\bar\rho\partial_{\bar\rho}\xi,\xi\rangle_{\mathcal{W}_2^k}-
		\Big(\frac{\lambda_s}{\lambda}+1\Big)\big\langle2\xi+\frac{(1+\delta)}{\alpha}\bar\rho\partial_{\bar\rho}\xi,\xi\big\rangle_{\mathcal{W}_2^k}+\Big(2+\frac{l'_{1}}{l_1}\Big)\langle\xi,\xi\rangle_{\mathcal{W}_2^k}\\
		&+\Big(1+\delta+\frac{l'_{2}}{l_2}\Big)\Big\langle\frac{1}{\alpha}\bar\rho\partial_{\bar\rho}\xi,\xi\Big\rangle_{\mathcal{W}_2^k}+\langle\mathcal{M}_\xi(\xi,\phi),\xi\rangle_{\mathcal{W}_2^k}=\Big(\frac{\mu l_2}{\lambda^{1+\delta}}\Big)^{\frac{2}{\alpha}}\lambda\Big\langle\bar\rho^{-2}\big(\tilde{\Delta}\xi
		-\frac{1}{\cos^2\beta}\big)\xi,\xi\Big\rangle_{\mathcal{W}_2^k},\\
		&\frac{1}{2}\frac{d}{ds}\langle\phi,\phi\rangle_{\mathcal{W}_3^k}+\frac{\mu_s}{\mu}\langle\frac{1}{\alpha}\bar\rho\partial_{\bar\rho}\phi,\phi\rangle_{\mathcal{W}_3^k} -\Big(\frac{\lambda_s}{\lambda}+1\Big)\Big\langle(2\phi+\frac{(1+\delta)}{\alpha}\bar\rho\partial_{\bar\rho}\phi),\phi\Big\rangle_{\mathcal{W}_3^k}+\Big(2+\frac{l'_{1}}{l_1}\Big)\langle\phi,\phi\rangle_{\mathcal{W}_3^k}\\
		&+\Big(1+\delta+\frac{l'_{2}}{l_2}\Big)\langle\frac{1}{\alpha}\bar\rho\partial_{\bar\rho}\phi,\phi\rangle_{\mathcal{W}_3^k}+\langle\mathcal{M}_\phi(\xi,\phi),\phi\rangle_{\mathcal{W}_3^k}=\Big(\frac{\mu l_2}{\lambda^{1+\delta}}\Big)^{\frac{2}{\alpha}}\lambda\big\langle\bar{\rho}^{-2}\tilde{\Delta}\phi,\phi\big\rangle_{\mathcal{W}_3^k}.
	\end{aligned}
\end{equation}
The reason of introducing these energy estimates is that the sum of three Laplace terms above is coercive under corresponding norm and bounded by
\begin{equation}\label{normestimate}
	-Cl_2^{2/\alpha}(0)\left(|\frac{1}{\bar\rho}\partial_\beta\xi|_{\mathcal{W}^{k}_1}^{2}+|\frac{1}{\bar\rho}D_{\bar\rho}\xi|_{\mathcal{W}^{k}_1}^{2}+
	|\frac{1}{\bar\rho}\partial_\beta\xi|_{\mathcal{W}^{k}_2}^{2}+|\frac{1}{\bar\rho}D_{\bar\rho}\xi|_{\mathcal{W}^{k}_2}^{2}+|\frac{1}{\bar\rho}\partial_\beta\phi|_{\mathcal{W}^{k}_3}^{2}+|\frac{1}{\bar\rho}D_{\bar\rho}\phi|_{\mathcal{W}^{k}_3}^{2}\right),
\end{equation}
which we will proof in Section \ref{est_laplace}. 

We note that this term can eliminate the term $l_1^2(s)|\xi(l_2(s)y, \beta,s)|_{\mathcal{H}^k}^2$ which is appeared in the equation of $\varepsilon$ (\ref{varepsilonenergy}). Recall the definition of $\mathcal{H}^k$ norm (\ref{HsNorm}) , we have two kinds of term $l_1^2(s)|\xi(l_2(s)y, \beta,s)|_{\mathcal{H}^k}^2$  to estimate
\begin{equation}
	l_1^2(s)\int_{0}^{\infty}\int_{0}^{\pi/2}D_y^i\xi(l_2(s)y,\beta,s)^2\sin^{-\eta}(2\beta)\frac{(1+y)^4}{y^4}dyd\beta,\quad 0\leq i\leq k,
\end{equation}
and
\begin{equation}
	l_1^2(s)\int_{0}^{\infty}\int_{0}^{\pi/2}D_y^i D_\beta^j\xi(l_2(s)y,\beta,s)^2\sin^{-\gamma}(2\beta)\frac{(1+y)^4}{y^4}dyd\beta, \quad 1\leq j, 1\leq i+j\leq k.
\end{equation}
Recall that $\bar{y}=\bar\rho^\alpha=l_2(s)y,$ the radial component of the weight of $\mathcal{H}^k$ norm, i.e.$\frac{(1+y)^4}{y^4} dy$, has the form in $\bar\rho$ variable:
\begin{equation}{\label{weight}}
	\frac{(1+y)^4}{y^4} dy=\frac{\alpha}{l_2(s)}\frac{(l_2(s)+\bar\rho^{\alpha})^4}{\bar\rho^{1+3\alpha}}d\bar\rho.
\end{equation}
We can rewrite $l_1^2(s)|\xi(l_2(s)y, \beta,s)|_{\mathcal{H}^k}^2$ as
\begin{equation}
	\begin{aligned}
	\frac{l_1^2(s)}{l_2(s)}\Big\{&
		\sum_{0\leq i\leq k}\int_{0}^{\infty}\int_{0}^{\pi/2}\alpha^{-2i}(D_{\bar\rho}^i\xi)^2\sin^{-\eta}(2\beta)\alpha\frac{(l_2(s)+\bar\rho^\alpha)^4}{\bar\rho^{1+3\alpha}}d\bar\rho d\beta +\\
		& \sum_{1\leq i+j\leq k, 1\leq j}\int_{0}^{\infty}\int_{0}^{\pi/2}\alpha^{-2i}(D_{\bar\rho}^i D_\beta^j\xi)^2\sin^{-\gamma}(2\beta)\alpha\frac{(l_2(s)+\bar\rho^\alpha)^4}{\bar\rho^{1+3\alpha}}d\bar\rho d\beta \Big\}.
			\end{aligned}
\end{equation}
By the choice of $l_1(s), l_2(s)$ in Section \ref{section3}, $l_2(s)$ and $\frac{l_1^2(s)}{l_2(s)}$ have upper bound independent of $\alpha$, can be viewed as constant here. Then for every integral above, the order at infinity is $\bar\rho^{-1+\alpha}$, and order at 0 is $\bar\rho^{-1-3\alpha}$, and the radial component of the weight of (\ref{normestimate}), consider  $|\partial_\beta\xi|_{\mathcal{W}_1^k}$ and $|\partial_\beta\xi|_{\mathcal{W}_2^k}$ in (\ref{normestimate}) for example and write the radial integral for simplicity
\begin{equation}
	\int_{0}^{\infty}\frac{1}{\bar\rho^2}|\partial_\beta\xi|^2\bar\rho^2 d\bar\rho \quad for~ \mathcal{W}_1^k,\quad \int_{0}^{\infty}\frac{1}{\bar\rho^2}|\partial_\beta\xi|^2\bar\rho^\eta d\bar\rho\quad for~\mathcal{W}_2^k.
\end{equation}
Note that $\eta<1$, the order of (\ref{weight}) at 0 and infinity, can be controlled by $\mathcal{W}_1^k,~\mathcal{W}_2^k$ norm respectively.

%$(\ref{normestimate})$ ?$\bar\rho$ ????$\mathcal{W}^k_1$????????$\mathcal{H}^k$???????$\bar\rho^{\alpha-1}$?$\mathcal{W}^k_2$??$\bar\rho$ ???$\bar\rho^{\eta-2}$,??????0???$\bar\rho^{-1-3\alpha}$??????????k?????~$D_y^k=\alpha^{-k} D_{\bar\rho}^k $,????????~$\alpha^{-2k+1}$ ???
As for the integral of $\beta$, recall that $D=D_{\bar\rho}$ or~$D_\beta$, then consider the singular term in $\mathcal {H}^k$ norm, i.e.~$D_\beta D^{k-1}\xi$ term, by weighted Hardy inequality we proof in Lemma \ref{hardy_eta}, we have
\begin{equation*}
	\begin{aligned}
		\int_{0}^{\infty}\int_{0}^{\pi/2}(D_\beta D^{k-1}\xi)^2\sin^{-\gamma}(2\beta) d\beta d\bar\rho\leq&\int_{0}^{\infty}\int_{0}^{\pi/2}(D_\beta D^{k-1}\xi)^2\sin^{-2-\eta}(2\beta) d\beta d\bar\rho\\
		\leq&\frac{1}{(1+\eta)^2}\int_{0}^{\infty}\int_{0}^{\pi/2}(\partial_\beta D_\beta D^{k-1}\xi)^2\sin^{-\eta}(2\beta) d\beta d\bar\rho,
	\end{aligned}
\end{equation*}
which is controlled by~$|\frac{1}{\bar\rho}\partial_\beta\xi|_{\mathcal{W}^{k}_1}$, so there exists constant $C>0$ independent of $\alpha$ such that
\begin{equation}\label{embingding_laplace}
	|\xi|_{\mathcal{H}^k}^2\leq C\alpha^{-2k+1}\left\{|\frac{1}{\bar\rho}\partial_\beta\xi|_{\mathcal{W}^{k}_1}^{2}+|\frac{1}{\bar\rho}D_{\bar\rho}\xi|_{\mathcal{W}^{k}_1}^{2}+
	|\frac{1}{\bar\rho}\partial_\beta\xi|_{\mathcal{W}^{k}_2}^{2}+|\frac{1}{\bar\rho}D_{\bar\rho}\xi|_{\mathcal{W}^{k}_2}^{2}\right\}.
\end{equation}
\\
For $k\geq4$ we set
\begin{equation}\label{XY}
	\begin{aligned}
		X(s)=&|\xi(s)|_{\mathcal{W}^k_1}^2+|\xi(s)|_{\mathcal{W}^k_2}^2+|\phi(s)|_{\mathcal{W}^k_3}^2,\\
		Y(s)=&|\frac{1}{\bar\rho}\partial_\beta\xi(s)|_{\mathcal{W}^{k}_1}^{2}+|\frac{1}{\bar\rho}D_{\bar\rho}\xi(s)|_{\mathcal{W}^{k}_1}^{2}+
		|\frac{1}{\bar\rho}\partial_\beta\xi(s)|_{\mathcal{W}^{k}_2}^{2}+|\frac{1}{\bar\rho}D_{\bar\rho}\xi(s)|_{\mathcal{W}^{k}_2}^{2}+
		|\frac{1}{\bar\rho}\partial_\beta\phi(s)|_{\mathcal{W}^{k}_3}^{2}+|\frac{1}{\bar\rho}D_{\bar\rho}\phi(s)|_{\mathcal{W}^{k}_3}^{2}\\
		\mathcal{E}(s)=&C\alpha^{-2k+1}X(s)+|\varepsilon(s)|_{\mathcal{H}^k}^2
	\end{aligned}
\end{equation}

Then~$(\ref{embingding_laplace})$ can be written as:
\begin{equation}\label{xiest}
	|\xi(s)|_{\mathcal{H}^k}^2\leq C\alpha^{-2k+1} Y(s),
\end{equation}
where $C$ is constant independent of $\alpha$.
And by (\ref{normestimate}) can be refitted as

$$\Big(\frac{\mu l_2}{\lambda^{1+\delta}}\Big)^{\frac{2}{\alpha}}\lambda Y(s) \approx  l_2^{2/\alpha}(0)Y(s)$$

$\mathcal{E}(s)$ give the control of $\varepsilon$ and $\xi,\phi$, then our goal is to prove:
\begin{thm}\label{thm2}
	For $\alpha<10^{-14}$, there exist $\delta_0$ and $\kappa >0$, such that for any initial value $(\varepsilon_0,\theta_0)$ satisfying $\mathcal{E}(0)\leq\delta_0\alpha^{3}$ and $L_{12}(\varepsilon_0)=0$.
	Then for any $s\geq 0$, the solution of $(\ref{Equationepsilon})$ satisfies:
	\begin{equation}\label{finalest}
		\mathcal{E}(s)\leq C\mathcal{E}(0)e^{-\kappa s},
	\end{equation}
and
	\begin{equation}\label{finalest2}
		\int_0^\infty |\frac{\mu_s}{\mu}|+|\frac{\lambda_s}{\lambda}+1|ds\leq C\mathcal{E}(0).
	\end{equation}
\end{thm}
\begin{rem}\label{rem_thm}
	By~$(\ref{finalest2})$, we have:
	$$\ln(\lambda(s)e^s)=\ln\lambda(0)+\int_0^s(\frac{\lambda_s}{\lambda}(s')+1)ds',$$
	~$\lambda(s)e^s$ is bounded for any $s\geq 0,$ then we can resolve the time $t$, by~$(\ref{changeofvariable})$ we have:
	$$t(s)=\int_0^s\lambda(s')ds'=\int_0^s(\lambda(s')e^{s'})e^{-s'}ds'\leq C\int_0^s e^{-s'}ds',$$
	then $t|_{s=\infty}$ is finite, $(\ref{finalest2})$ ensures that blowup occurs at finite time.
\end{rem}
The choice of initial value of $\varepsilon_0$ is inherited from \cite{Masmoudi},  $W_0=F+\varepsilon_0$ is compact supported function, and $|\varepsilon_0|_{\mathcal{H}^k}\leq C \alpha^2$. As for~$\theta_0$,
%?????~$\theta$???????????
we set $\theta$ as disturbance term, which means we can choose $\theta_0$ small enough such that $X(0)\leq C\alpha^{2k+3}$.

\section{The Dynamic Rescaling Formulation}\label{section3}

Inherited from \cite{Masmoudi}, to ensure the coercivity of the linearization operator $\mathcal{M}_F$ (\ref{elgindiest2})
$$\left<\mathcal{M}_F(\varepsilon),\varepsilon\right>_{\mathcal{H}^k}\geq c|\varepsilon|^2_{\mathcal{H}^k},$$
it's suffices to impose following conditions
\begin{equation}
	\partial_y\varepsilon(0,\beta,s)=0,\quad L_{12}(\varepsilon)(0,\beta,s)=0,
\end{equation}
for any $\beta\in[0,\pi/2],s>0$. following the framework in \cite{Masmoudi}, we  set $\lambda$and$\mu$ satisfying following ODEs:
\begin{equation}\label{ode}
	\begin{aligned}
		&-\alpha\left(\frac{\lambda_{s}}{\lambda}+1\right)+3 L_{12}\left(\frac{\sin (2 \beta)}{1+y} \partial_{\beta} \varepsilon\right)(0)=\sqrt{\alpha} L_{12}(\tilde{L} \varepsilon)(0)+L_{12}\left(N_{2}(\varepsilon)\right)(0)+l_1(s)L_{12}(\xi)(0),\\
		&\frac{\mu_{s}}{\mu}=(2+\delta)\left(\frac{\lambda_{s}}{\lambda}+1\right).
	\end{aligned}
\end{equation}
The difference with the Euler case in \cite{Masmoudi} is the vanishing of swirl and appearance of temperature term $\xi$. And we choose $l_1(s), l_2(s)$ as following
\begin{equation}\label{l_1l_2}
	\begin{aligned}
		l_1(s)&=e^{-s},\\
		l_2(s)&=l_2(0)\exp\left(-(1+\delta)s+\frac{1}{2}\alpha s-(1+\frac{\alpha}{2})\int_0^s(\frac{\lambda_s}{\lambda}(\sigma)+1)d\sigma\right).
	\end{aligned}
\end{equation}
Then we have following bounds of $\lambda(s), \mu(s), l_1(s), l_2(s)$.
\begin{prop}\label{prop_time}
	If $\lambda,\mu$ satisfy ODE (\ref{ode}), and initial value  $\varepsilon_0$ satisfies $$\partial_y\varepsilon_0(0,\beta)=0,\quad L_{12}(\varepsilon_0)(0)=0,$$ then we have \begin{equation}
\partial_y\varepsilon(0,\beta,s)=0,L_{12}(\varepsilon(s))(0)=0,	\end{equation}
for $\beta\in[0,\pi/2],s>0.$ And we also have following estimates:
	\begin{equation}\label{lambda_mu}
		\begin{aligned}
			&\big|\alpha\left(\frac{\lambda_{s}}{\lambda}+1\right)-3 L_{12}\left(\frac{\sin (2 \beta)}{1+y} \partial_{\theta} \varepsilon\right)(0) \big| \leq C\left(\sqrt{\alpha}|\varepsilon(s)|_{\mathcal{H}^{2}}+C\alpha Y(s)^{1/2}\right),\\
			&\left|\frac{\lambda_{s}}{\lambda}+1\right| \leq \frac{C}{\alpha}|\varepsilon(s)|_{\mathcal{H}^{2}}+CY^{1/2}(s).
		\end{aligned}
	\end{equation}
\end{prop}
\begin{proof}

	Compared with the case of Euler equation, the difference is the appearance of the temperature term $\xi$.
	Note that the $\xi$ term is corresponding to $\partial_r\theta$ in $(\rho,\beta)$ coordinate, which satisfies boundary condition $(\ref{boundary_condition})$:$\partial_\rho(\partial_r\theta)(0,\beta,s)=0$, then in self-similar coordinate we have $\partial_y\xi(0,\beta,s)=0$, for any $\beta\in[0,\pi/2]$, the proof of $$\partial_y\varepsilon(0,\beta,s)=0$$ is just a repetition of that in \cite{Masmoudi}.
	
As for $L_{12}(\varepsilon(s))(0)$, by the definition of $L_{12}$ $(\ref{L_12})$, we have:

	%In equation $\varepsilon$ ???~$(\ref{Equationepsilon})$??~$L_{12}$???????~$\xi$??????~$\cite{Masmoudi}$???~5.1??????
	%?~$L_{12}$???~$(\ref{L_12})$????
	\begin{align}\label{estL_12}
		L_{12}(l_1(s)\xi)(0)=&l_1(s)\int_0^\infty\int_0^{\pi/2} \xi(z,\beta)3\sin\beta\cos^2\beta\frac{dz d\beta}{z}\\
		=&l_1(s)\int_0^\infty\int_0^{\pi/2} \left(\xi(z,\beta)\frac{(1+z)^2}{z^2}\right)\left(3\sin\beta\cos^2\beta\cdot \frac{z}{(1+z)^2}\right) dz d\beta\\
		\leq& C|\xi|_{\mathcal{H}^0}.
	\end{align}
	By~$(\ref{embingding_laplace})$ we have~$|\xi|_{\mathcal{H}^0}\leq C\alpha Y(s)$. To keep $L_{12}(\varepsilon)(0)=0,$ we impose $L_{12}$ on the equation of $\varepsilon$ (\ref{Equationepsilon}), follow the treatment of that in \cite{Masmoudi} and the estimate of $L_{12}(l_1(s)\xi)(0)$ (\ref{estL_12}),  we deduce $(\ref{lambda_mu})$.
\end{proof}
Next for $l_1(s), l_2(s)$ chosen as in (\ref{l_1l_2}), we have

\begin{prop}\label{prop_s}
	we set $l_1(s),l_2(s)$ satisfy (\ref{l_1l_2}), then there exist $c,C>0$ that don't depend on $\alpha$ or $\lambda(s)$ such that:
	\begin{equation}\label{l1l2}
		\begin{aligned}
			&\left|l_1(s)\xi(l_2(s)y)\right|_{\mathcal{H}^k}^2\leq\frac{l_1^2(s)}{l_2(s)}\left|\xi\right|_{\mathcal{H}^k}^2,\\
			&cl^{2/\alpha}_2(0)\leq\Big(\frac{\mu l_2}{\lambda^{1+\delta}}\Big)^{\frac{2}{\alpha}}\lambda\leq Cl^{2/\alpha}_{2}(0),
		\end{aligned}
	\end{equation}
for all $s>0.$
\end{prop}
\begin{proof}
	 For~$l_1(s)$, we set $l_1(s)=e^{-s}$, since $\xi,\phi$ are linear with respect to the equation~$(\ref{Equationepsilon})$, and the only term we need consider in energy estimate $(\ref{varepsilonenergy})$ is $l_1(s)\left<\xi(l_2(s)y), \varepsilon\right>_{\mathcal{H}^{k}}$.
	
	Note that the time scaling of $\xi$ and $\varepsilon$ is different only in the radio variable $\rho$, and differential operator $D_\rho=D_{\bar\rho}$ is invariant under time scaling since $\bar\rho=l^{1/\alpha}_2(s)\rho$.
	
	%?????~$D_\rho=D_{\bar\rho}$ ???~scaling ????
	We consider the case that $k=0$ for simplicity, and concern on the integral of radio variable $\bar\rho$:
	
	%??????~$k=0$??????????~$\bar\rho$???~$\varepsilon$?~$L_2$???
	\begin{equation*}
		\begin{aligned}
			&l_1(s)\int_0^\infty\xi(l_2(s)y)\varepsilon(y)\frac{(1+y)^4}{y^4}dy\\
			&\leq C(\delta_0)l_1(s)^2\int_0^\infty\xi(l_2(s)y)^2\frac{(1+y)^4}{y^4}dy+\delta_0|\varepsilon|^2_{\mathcal{H}^0}.
		\end{aligned}
	\end{equation*}
	For small enough~$\delta_0$, the second term above can be controlled by the negative term in $(\ref{elgindiest})$. For the first term in the above formula, by the decreasing of $l_2(s)$, we have:
	%???????????~$(\ref{elgindiest})$??????????????????$l_2(s)$??,????
	\begin{equation*}
		\begin{aligned}
			&l_1(s)^2\int_0^\infty\xi(l_2(s)y)^2\frac{(1+y)^4}{y^4}dy\\
			\leq&l_1(s)^2\int_0^\infty\xi(\bar{y})^2\frac{(l_2(s)+\bar{y})^4}{\bar{y}^4}\frac{d\bar y}{l_2(s)} \leq\frac{l_1^2(s)}{l_2(s)}\int_0^\infty\xi(\bar{y})^2\frac{(1+\bar{y})^4}{\bar{y}^4}d\bar y,
		\end{aligned}
	\end{equation*}
	where $\bar{y}=l_2(s)y$.
	
	The coefficient of time in the above formula $\frac{l_1^2(s)}{l_2(s)}\leq \frac{e^{(-1+\delta)s}}{l_2(0)}$, is bounded, so $l_1,l_2$ don't tend to infinity, in fact by using $(\ref{xiest})$, we have:
	%??????????~$\frac{l_1^2(s)}{l_2(s)}\leq \frac{e^{(-1+\delta)s}}{l_2(0)}$ ???????????~$\xi,\phi$ ??????????~$l_1,l_2$ ?????????????????~$(\ref{xiest})$???
	\begin{equation}\label{xinorm}
		l_1(s)|\xi(l_2(s)y|^2_{\mathcal{H}^{k}}\leq\frac{l_1(s)^2}{l_2(s)}|\xi|^2_{\mathcal{H}^{k}}\leq C\alpha^{-2k+1}\frac{l_1(s)^2}{l_2(s)}Y(s)\leq C'\alpha^{-2k+1}Y(s).
	\end{equation}
	
	In the equations of $\xi$ and $\phi$, we consider the time scaling coefficient Laplace term, by $(\ref{lambda_mu})$ we have:
	\begin{equation*}
		\Big(\frac{\mu l_2}{\lambda^{1+\delta}}\Big)^{\frac{2}{\alpha}}\lambda=\Big(\frac{(\lambda e^s)^{(2+\delta)} l_2}{\lambda^{1+\delta}}\Big)^{\frac{2}{\alpha}}\lambda
		=e^{-s}\left(e^s\lambda\right)^{\frac{2}{\alpha}+1}\left(e^{(1+\delta)s}l_2 \right)^{\frac{2}{\alpha}}.
	\end{equation*}
	
	By $(\ref{ode})$, we have~$\mu(s)=(e^s\lambda(s))^{2+\delta}$, and
	\begin{equation*}
		\Big(\frac{\mu l_2}{\lambda^{1+\delta}}\Big)^{\frac{2}{\alpha}}\lambda=\lambda^{\frac{2}{\alpha}+1}(0)e^{-s}\exp\left((1+\frac{2}{\alpha})
		\int_0^s(\frac{\lambda_s}{\lambda}(\sigma)+1)d\sigma\right)\left(e^{(1+\delta)s}l_2\right)^{\frac{2}{\alpha}}.
	\end{equation*}
	Since $l_2(s)$ have the form:
	\begin{equation}\label{l_2}
		l_2(s)=l_2(0)\exp\left(-(1+\delta)s+\frac{1}{2}\alpha s-(1+\frac{\alpha}{2})\int_0^s(\frac{\lambda_s}{\lambda}(\sigma)+1)d\sigma\right),
	\end{equation}
	we deduce $(\ref{l1l2})$.
\end{proof}
Next, we consider time derivative terms that related to time variable $\mu(s),\lambda(s),l_1(s),l_2(s)$ in energy estimate $(\ref{energy})$

\begin{equation}\label{estimate_self1}
	\begin{aligned}
		&\frac{\mu_s}{\mu}\langle \bar{y}\partial_{\bar{y}}\xi,\xi\rangle_{\mathcal{W}_1^k}-
		\Big(\frac{\lambda_s}{\lambda}+1\Big)\big\langle2\xi+(1+\delta)\bar{y}\partial_{\bar{y}}\xi,\xi\big\rangle_{\mathcal{W}_1^k}+\Big(2+\frac{l'_{1}}{l_1}\Big)\left<\xi,\xi\right>_{\mathcal{W}_1^k}+\Big(1+\delta+\frac{l'_{2}}{l_2}\Big)		\left<\bar{y}\partial_{\bar{y}}\xi,\xi\right>_{\mathcal{W}_1^k},\\
		&\frac{\mu_s}{\mu}\langle \bar{y}\partial_{\bar{y}}\xi,\xi\rangle_{\mathcal{W}_2^k}-
		\Big(\frac{\lambda_s}{\lambda}+1\Big)\big\langle2\xi+(1+\delta)\bar{y}\partial_{\bar{y}}\xi,\xi\big\rangle_{\mathcal{W}_2^k}+\Big(2+\frac{l'_{1}}{l_1}\Big)\left<\xi,\xi\right>_{\mathcal{W}_2^k}+\Big(1+\delta+\frac{l'_{2}}{l_2}\Big)		\left<\bar{y}\partial_{\bar{y}}\xi,\xi\right>_{\mathcal{W}_2^k},\\
		&\frac{\mu_s}{\mu}\langle\bar{y}\partial_{\bar{y}}\phi,\phi\rangle_{\mathcal{W}_3^k} -\Big(\frac{\lambda_s}{\lambda}+1\Big)\big\langle(2\phi+(1+\delta)\bar{y}\partial_{\bar{y}}\phi),\phi\big\rangle_{\mathcal{W}_3^k}+\Big(2+\frac{l'_{1}}{l_1}\Big)\left<\phi,\phi\right>_{\mathcal{W}_3^k}+\Big(1+\delta+\frac{l'_{2}}{l_2}\Big)\left<\bar{y}\partial_{\bar{y}}\phi,\phi\right>_{\mathcal{W}_3^k}.
	\end{aligned}
\end{equation}
\begin{cor}\label{l1l2cor}
	Setting~$l_1(s),l_2(s)$ as the formation in $(\ref{l_1l_2})$, we have following estimates of $(\ref{estimate_1})$:
	%???????$(\ref{estimate_1})$??????
	\begin{equation}\label{estimate_2}
		\begin{aligned}
			&\frac{\mu_s}{\mu}\langle\frac{1}{\alpha}\bar\rho\partial_{\bar\rho}\xi,\xi\rangle_{\mathcal{W}_1^k}-
			\Big(\frac{\lambda_s}{\lambda}+1\Big)\big\langle2\xi+\frac{(1+\delta)}{\alpha}\bar\rho\partial_{\bar\rho}\xi,\xi\big\rangle_{\mathcal{W}_1^k}+\left(2+\frac{l'_{1}}{l_1}\right)\left<\xi,\xi\right>_{\mathcal{W}_1^k}+\Big(1+\delta+\frac{l'_{2}}{l_2}\Big)\left<\bar{y}\partial_{\bar{y}}\xi,\xi\right>_{\mathcal{W}_1^k}\\
			\geq&\left(\frac{1}{4}-C\alpha^{1/2}\right)\left<\xi,\xi\right>_{\mathcal{W}_1^k}-CY(s)\left<\xi,\xi\right>_{\mathcal{W}_1^k},\\
			&\frac{\mu_s}{\mu}\langle\frac{1}{\alpha}\bar\rho\partial_{\bar\rho}\xi,\xi\rangle_{\mathcal{W}_2^k}-
			\Big(\frac{\lambda_s}{\lambda}+1\Big)\big\langle2\xi+\frac{(1+\delta)}{\alpha}\bar\rho\partial_{\bar\rho}\xi,\xi\big\rangle_{\mathcal{W}_2^k}+\left(2+\frac{l'_{1}}{l_1}\right)\left<\xi,\xi\right>_{\mathcal{W}_2^k}+\Big(1+\delta+\frac{l'_{2}}{l_2}\Big)\left<\bar{y}\partial_{\bar{y}}\xi,\xi\right>_{\mathcal{W}_2^k}\\
			\geq&\left(\frac{1}{4}-C\alpha^{1/2}\right)\left<\xi,\xi\right>_{\mathcal{W}_2^k}-CY(s)\left<\xi,\xi\right>_{\mathcal{W}_2^k},\\
			&\frac{\mu_s}{\mu}\langle\frac{1}{\alpha}\bar\rho\partial_{\bar\rho}\phi,\phi\rangle_{\mathcal{W}_3^k} -\Big(\frac{\lambda_s}{\lambda}+1\Big)\Big\langle(2\phi+\frac{(1+\delta)}{\alpha}\bar\rho\partial_{\bar\rho}\phi),\phi\Big\rangle_{\mathcal{W}_3^k}+\left(2+\frac{l'_{1}}{l_1}\right)\left<\phi,\phi\right>_{\mathcal{W}_3^k}+\Big(1+\delta+\frac{l'_{2}}{l_2}\Big)\left<\bar{y}\partial_{\bar{y}}\phi,\phi\right>_{\mathcal{W}_3^k}\\
			\geq&\left(\frac{1}{4}-C\alpha^{1/2}\right)\left<\phi,\phi\right>_{\mathcal{W}_3^k}-CY(s)\left<\phi,\phi\right>_{\mathcal{W}_3^k}.
		\end{aligned}
	\end{equation}
\end{cor}
Recall that $\bar{y}=\bar\rho^\alpha, \bar{y}\partial_{\bar{y}}=\frac{1}{\alpha}\bar\rho\partial_{\bar{\rho}}$, using integration by parts, we have
$$\left<\bar{y}\partial_{\bar{y}}\xi,\xi\right>_{\mathcal{W}_i^k}=-\frac{C_i}{2\alpha}\left<\xi,\xi\right>_{\mathcal{W}_i^k},$$ with $C_1=C_3=3, C_2=1+\eta .$

We consider $\mathcal{W}_1^k$ norm for example, using $\frac{\mu_s}{\mu}=(2+\delta)(\frac{\lambda_s}{\lambda}+1),$ we have 
\begin{equation}
	\begin{aligned}
		&\frac{\mu_s}{\mu}\langle\frac{1}{\alpha}\bar\rho\partial_{\bar\rho}\xi,\xi\rangle_{\mathcal{W}_1^k}-
		\Big(\frac{\lambda_s}{\lambda}+1\Big)\big\langle2\xi+\frac{(1+\delta)}{\alpha}\bar\rho\partial_{\bar\rho}\xi,\xi\big\rangle_{\mathcal{W}_1^k}+\left(2+\frac{l'_{1}}{l_1}\right)\left<\xi,\xi\right>_{\mathcal{W}_1^k}+\Big(1+\delta+\frac{l'_{2}}{l_2}\Big)\left<\bar{y}\partial_{\bar{y}}\xi,\xi\right>_{\mathcal{W}_1^k}\\
		=&\left[2+\frac{l'_{1}}{l_1}-2\left(\frac{\lambda_s}{\lambda}+1\right)\right]\left<\xi,\xi\right>_{\mathcal{W}_1^k}+\frac{1}{\alpha}\left[\left(\frac{\lambda_s}{\lambda}+1\right)+\left(1+\delta+\frac{l'_2}{l_2}\right)\right]\left<\bar\rho\partial_{\bar\rho}\xi,\xi\right>_{\mathcal{W}_1^k}.
	\end{aligned}
\end{equation}
 Then using integrating by parts to $(\ref{estimate_self1})$, we have

\begin{equation}\label{estimate_1}
	\begin{aligned}
		&\left[2+\frac{l'_{1}}{l_1}-2\left(\frac{\lambda_s}{\lambda}+1\right)\right]\left<\xi,\xi\right>_{\mathcal{W}_1^k}+\frac{1}{\alpha}\left[\left(\frac{\lambda_s}{\lambda_s}+1\right)+\left(1+\delta+\frac{l'_2}{l_2}\right)\right]\left<\bar\rho\partial_{\bar\rho}\xi,\xi\right>_{\mathcal{W}_1^k}\\
		=&\left[1-2\left(\frac{\lambda_s}{\lambda}+1\right)-\frac{3}{2\alpha}\left(\frac{\alpha}{2}-\frac{\alpha}{2}\left(\frac{\lambda_s}{\lambda}+1\right)\right)\right]\left<\xi,\xi\right>_{\mathcal{W}_1^k}\\
		=&\left[\frac{1}{4}-\frac{5}{4}\left(\frac{\lambda_s}{\lambda}+1\right)\right]\left<\xi,\xi\right>_{\mathcal{W}_1^k}\\
		\geq&\left(\frac{1}{4}-\frac{C}{\alpha}|\varepsilon|_{\mathcal{H}^2}\right)\left<\xi,\xi\right>_{\mathcal{W}_1^k}-CY(s)\left<\xi,\xi\right>_{\mathcal{W}_1^k}.
	\end{aligned}
\end{equation}
\begin{rem}
If $|\varepsilon|_{\mathcal{H}^k}\leq C\alpha^{3/2}$, by the smallness of $\alpha$, the first term on the right-hand of the inequalities are positive.
\end{rem}
\section{Gradient estimates of temperature equation}\label{chap:tabfig}
Our goal in this section is to derive estimates for transport term $\mathcal{M}_\xi(\xi,\phi),\mathcal{M}_\phi(\xi,\phi)$, and Laplace term of temperature equation in energy estimate~$(\ref{energy})$. First we prove the Laplace term is coercive under $\mathcal{W}_1^k, \mathcal{W}_2^k, \mathcal{W}_3^k$ norm. Then in (\ref{energy}), we add three inequalities up, Laplace terms correspond to a negative term. By the choice of the weights of $\mathcal{W}_1^k, \mathcal{W}_2^k, \mathcal{W}_3^k$ norms,  the transport term $\mathcal{M}_\xi(\xi,\phi),\mathcal{M}_\phi(\xi,\phi)$ is a high order terms in energy estimates (\ref{energy}).
%?????????????????~$\xi,\phi$??????~$(\ref{energy})$????~$\mathcal{M}_\xi(\xi,\phi),\mathcal{M}_\phi(\xi,\phi),$ ???~Laplace?????????
\subsection{Estimates of Laplace term}\label{est_laplace}
Consider the Laplace term in energy estimate $(\ref{energy})$, omit the coefficient coming from time scaling, we proceed to establish following estimates:
%?????????????~$(\ref{energy})$??~Laplace?????????????
\begin{prop}\label{laplaceprop}
	For all $\xi\in\mathcal {W}_1^k\bigcap\mathcal {W}_2^k,\phi\in\mathcal {W}_3^k,\quad k\geq 1$,
	%?????~$(\ref{Equationepsilon})$?~$\xi\in\mathcal {W}_1^k\bigcap\mathcal {W}_2^k,\phi\in\mathcal {W}_3^k,\quad k\geq 1$,
	we have that
	\begin{equation}\label{laplaceest}
		\begin{aligned}
			\left<(\Delta-\frac{1}{\bar\rho^2\cos^2\beta)})\xi,\xi\right>_{\mathcal{W}_1^k}\leq&-C(|\frac{1}{\bar\rho}\partial_\beta\xi|^2_{\mathcal W^k_1}+|\frac{1}{\bar\rho}D_{\bar\rho}\xi|^2_{\mathcal W^k_1}),\\
			\left<(\Delta-\frac{1}{\bar\rho^2\cos^2\beta})\xi,\xi\right>_{\mathcal{W}_2^k}\leq&-C(|\frac{1}{\bar\rho}\partial_\beta\xi|^2_{\mathcal W^k_2}+|\frac{1}{\bar\rho}D_{\bar \rho}\xi|^2_{\mathcal W^k_2}),\\
			\left<\Delta\phi,\phi\right>_{\mathcal{W}_3^k}\leq&-C(|\frac{1}{\bar\rho}\partial_\beta\phi|^2_{\mathcal W^k_3}+|\frac{1}{\bar\rho}D_{\bar \rho}\phi|^2_{\mathcal W^k_3}).
		\end{aligned}
	\end{equation}
	
\end{prop}

To prove this proposition, we need following Hardy inequality with different weights
\begin{lemma}\label{hardy_eta}
	If $ f \in H^1([0,\pi/2])$ and~$-1<\eta<1$, satisfies $ f (0)= f (\pi/2)=0$, we have that
	\begin{equation}
		\int_0^{\pi/2}\frac{ f ^2}{\sin^{\eta+2}(2\beta)}d\beta\leq\frac{1}{(\eta+1)^2}\int_0^{\pi/2}\frac{(\partial_\beta f )^2}{\sin^\eta (2\beta)}d\beta.
	\end{equation}
\end{lemma}
\begin{proof}
	By integration by parts, we have
	\begin{equation*}
		\begin{aligned}
			&\int_0^{\pi/2}\frac{ f ^2}{\sin^{\eta+2}(2\beta)}d\beta=\int_0^{\pi/2}\frac{ f ^2}{\sin^\eta(2\beta)} d\left(-\frac{1}{2}\frac{\cos(2\beta)}{\sin(2\beta)}\right)\\
			=&\int_0^{\pi/2}\partial_\beta f \cdot f \frac{\cos(2\beta)}{\sin^{\eta+1}(2\beta)}d\beta-\eta\int_0^{\pi/2} f ^2\frac{\cos^2 (2\beta)}{\sin^{\eta+2}(2\beta)}d\beta,
		\end{aligned}
	\end{equation*}
	then
	\begin{equation*}
		\begin{aligned}
			&(1+\eta)\int_0^{\pi/2}\frac{ f ^2}{\sin^{\eta+2}(2\beta)}d\beta=\int_0^{\pi/2}\partial_\beta f \cdot f \frac{\cos(2\beta)}{\sin^{\eta+1}(2\beta)}d\beta+\eta\int_0^{\pi/2} f ^2\frac{d\beta}{\sin^\eta (2\beta)}\\
			\leq&\varepsilon\int_0^{\pi/2} f ^2\frac{\cos^2 (2\beta)}{\sin^{\eta+2}(2\beta)} d\beta+\frac{1}{4\varepsilon}\int_0^{\pi/2}(\partial_\beta f )^2\frac{1}{\sin^\eta (2\beta)}d\beta+\eta\int_0^{\pi/2} f ^2\frac{d\beta}{\sin^{\eta}(2\beta)}.
		\end{aligned}
	\end{equation*}
	Setting $\varepsilon=\frac{1+\eta}{2}>\eta$, we have that
	\begin{equation*}
		\begin{aligned}
			(1+\eta)\int_0^{\pi/2}\frac{ f ^2}{\sin^{\eta+2}(2\beta)}d\beta\leq&\frac{1+\eta}{2}\int_0^{\pi/2} f ^2\frac{\cos^2 (2\beta)}{\sin^{\eta+2}(2\beta)} d\beta+\frac{1}{2(1+\eta)}\int_0^{\pi/2}(\partial_\beta f )^2\frac{1}{\sin^\eta (2\beta)}d\beta\\
			\leq& \frac{1+\eta}{2}\int_0^{\pi/2} f ^2\frac{d\beta}{\sin^{\eta+2}(2\beta)} +\frac{1}{2(1+\eta)}\int_0^{\pi/2}(\partial_\beta f )^2\frac{1}{\sin^\eta (2\beta)}d\beta.
		\end{aligned}
	\end{equation*}
\end{proof}

\begin{lemma}\label{hardy}
	For all $ f \in H^1([0,\pi/2]), \eta<1$, we have that
	\begin{equation*}
		\begin{aligned}
			\int_0^{\pi/2} f ^2\cos^{-\eta}\beta d\beta\leq & \frac{4}{(1-\eta)^2}\int_0^{\pi/2}(\partial_\beta f )^2\cos^{2-\eta}\beta d\beta+\\
			&+\left(\frac{2}{1-\eta}+1\right)\int_0^{\pi/2} f ^2\cos^{2-\eta}\beta d\beta.
		\end{aligned}
	\end{equation*}
\end{lemma}
\begin{proof}
	Since $\eta<1$, we use integration by parts
	\begin{equation*}
	\begin{aligned}
		&\int_0^{\pi/2} f ^2\cos^{-\eta}\beta d\beta\\
		=&\int_0^{\pi/2} f ^2\cos^{2-\eta}\beta d\big(\frac{\sin\beta}{\cos\beta}\big)\\
		=&-\int_0^{\pi/2}2\partial_\beta f \cdot f \sin\beta\cos^{1-\eta}\beta d\beta+(2-\eta)\int_0^{\pi/2} f ^2\sin^2\beta\cos^{-\eta}\beta d\beta\\
		=&-\int_0^{\pi/2}2\partial_\beta f \cdot f \sin\beta\cos^{1-\eta}\beta d\beta+(2-\eta)\int_0^{\pi/2} f ^2\cos^{-\eta}\beta d\beta-(2-\eta)\int_0^{\pi/2} f ^2\cos^{2-\eta}\beta d\beta.
	\end{aligned}
\end{equation*}
	we have
	\begin{equation*}
		\begin{aligned}
			(1-\eta)\int_0^{\pi/2} f ^2\cos^{-\eta}\beta d\beta\leq &\varepsilon\int_0^{\pi/2} f ^2\sin^2\beta\cos^{-\eta}\beta d\beta+\frac{1}{\varepsilon}\int_0^{\pi/2}(\partial_\beta f )^2\cos^{2-\eta}\beta d\beta\\
			&+(2-\eta)\int_0^{\pi/2} f ^2\cos^{2-\eta}\beta d\beta\\
			=&\varepsilon\int_0^{\pi/2} f ^2\cos^{-\eta}\beta d\beta+\frac{1}{\varepsilon}\int_0^{\pi/2}(\partial_\beta f )^2\cos^{2-\eta}\beta d\beta\\
			&+(2-\eta-\varepsilon)\int_0^{\pi/2} f ^2\cos^{2-\eta}\beta d\beta.
		\end{aligned}
	\end{equation*}
	Setting $\varepsilon=\frac{1-\eta}{2}$, we obtain the inequality.
	\begin{rem}
		Replace~$\cos^{-\eta}\beta$ by~$\sin^{-\eta}(2\beta)$, with the same method, we also have:
		\begin{equation}
			\begin{aligned}
				\int_0^{\pi/2} f ^2\sin^{-\eta}(2\beta) d\beta\leq & \frac{1}{(1-\eta)^2}\int_0^{\pi/2}(D_\beta f )^2\sin^{-\eta}(2\beta) d\beta\\
				&+\left(\frac{2}{1-\eta}+1\right)\int_0^{\pi/2} f ^2\sin^{2-\eta}(2\beta) d\beta.
			\end{aligned}
		\end{equation}
	\end{rem}
\end{proof}

\begin{proof}[\textbf{proof of Proposition~$\ref{laplaceprop}$}]

	Setting $f=D^i_{\bar\rho}D^j_\beta \xi$, $i+j\leq k$, recall that
\begin{align}
	(\Delta-\frac{1}{\bar\rho^2\cos^2\beta}) f =&\left\{\frac{1}{\bar\rho^2}\partial_{\bar\rho}(\bar\rho^2\partial_{\bar{\rho}} f )\right\}+\left\{\frac{1}{\bar\rho^2 \cos\beta}\partial_{\beta}(\cos\beta\partial_{\beta} f )-\frac{1}{\bar\rho^2 \cos^2\beta} f \right\}\\
	=:&A_{\bar\rho}  f +A_\beta  f .
\end{align}
First we consider $A_{\bar\rho}f, A_\beta f$, i.e. the radial and angle derivative separately, proof the coercivity under different weights in $\mathcal{W}_1^k, \mathcal{W}_2^k,\mathcal{W}_3^k$. And we omit the commutator which we discuss at the end, since the commutators are low order terms.

\textbf{Step 1}: $A_{\bar\rho}  f ,$ The radial direction. 

We have the $\bar\rho$ weight in norm of~$\mathcal{W}_1^k$, $\mathcal{W}_3^k$ under polar coordinates are $\bar\rho^2 d\bar\rho$, the Laplace operator is negative definite. The $\mathcal{W}_2^k$ norm is different, under the weight~$\bar\rho^\eta d\bar\rho$, we integrate by parts:

\begin{equation}\label{rho}
	\begin{aligned}
		&\int_0^\infty A_{\bar\rho} f   f \bar\rho^{\eta}d\bar\rho\\
		=&\int_0^\infty \frac{1}{\bar\rho^2}\partial_{\bar\rho}(\bar\rho^2\partial_{\bar\rho}  f ) f \bar\rho^{\eta}d\bar\rho\\
		=&-\int_0^\infty (\partial_{\bar\rho}  f )^2\bar\rho^{\eta}d\bar\rho+\int_0^\infty (2-\eta)\bar\rho^{\eta-1}\partial_{\bar\rho} f \cdot f  d\bar\rho\\
		=&-\int_0^\infty (\partial_{\bar\rho}  f )^2\bar\rho^{\eta}d\bar\rho+\frac{(2-\eta)(1-\eta)}{2}\int_0^\infty  f ^2 \bar\rho^{\eta-2}d\bar\rho.
	\end{aligned}
\end{equation}
Where~$1-\eta=\frac{1}{100}$,therefore, the second term in the above formula is a positive term but has a small coefficient, and can be absorbed by following $A_\beta  f $ part energy integral $(\ref{laplace}).$

\textbf{Step 2.} $A_\beta  f $, i.e. the $\beta$ direction.

First we consider the more singular terms that with $\sin^{-\eta}(2\beta)$ weight, by the definition of $\mathcal{W}_i^k$ (\ref{normdef}), $f=D_{\bar\rho}^iD_\beta^j \xi$, with $(i,j)\neq(k,0)$.
\begin{equation}\label{laplace}
	\begin{aligned}
		&\int_0^{\pi/2}A_\beta  f  \cdot  f  \sin^{-\eta}(2\beta) d\beta\\
		=&\int_0^{\pi/2}\left(\frac{1}{\cos\beta}\partial_\beta(\cos\beta\partial_\beta f )-\frac{ f }{\cos^2\beta}\right)\cdot f \sin^{-\eta}(2\beta) d\beta\\
		=&\int_0^{\pi/2}\left\{-(\partial_\beta f )\sin^{-\eta}(2\beta)-\partial_\beta f \cdot f \cos\beta\partial_\beta\left(\frac{1}{\cos\beta\sin^\eta (2\beta)}\right)-\frac{ f ^2}{\cos^2 \beta}\sin^{-\eta}(2\beta)\right\} d\beta\\
		=&\int_0^{\pi/2}\left\{-(\partial_\beta f )^2\sin^{-\eta}(2\beta)+\frac{1}{2} f ^2\cdot\partial_\beta\left(\cos\beta\partial_\beta\left(\frac{1}{\cos\beta\sin^\eta (2\beta)}\right)\right)-\frac{ f ^2}{\cos^2 \beta}\sin^{-\eta}(2\beta)\right\} d\beta\\
		=&\int_0^{\pi/2}\left\{-(\partial_\beta f )^2\sin^{-\eta}(2\beta)+ f ^2\left(-\frac{\eta(1+2\eta)}{\sin^\eta(2\beta)}\!+\!\frac{2(1+\eta)\sin^2\beta}
		{\sin^{\eta+2}(2\beta)}\!+\!\frac{2\eta(1+\eta)}{\sin^{\eta+2}(2\beta)}\!-\!4\frac{\sin^2\beta}{\sin^{\eta+2}(2\beta)}\right)\right\}d\beta \\
		=&\int_0^{\pi/2}\left\{-(\partial_\beta f )^2\sin^{-\eta}(2\beta)-\frac{\eta(1+2\eta)}{\sin^\eta (2\beta)} f ^2-\frac{2(1-\eta)\sin^2\beta}
		{\sin^{\eta+2}(2\beta)} f ^2+\frac{2\eta(1+\eta)}{\sin^{\eta+2}(2\beta)} f ^2\right\} d\beta.
	\end{aligned}
\end{equation}
For the positive term in $(\ref{laplace})$, by Lemma \ref{hardy_eta}, we have
\begin{equation*}
	\int_0^{\pi/2}\frac{2\eta(1+\eta)}{\sin^{\eta+2}(2\beta)} f ^2 d\beta\leq\frac{2\eta}{1+\eta}\int_0^{\pi/2}\frac{(\partial_\beta f )^2}{\sin^\eta (2\beta)}d\beta.
\end{equation*}

Note that $\eta=\frac{99}{100}<1$, $\frac{2\eta}{1+\eta}<1$, this term controlled by $(\partial_\beta f )^2$ term, which deduces that $(\ref{laplace})$ is negative definite.

Then by $(\ref{laplace})$ and $(\ref{rho})$, we have that
\begin{small}
	\begin{equation}
		\begin{aligned}
			\int\!(\!\Delta\!-\!\frac{1}{\bar\rho^2\!\cos^2\!\beta}) f \!\cdot\! f \bar\rho^2\sin^{-\eta}\!(2\beta) d\bar\rho d\beta&\!\leq\!-\!\int\left\{(\bar\rho\partial_{\bar\rho} f )^2\sin^{-\eta}(2\beta)\!+\!\frac{1-\eta}{1+\eta}(\partial_\beta f )^2
			\sin^{-\eta}\!(2\beta)\right\} d\bar\rho d\beta\\\label{laplace_xi1}
			&-\eta(1+2\eta)\int  f ^2 \frac{d\bar\rho d\beta}{\sin^\eta (2\beta)}  - 2(1-\eta)\int  f ^2 \frac{\sin^2\beta}{\sin^{\eta+2}(2\beta)} d\bar\rho d\beta,\\
		\end{aligned}
	\end{equation}
	
	\begin{equation}\label{laplace_xi2}
		\begin{aligned}
			\int\!(\!\Delta\!-\!\frac{1}{\bar\rho^2\cos^2\!\beta}) f \!\cdot\! f \bar\rho^\eta\sin^{-\eta}\!(2\beta)\!d\bar\rho d\beta&\!\leq\!-\!\int \left\{(\bar\rho\partial_{\bar\rho} f )^2\!\sin^{-\eta}(2\beta)\!+\!\frac{1-\eta}{1+\eta}\!(\partial_\beta f )^2\!\sin^{-\eta}\!(2\beta) \right\}\bar\rho^{\eta-2} d\bar\rho d\beta\\
			&-\int\left(\eta(1+2\eta)-\frac{(2-\eta)(1-\eta)}{2}\right)\frac{ f ^2}{\sin^\eta (2\beta)}\bar\rho^{\eta-2}d\bar\rho d\beta\\
			&- \int\frac{2(1-\eta)\sin^2\beta}
			{\sin^{\eta+2}(2\beta)} f ^2\bar\rho^{\eta-2} d\bar\rho d\beta.
		\end{aligned}
	\end{equation}
\end{small}
As for the case that the $\beta$ weight is $\sin^{2-\eta}(2\beta) d\beta$, recall the definition of $\mathcal{W}_1^k, \mathcal{W}_2^k, \mathcal{W}_3^k $ (\ref{normdef}), only $f=D_{\bar\rho}^k \xi$ is with $\sin^{2-\eta}(2\beta) d\beta$ weight, for this term we have that

\begin{equation}
	\begin{aligned}
		&\int_0^{\pi/2} \left(\frac{1}{\cos\beta}\partial_\beta(\cos\beta\partial_\beta f )-\frac{ f }{\cos^2\beta}\right) f \sin^{2-\eta}(2\beta) d\beta\\
		=&\int_0^{\pi/2} \Big\{-(\partial_\beta f )^2\sin^{2-\eta}(2\beta)-(2-\eta)(3-2\eta) f ^2\sin^{2-\eta} (2\beta)-\frac{2(3-\eta)\sin^2\beta}
		{\sin^{\eta}(2\beta)} f ^2\\
		&+\frac{2(2-\eta)(1-\eta)}{\sin^{\eta}(2\beta)} f ^2 \Big\}d\beta.
	\end{aligned}
\end{equation}
The only positive term of the right-hand side above has coefficient $1-\eta$, note that in (\ref{laplace_xi1}, \ref{laplace_xi2}), we have negative term $(D_{\bar\rho}g)^2 \sin^{-\eta}(2\beta)$ with $g=D_{\bar\rho}^{k-1}\xi$,  and satisfies $f=D_{\bar\rho}g$. then the positive term above $f^2\sin^{-\eta}(2\beta)$ is absorbed in (\ref{laplaceest}).

For $\mathcal{W}_3^k$ norm, setting $f=D_{\bar\rho}^iD_{\beta}^j \phi$, recall the definition of $\mathcal{W}_3^k$ (\ref{normdef}), the difference is $\beta$ weight, there is one more $\partial_\beta$ derivative (not $D_\beta$), we need to consider following two integral. Recall the boundary condition (\ref{boundary_condition3}), using integrating by parts we have

\begin{equation}
	\begin{aligned}
		&\int \partial_{\beta} \Delta f \cdot \partial_{\beta} f \bar{\rho}^{2} \cos ^{2-\eta} \beta d \bar{\rho} d \beta=-\int\left(\bar{\rho} \partial_{\bar{\rho}} \partial_{\beta} f\right)^{2} \cos ^{2-\eta} \beta+\left(\partial_{\beta \beta} f\right)^{2} \cos ^{2-\eta} \beta d \bar{\rho} d \beta \\
		&-\int\left(1-\frac{(1-\eta)^{2}}{2}\right)\left(\partial_{\beta} f\right)^{2} \cos ^{-\eta} \beta+\frac{(2-\eta)(1-\eta)}{2}\left(\partial_{\beta} f\right)^{2} \cos ^{2-\eta} \beta d \bar{\rho} d \beta,
	\end{aligned}
\end{equation}
and $\beta$ term of $\Delta f$
\begin{equation}
	\begin{aligned}
		&\int_0^{\pi/2} \frac{1}{\cos \beta} \partial_{\beta}\left(\cos \beta \partial_{\beta} f\right) \cdot f \cos ^{2-\eta} \beta d \beta \\
		=&-\int_0^{\pi/2}\left(\partial_{\beta} f\right)^{2} \cos ^{2-\eta} \beta d \beta+\frac{(1-\eta)^{2}}{2} \int_0^{\pi/2} f^{2} \cos ^{-\eta} \beta d \beta-\frac{(1-\eta)(2-\eta)}{2} \int_0^{\pi/2} f^{2} \cos ^{2-\eta} \beta d \beta.
	\end{aligned}
\end{equation}
By Lemma \ref{hardy}, we have
\begin{equation}
	\int \frac{1}{\cos \beta} \partial_{\beta}\left(\cos \beta \partial_{\beta} f\right) \cdot f \cos ^{2-\eta} \beta d \beta \leq \int\left(\partial_{\beta} f\right)^{2} \cos ^{2-\eta} \beta d \beta+\frac{(1-\eta)(3-2 \eta)}{2} \int f^{2} \cos ^{2-\eta} \beta d \beta,
\end{equation}

then we have estimate of $\Delta\phi$ under weight $\bar\rho^2\cos^{2-\eta}\beta d\bar\rho d\beta$
	\begin{equation}
	\int\Delta\phi\cdot\phi\bar\rho^2\cos^{2-\eta}\beta d\bar\rho d\beta\leq\!-\!(1-2(1-\eta))\!\int(\bar\rho\partial_{\bar\rho}\phi)^2\cos^{2-\eta}\beta d\bar\rho d\beta+\int(\partial_\beta\phi)^2\cos^{2-\eta}\beta d\bar\rho d\beta.
\end{equation}

Next we consider the commutator, note that the exchange terms of the high-order derivative and the Laplace operator are all low-order derivatives, which can be absorbed by selecting an appropriate equivalent norm. Note that $D_\beta$  don't change the order of $\beta$ at $\beta=0$ or $\pi/2$, for simplicity we only write the integral of variable $\bar \rho$, the $D_\beta$ derivative terms are similar. Note that $D_{\bar \rho}\tilde\Delta=\tilde\Delta D_{\bar \rho}$, remainder terms come from $1/\bar\rho^2$

	\begin{equation*}
		D_{\bar\rho}^l(\Delta\xi)=D_{\bar \rho}^l\left(\frac{1}{\bar\rho^2}\tilde{\Delta}\xi\right)=\frac{1}{\bar{\rho}^{2}}\tilde{\Delta}D_{\bar\rho}^l\xi+\frac{1}{\bar{\rho}^{2}}\sum_{i=0}^{l-1}c_i\tilde{\Delta}D_{\bar\rho}^{i}\xi,
	\end{equation*}
	where~$c_i$ is the binomial coefficients independent of $\alpha$.
	
Then the first term of the above formula is estimated by the $L_2$ estimate, for remaining terms we use Cauchy-Schwarz inequality, then we have that

\begin{equation}
	\left<D^l(\Delta\xi),D^l\xi\right>_{\mathcal{W}_1^0}\leq-a_l(|\partial_\beta D_{\bar \rho}^l\xi|_{\mathcal{W}_1^0}^2+|\partial_{\bar\rho}D_{\bar\rho}^{l}\xi|_{\mathcal{W}_1^0}^2)+A_l \sum_{i=0}^{l-1} (|\partial_\beta D_{\bar \rho}^{i}\xi|_{\mathcal{W}_1^0}^2+|\partial_{\bar\rho}D_{\bar \rho}^{i}\xi|_{\mathcal{W}_1^0}^2),
	\end{equation}
by choosing a equivalent norm, we obtain the coercivity of the Laplace term in energy estimates~$(\ref{Equationepsilon})$.

%	In order to eliminate the lower order derivative terms, we need to select an appropriate equivalent norm. First, there are positive constants $C_{k,i}>0$ that don't depend on $\alpha$ (the selection only depends on $a_k,A_k$), such that
%
%	\begin{equation}
%		\sum_{i=0}^k C_{k,i}\left<D_{\bar \rho}^i(\Delta\xi),D_{\bar \rho}^i\xi\right>_{\mathcal{W}_1^0}\leq -c\sum_{i=0}^k(|\partial_\beta D_{\bar \rho}^i\xi|_{\mathcal{W}_1^0}^2+|D_{\bar \rho}^{i+1}\xi|_{\mathcal{W}_1^0}^2).
%	\end{equation}
%	So we consider the equivalent norm:
%	%??????????
%	\begin{equation}
%		|f|_{\mathcal{W}_1^k}^2=\sum_{h=0}^k\sum_{i=0}^h{C_{h,i}}|D_{\bar{y}}^{i}D_\beta^{h-i}f|_{\mathcal{W}_1^0}^2.
%	\end{equation}
%	Then we obtain the coercivity of the Laplace term in energy estimates~$(\ref{Equationepsilon})$.
\end{proof}

%\subsection{Estimates of $\mathcal{M}_{\xi}$ and $\mathcal{M}_{\phi}$}
\subsection{Estimates of advection term}
Next, we give estimates of the transport terms in equation (\ref{Equationepsilon}), namely $\mathcal{M}_{\xi}$ and $\mathcal{M}_{\phi}$. Note that $\Omega=F+\varepsilon$, we have the decomposition that
%???????$\xi,\phi$??$(\ref{Equationepsilon})$?????$\mathcal{M}_\xi$ ? $\mathcal{M}_\phi$ ?????????$\xi,\phi$ ??????$u$???????????? $\Phi_W$????????????????$(\ref{decomposition})$??????????$\Phi_F$ ?????????????$\Phi_\varepsilon$ ?????????????
\begin{align*}
	\mathcal{M}_\xi=&\mathcal{M}^{F}_\xi+\mathcal{M}^{\varepsilon}_\xi,\\
	\mathcal{M}_\phi=&\mathcal{M}^{F}_\phi+\mathcal{M}^{\varepsilon}_\phi,
\end{align*}
where
\begin{align}\label{transport}
	&\mathcal{M}^{F}_\xi(\xi,\phi)=U(\Phi_F)\partial_\beta\xi+V(\Phi_F)\bar{\rho}\partial_{\bar{\rho}}\xi+\Lambda_1(\Phi_F)\xi
	+\Lambda_2(\Phi_F)\phi,\\
	&\mathcal{M}^{\varepsilon}_\xi(\xi,\phi)=U(\Phi_\varepsilon)\partial_\beta\xi+V(\Phi_\varepsilon)\bar{\rho}\partial_{\bar{\rho}}\xi
	+\Lambda_1(\Phi_\varepsilon)\xi+\Lambda_2(\Phi_\varepsilon)\phi, \label{trans_M2}\\
	&\mathcal{M}^{F}_\phi(\xi,\phi)=U(\Phi_F)\partial_\beta\phi+V(\Phi_F)\bar{\rho}\partial_{\bar{\rho}}\phi+\Lambda_3(\Phi_F)\xi
	+\Lambda_4(\Phi_F)\phi,\\
	&\mathcal{M}^{\varepsilon}_\phi(\xi,\phi)=U(\Phi_\varepsilon)\partial_\beta\phi+V(\Phi_\varepsilon)\bar{\rho}\partial_{\bar{\rho}}\phi
	+\Lambda_3(\Phi_\varepsilon)\xi+\Lambda_4(\Phi_\varepsilon)\phi.
\end{align}
For $\mathcal{M}^{F}_\xi(\xi,\phi), \mathcal{M}^{F}_\phi(\xi,\phi)$, we have the formula of $\Phi_{F}$ (\ref{decomposition}), for $\mathcal{M}^{\varepsilon}_\xi(\xi,\phi), \mathcal{M}^{\varepsilon}_\phi(\xi,\phi)$, we have the estimate of $\Phi_{\varepsilon}$ (\ref{thm_biot}), thus we can estimate separately.

\subsubsection*{Estimates of the main terms}
For $\mathcal{M}^{F}_\xi(\xi,\phi)$ and $\mathcal{M}^{F}_\phi(\xi,\phi)$,  which are corresponding to the gradient of convection terms in the energy estimates~$(\ref{energy})$, we have the following proposition

\begin{prop}\label{Mest1}
	For all $\xi\in\mathcal {W}_1^k\bigcap\mathcal {W}_2^k,\phi\in\mathcal {W}_3^k,\quad k\geq 1$, we have that:
\begin{equation}\label{M^F}
	\begin{aligned}
		&\left<\mathcal{M}^{F}_\xi(\xi,\phi),\xi\right>_{\mathcal{W}^k_1}\leq l_2^{-1/2}(0)\left(|\xi|_{\mathcal{W}^k_1}^2+|\phi|_{\mathcal{W}^k_3}^2\right)+l_2^{1/\alpha}(0)\left(|\frac{1}{\bar\rho}\xi|_{\mathcal{W}^k_1}^2+|\frac{1}{\bar\rho}\phi|_{\mathcal{W}^k_3}^2\right),\\
		&\left<\mathcal{M}^{F}_\xi(\xi,\phi),\xi\right>_{\mathcal{W}^k_2}\leq l_2^{-1/2}(0)\left(|\xi|_{\mathcal{W}^k_2}^2+|\phi|_{\mathcal{W}^k_3}^2\right)+l_2^{1/\alpha}(0)\left(|\frac{1}{\bar\rho}\xi|_{\mathcal{W}^k_2}^2+|\frac{1}{\bar\rho}\phi|_{\mathcal{W}^k_3}^2\right),\\
		&\left<\mathcal{M}^{F}_\phi(\xi,\phi),\phi\right>_{\mathcal{W}^k_3}\leq l_2^{-1/2}(0)\left(|\xi|_{\mathcal{W}^k_2}^2+|\phi|_{\mathcal{W}^k_3}^2\right)+l_2^{1/\alpha}(0)\left(|\frac{1}{\bar\rho}\xi|_{\mathcal{W}^k_1}^2+|\frac{1}{\bar\rho}\phi|_{\mathcal{W}^k_3}^2\right).
	\end{aligned}
\end{equation}

\end{prop}
\begin{rem}
	Note that the we can choose $l_2(0)$ big enough in Proposition \ref{prop_time}, such that the first term in the right-hand side above like  $|\xi|_{\mathcal{W}^k_1}^2+|\phi|^2_{\mathcal{W}_3^k}$ can be absorbed by $(\ref{estimate_2})$ in energy estimates (\ref{energy}). And by Proposition \ref{prop_s}, the behavior of Laplace term in energy estimate (\ref{energy}) $\Big(\frac{\mu l_2}{\lambda^{1+\delta}}\Big)^{\frac{2}{\alpha}}\lambda Y(s)\approx l_2^{2/\alpha}(0)Y(s)$, the second term in the above equalities, can be absorbed by (\ref{laplaceest}), recall that $Y(s)$.
\end{rem}
	For $\mathcal{M}^{F}_\xi(\xi,\phi),\mathcal{M}^{F}_\phi(\xi,\phi)$, there are $\Lambda_2(\Phi_F)\phi,\Lambda_3(\Phi_F)\xi$ appearing respectively, by using elliptic estimate $(\ref{thm_biot})$, we have

\begin{equation}\label{UV2}
	\begin{aligned}
		U\left(\Phi_{F}\right)&=-3 \sin (2 \beta) \frac{1}{1+y}+O(\alpha), \\
		V\left(\Phi_{F}\right)&=\left(2 \cos (2 \beta)-2 \sin ^{2}\beta\right) \frac{1}{1+y}+O(\alpha),
	\end{aligned}
\end{equation}
\begin{equation}\label{Lambda2}
	\begin{aligned}
		\Lambda_1\left(\Phi_{F}\right)&=\frac{2}{1+y}+O(\alpha),\quad		\Lambda_2\left(\Phi_{F}\right)=O(\alpha),\\
		\Lambda_3\left(\Phi_{F}\right)&=O(\alpha),\quad		\Lambda_4\left(\Phi_{F}\right)=-\frac{4\cos^2\beta}{1+y}+O(\alpha).
	\end{aligned}
\end{equation}
Where $O(1)$ is bounded term of $\mathcal{H}^k$ norm.

Here, the difficulty is the main term, since the remainder term is  $\mathcal{H}^k$ norm bounded with coefficient of order $\alpha$. Note that the formulation of $y$ variable is $\frac{1}{1+y}$, we have
%??~$\Phi_{F}$?~$y$ ????~$\frac{1}{1+y}$, ??????????????~$\bar y=l_2(s)y$, ??~$l_2(s)$ ?????? ???
\begin{equation*}
	\frac{1}{1+y}=\frac{l_2(s)}{l_2(s)+\bar y}\leq\frac{l_2(s)}{\bar y}.
\end{equation*}
In $(\ref{energy})$, first we consider the integral in $\bar\rho$ variable, we have
%?$U\left(\Phi_{F}\right)$??%
%???$\mathcal W_1^k$ ???,??$\bar\rho$ ?????,??$D_{\bar \rho}=\alpha\bar y\partial_{\bar y}$ ???$\bar y$ ????????$\{c_{k,j}:1\leq j\leq k-1\}$??$D_{\bar \rho}^k\frac{l_2(s)}{l_2(s)+\bar y}$ ???
\begin{equation*}
	D_{\bar \rho}^k\frac{l_2(s)}{l_2(s)+\bar y}=\alpha^{k}\frac{l_2(s)}{l_2(s)+\bar y}\cdot\sum_{j=1}^{k-1}c_{k,j}\frac{\bar y ^j}{(l_2(s)+\bar y)^j}.
\end{equation*}
The order at infinity and 0 is $\bar y^{-1}=\bar\rho^{-\alpha}$ which doesn't depend on $k$, we consider $k=0$ case for simplicity, for $U(\Phi_ {F})$ we use integration by parts
%?????~0 ?????~$\bar y^{-1}=\bar\rho^{-\alpha}$????~k????????$k=0$, ??$U(\Phi_{F})$ ??$\beta$ ????????
\begin{equation*}
	\begin{aligned}
		&\left<U(\Phi_{F})\partial_\beta\xi,\xi\right>_{\mathcal{W}^0_1}\\
		\leq& C\int_0^{\infty}\int_0^{\pi/2}\frac{l_2(s)}{l_2(s)+\bar\rho^{\alpha}}\xi^2\bar \rho^{2}\sin^{-\eta}(2\beta) d\bar\rho d\beta+\alpha|\xi|_{\mathcal{W}^0_1}^2
	\end{aligned}
\end{equation*}
By smallness of $\alpha$, we set $0<\tau<1$ such that $2-\alpha=\tau\cdot2+(1-\tau)\cdot 0$, i.e. $\tau=1-\alpha/2$, then for $\delta_0 >0$, by Young inequality
\begin{equation*}
	\begin{aligned}
		&\int_0^{\infty}\int_0^{\pi/2}(\bar \rho^{2-\alpha}\frac{l_2(s)\bar\rho^{\alpha}}{l_2(s)+\bar\rho^\alpha})\xi^2\sin^{-\eta}(2\beta) d\bar\rho d\beta\\
		\leq&\frac{2}{2-\alpha}\delta_{0}^{\frac{2}{2-\alpha}} \int_0^{\infty}\int_0^{\pi/2}\xi^2\bar\rho^{2}\sin^{-\eta}(2\beta) d\bar\rho d\beta+\frac{\alpha}{2}\int_0^{\infty}\int_0^{\pi/2}\xi^2\sin^{-\eta}(2\beta)\left(\frac{\delta_0^{-1}l_2(s)\bar\rho^\alpha}{l_2(s)+\bar\rho^\alpha}\right)^{2/\alpha} d\bar\rho d\beta\\
		<&\delta_{0}\int_0^{\pi/2}\xi^2\bar\rho^{2}\sin^{-\eta}(2\beta) d\bar\rho d\beta+\delta^{-2/\alpha}_0 \int_0^{\infty}\int_0^{\pi/2}\xi^2\sin^{-\eta}(2\beta)d\bar\rho d\beta.	\end{aligned}
\end{equation*}
Choose $\delta_{0}=l_2^{-1/2}(0)$, then $\delta_{0}^{-2/\alpha}=l_2^{1/\alpha}(0)$, 
By $(\ref{l_1l_2})$, we obtain $(\ref{M^F})$. The treatment for for~$\mathcal{W}^k_2,\mathcal{W}^k_3$ cases is the same.
%?~$\mathcal{W}^k_2,\mathcal{W}^k_3$????????????~$(\ref{M^F})$?
\subsubsection*{Estimates of remainder term}
For $\mathcal{M}_\xi^\varepsilon$ and $\mathcal{M}_\phi^\varepsilon$, we have following proposition:
\begin{prop}\label{Mest2}
	For all~$\xi\in\mathcal{W}^k_1\bigcap\mathcal{W}^k_2, \phi\in\mathcal{W}^k_3,\quad k\geq4$, we have
	\begin{equation}\label{prop_est}
		\begin{aligned}
			& \left<\mathcal M^{\varepsilon}_\xi(\xi,\phi),\xi\right>_{\mathcal{W}^k_1}\leq\frac{C}{\alpha^{3/2}}|\varepsilon|_{\mathcal {H}^k}|\xi|_{\mathcal{W}^k_1}(|\xi|_{\mathcal{W}^k_1}+|\phi|_{\mathcal{W}^k_3}),\\
			& \left<\mathcal M^{\varepsilon}_\xi(\xi,\phi),\xi\right>_{\mathcal{W}^k_2}\leq\frac{C}{\alpha^{3/2}}|\varepsilon|_{\mathcal {H}^k}|\xi|_{\mathcal{W}^k_2}(|\xi|_{\mathcal{W}^k_2}+|\phi|_{\mathcal{W}^k_3}),\\
			& \left<\mathcal M^{\varepsilon}_\phi(\xi,\phi),\phi\right>_{\mathcal{W}^k_3}\leq\frac{C}{\alpha^{3/2}}|\varepsilon|_{\mathcal {H}^k}|\phi|_{\mathcal{W}^k_1}(|\xi|_{\mathcal{W}^k_1}+|\phi|_{\mathcal{W}^k_3}).
		\end{aligned}
	\end{equation}
\end{prop}
We have to consider the terms of $\Phi_\varepsilon$ in $(\ref{transport})$, in ~\cite{Elgindi1904}, author has estimate of $U(\Phi_\varepsilon)$:
%??????$(\ref{transport})$ ??? $\Phi_\varepsilon$ ???
%?~$\cite{Elgindi1904}$?????\\$U(\Phi_\varepsilon), V(\Phi_\varepsilon)$ ????$(\ref{U_V})$????$\ref{thm_biot}$,???
\begin{equation}\label{UV}
	\left|\frac{U(\Phi_\varepsilon)}{\sin (2\beta)}\right|_{\mathcal H^k}\leq\frac{C}{\alpha}\left|\varepsilon\right|_{\mathcal H^k }.
\end{equation}

In addition, in order to make $\mathcal W_i^k$ norm and $\mathcal H^k $ norm match (i=1,2,3), we need higher order estimates of $V(\Phi_\varepsilon)$and $\Lambda_i(\Phi_\varepsilon)$.
%??????????~$\mathcal W_i^k$???~$\mathcal H^k $ ????(i=1,2,3)??????~$V(\Phi_\varepsilon)$?~$\Lambda_i(\Phi_\varepsilon),(i=1,2,3,4)$????????
For $V(\Phi_\varepsilon)$, recall its definition (\ref{U_V}) we have that
%????~$V(\Phi_\varepsilon)$??????~$(\ref{U_V})$, ????????~$\partial_\beta$?
\begin{equation}
	\partial_\beta V(\Phi_\varepsilon)=\partial_{\beta\beta}\Phi_\varepsilon-\sin\beta\partial_\beta\left(\frac{\Phi_\varepsilon}{\cos\beta}\right)-\Phi_\varepsilon.
\end{equation}

By (\ref{thm_biot}), we have

\begin{equation}\label{partialV}
	\begin{aligned}
		\partial_\beta V(\Phi_\varepsilon)=&\partial_{\beta\beta}\Phi_\varepsilon-\sin\beta\partial_\beta(\frac{\Phi_\varepsilon}{\cos\beta})-\Phi_\varepsilon\\
		=&-\frac{3}{2\alpha}\sin(2\beta) L_{12}(\varepsilon)+\left(\partial_{\beta\beta}\bar{\Phi}_\varepsilon-\sin\beta\partial_\beta\left(\frac{\bar{\Phi}_\varepsilon}{\cos\beta}\right)-\bar{\Phi}_\varepsilon\right).	\end{aligned}
\end{equation}
The only term we need deal with is $\partial_\beta(\frac{\bar{\Phi}_\varepsilon}{\cos\beta})$, we have following lemma
%?~$\cite{Masmoudi}$????????~$\partial_\beta(\frac{\bar{\Phi}_\varepsilon}{\cos\beta})$ ?~$\mathcal{H}^k$ ?????~$\frac{1}{\alpha}|\varepsilon|_{\mathcal{H}^k}$ ??,?~$\partial_\beta(\frac{\bar{\Phi}_\varepsilon}{\cos\beta})$????????????
\begin{lemma}\label{hardy2}
	For all $f(z,\beta)\in C_0^\infty([0,\infty))\times[0,\pi/2])$ and $-1<\eta<1$, satisfying $f(z,0)=f(z,\pi/2)=0$, for $z>0$, we have that
	\begin{equation}
		|\partial_\beta(\frac{f}{\sin(2\beta)})|_{\mathcal{H}^k}\leq C|\partial_{\beta\beta}f|_{\mathcal{H}^k}.
	\end{equation}
\end{lemma}
\begin{proof}
	
	%???????????????~$\beta$ ?????????????~$\beta\in(0,\pi/4)$ ?????~$\beta\in(\pi/4,\pi/2)$ ???????????\\
	Observe that for $\beta\in(0,\pi/4)$, we have
	%??????~$\frac{4}{\pi}\beta\leq \sin(2\beta)\leq (2\beta)$???????~$\beta$???:
	\begin{equation}
		\int_{0}^{\pi/4}|D^k_\beta\partial_\beta(\frac{f}{\sin(2\beta)}) |^2\sin^{-\gamma}(2\beta) d\beta,
	\end{equation}
	recall that $D^k_\beta=(\sin(2\beta)\partial_\beta)^k$, for $1\leq i\leq k$ we have
	%?~$D^k_\beta=(\sin(2\beta)\partial_\beta)^k$???????????,???~$1\leq i\leq k$
	\begin{equation}
		\begin{aligned}
			&\int_{0}^{\pi/4}|(\sin(2\beta))^{i}\partial^{i+1}_\beta(\frac{f}{\sin(2\beta)}) |^2\sin^{-\gamma}(2\beta) d\beta\\
			\leq&C\int_{0}^{\pi/4}\beta^{2i-\gamma}(\partial^{i+1}_\beta(\frac{f}{\beta}))^2 d\beta.
		\end{aligned}
	\end{equation}
	%???~$\cite{Elgindi1904}$?????
	Setting $h=\frac{f}{\beta}$, then~$\partial^{i+2}_\beta f=\beta\partial^{i+2}h+(i+2)\partial_\beta^{i+1}h$, we have
	\begin{equation*}
		\begin{aligned}
			&\int_{0}^{\pi/4}\beta^{2i-\gamma}\partial^{i+1}_\beta h\cdot \partial_\beta^{i+2}f d\beta\\
			=&\int_{0}^{\pi/4}\beta^{2i-\gamma}((i+2)(\partial^{i+1}_\beta h)^2+\beta\partial_\beta^{i+1}h\cdot\partial_\beta^{i+2}h) d\beta\\
			=&\frac{\beta^{2i-\gamma+1}}{2}(\partial_\beta^{k}h)^2\Big|_{\beta=\pi/4}+\left((i+2)-\frac{2i-\gamma+1}{2}\right)\int_{0}^{\pi/4}\beta^{2i-\gamma}((\partial^{i+1}_\beta h)^2 d\beta.
		\end{aligned}
	\end{equation*}
	For the second term, from the symmetry that the boundary term in the integration by parts and the boundary term of the corresponding integral on $\beta\in(\pi/4,\pi/2)$ are eliminated, then by Cauchy-Schwarz inequality, we have
	%?????????????????????????????????~$\beta\in(\pi/4,\pi/2)$??????????????~Cauchy-Schwarz ???
	\begin{align*}
		&\int_{0}^{\pi/4}|\beta^{2i-\gamma}(\partial^{i+1}_\beta(\frac{f}{\beta}))|^2 d\beta+
		\int_{\pi/4}^{\pi/2}|(\pi/2-\beta)^{2i-\gamma}(\partial^{i+1}_\beta(\frac{f}{\beta}))|^2 d\beta\\
		\leq&C\left(\int_{0}^{\pi/4}\beta^{2i-\gamma}(\partial^{i+2}_\beta f)^2 d\beta+\int_{\pi/4}^{\pi/2}(\pi/2-\beta)^{2i-\gamma}(\partial^{i+2}_\beta f)^2 d\beta\right)\\
		\leq&C \int_{0}^{\pi/2} (D^i_\beta\partial_{\beta\beta}f)^2\sin^{-\gamma}(2\beta) d\beta,
	\end{align*}
	%??~$\mathcal{H}^k$????~$|D_z^i(\frac{f}{\sin(2\beta)})|^2\sin^{-\eta}(2\beta)$ ?????????????????????~$\beta$ ?????~$\sin^{-\eta}(2\beta)$?????
	\begin{equation*}
		\int_{0}^{\pi/2}(D^i_y(\frac{f}{\beta}))^2 \sin^{-\eta}(2\beta) d\beta	\leq C \int_{0}^{\pi/2} (D^i_y\partial_{\beta\beta}f)^2\sin^{-\eta}(2\beta) d\beta.
	\end{equation*}
	%??????????
\end{proof}
For~$\partial_\beta(\frac{\bar{\Phi}_\varepsilon}{\cos\beta})$ we have that
\begin{equation*}
	\left| \partial_\beta(\frac{\bar{\Phi}_\varepsilon}{\cos\beta})\right|_{\mathcal{H}^k}\leq C|\varepsilon|_{\mathcal{H}^k}.
\end{equation*}
For~$(\ref{partialV})$, we have that
\begin{equation}\label{V}
	\left|\partial_\beta V(\Phi_\varepsilon)\right|_{\mathcal{H}^k}\leq\frac{C}{\alpha}|\varepsilon|_{\mathcal{H}^k}.
\end{equation}

As for~$\Lambda_i(\Phi_\varepsilon),i=1,2,3,4$, note that every terms in $(\ref{Lambda})$ is Second order derivative terms of stream function $\Phi_\varepsilon$.
%??%
%???????~$(\ref{Lambda})$ ???????????~$\Phi_\varepsilon$ ????????????~$\Lambda_2(\Phi_\varepsilon)$ ???~$\beta$ ?????
%\begin{equation*}
%	\begin{aligned}
%		\Lambda_2(\Phi_\varepsilon)=&-\cos^2\beta)(\alpha D_y)^2\Phi_\varepsilon+\sin(2\beta)\alpha D_y\partial_\beta\Phi_\varepsilon-\sin^2\beta)\partial_{\beta\beta}\Phi_\varepsilon-\\&-2(1+\cos^2\beta))\alpha D_y\Phi_\varepsilon+(\tan\beta+\sin(2\beta))\partial_\beta\Phi_\varepsilon+(\tan^2\beta)-3)\Phi_\varepsilon\\
%		=&\cos^2\beta)(\alpha D_y)^2\Phi_\varepsilon+\sin(2\beta)\alpha D_y\partial_\beta\Phi_\varepsilon-\sin^2\beta)\partial_{\beta\beta}\Phi_\varepsilon-\\&-2(1+\cos^2\beta))\alpha D_y\Phi_\varepsilon+\sin(2\beta)\partial_\beta\Phi_\varepsilon+\sin\beta\partial_\beta(\frac{\Phi_\varepsilon}{\cos\beta})-3\Phi_\varepsilon
%	\end{aligned}
%\end{equation*}
%???~$\ref{hardy2}$????????????~$|\partial_{\beta\beta}\Phi_\varepsilon|_{\mathcal{H}^k}$ ???????????~$U(\Phi_\varepsilon),V(\Phi_\varepsilon)$ ???????~$\ref{thm_biot}$ ????~$(\ref{decomposition})$??~$(\ref{estimate})$ ????
%\begin{equation}
%|\Lambda_i(\Phi_\varepsilon)|_{\mathcal{H}^k}\leq \frac{C}{\alpha}|\varepsilon|_{\mathcal{H}^k} \quad i=1,2,3,4
%\end{equation}
The estimates of the transport term also needs the following $L^\infty$ estimates (\cite{inequality})
%??????????????~$L^\infty$ ??~$(\cite{inequality})$
\begin{lemma}\label{lemma_inq}
	For $f(\bar\rho)\in C^\infty_0(0,\infty),g(\beta)\in C^\infty_0(0,\pi/2)$, we have that
	\begin{equation*}
		\begin{aligned}
			&|f|_{L^\infty}\leq C \int_0^\infty (|f|^2+|D_{\bar\rho} f|^2)\frac{d\bar\rho}{\bar\rho},\\
			&|g|_{L^\infty}\leq C \int_0^{\pi/2} (|g|^2+|D_{\beta} g|^2)\frac{d\beta}{\sin (2\beta)}.
		\end{aligned}
	\end{equation*}
\end{lemma}

\begin{proof}
	By changing of variable $\bar\rho=e^z,\tau=\ln|\csc(2\beta)-\cot(2\beta)|$, we have
	\begin{align*}
		&\frac{d\bar\rho}{\bar\rho}=dz,\quad D_{\bar\rho}=\partial_z,\\
		&\frac{2}{\sin(2\beta)}d\beta=d\tau,\quad D_\beta=2\partial_\tau,
	\end{align*}
	we obtain the inequalities by the Sobolev embedding.
\end{proof}
\begin{proof}[\textbf{proof of proposition of~$\ref{Mest2}$}]
	First, for $\mathcal{M}^{\varepsilon}_\xi$, we give an estimate under the norm of $\mathcal{W}^k_1$, we consider
	%????~$\mathcal{M}^{\varepsilon}_\xi$ ???????$\mathcal{W}^k_1$ ??????????
	\begin{equation}
		\left<\mathcal{M}^{\varepsilon}_\xi,\xi\right>_{\mathcal{W}^k_1}\!=\!\left<U(\Phi_\varepsilon\!)\partial_\beta\xi,\xi\right>_{\mathcal{W}^k_1}+
		\left<V(\Phi_\varepsilon\!)\bar{\rho}\partial_{\bar{\rho}}\xi,\xi\right>_{\mathcal{W}^k_1}+
		\left<\Lambda_1(\Phi_\varepsilon\!)\xi,\xi\right>_{\mathcal{W}^k_1}+
		\left<\Lambda_2(\Phi_\varepsilon\!)\phi,\xi\right>_{\mathcal{W}^k_1}.
	\end{equation}
	
	Set $D=D_\beta$ or~$D_{\bar \rho}$, for the first term~ $\left<U(\Phi_\varepsilon)\partial_\beta\xi,\xi\right>_{\mathcal{W}^k_1}$, we consider following integral
	%????????
	\begin{equation}\label{Uestimate}
		\begin{aligned}
			&\int D_{\bar \rho}^k\left(U(\Phi_\varepsilon)\partial_\beta\xi\right)\cdot D_{\bar\rho}^k\xi\sin^{2-\eta} (2\beta) \bar \rho^2 d\beta d\bar\rho\\
			&\int D^{k-1}D_\beta\left(U(\Phi_\varepsilon)\partial_\beta\xi\right)\cdot D^{k-1}D_\beta\xi\sin^{-\eta} (2\beta) \bar \rho^2 d\beta d\bar\rho.
		\end{aligned}	
	\end{equation}
	
	For the first term
	\begin{equation}\label{transest4}
		\begin{aligned}
			&\int D_{\bar \rho}^k\left(U(\Phi_\varepsilon)\partial_\beta\xi\right)\cdot D_{\bar\rho}^k\xi\sin^{2-\eta} (2\beta) \bar \rho^2 d\beta d\bar\rho\\
			=&\sum_{i=0}^k c_i \int D_{\bar \rho}^{i}\left(\frac{U(\Phi_\varepsilon)}{\sin (2\beta)}\right)D_{\bar\rho}^{k-i} D_\beta\xi\cdot D_{\bar\rho}^k\xi\sin^{2-\eta} (2\beta) \bar \rho^2 d\beta d\bar\rho.
		\end{aligned}
	\end{equation}
	We consider each term in above formula separately:
	
	\text{ Step 1:\quad $i=0$.}
	
	When $i=0$, note that $D_{\bar\rho}^{k} D_\beta\xi\cdot D_{\bar\rho}^k\xi=\frac{1}{2}D_\beta((D_{\bar\rho}^k\xi)^2)$,
	we use integration by parts
	%??$\beta$??????
	\begin{align*}
		&\int \left(\frac{U(\Phi_\varepsilon)}{\sin (2\beta)}\right) D_{\bar\rho}^{k} D_\beta\xi\cdot D_{\bar\rho}^k\xi\sin^{2-\eta} (2\beta) \bar \rho^2 d\beta d\bar\rho\\
		=&-\frac{1}{2}\int (D_{\bar\rho}^k\xi)^2 \partial_\beta\left(\frac{U(\Phi_\varepsilon)}{\sin (2\beta)}\cdot\sin^{3-\eta} (2\beta)\right) \bar\rho^2 d\beta d\bar\rho\\
		\leq& C \int (D_{\bar\rho}^k\xi)^2 \left(\left|D_\beta\frac{U(\Phi_\varepsilon)}{\sin (2\beta)}\right|+\left|\frac{U(\Phi_\varepsilon)}{\sin (2\beta)}\right|\right)\sin^{2-\eta} (2\beta) \bar\rho^2 d\beta d\bar\rho\\
		=&C\left(\left|D_\beta\left(\frac{U(\Phi_\varepsilon)}{\sin (2\beta)}\right)\right|_{L^\infty}+\left|\frac{U(\Phi_\varepsilon)}{\sin (2\beta)}\right|_{L^\infty}\right)\cdot|\xi|_{\mathcal{W}^k_1}^2\\
		\leq & \frac{C}{\alpha^{3/2}}|\varepsilon|_{\mathcal{H}^k}|\xi|_{\mathcal{W}^k_1}^2.
	\end{align*}

	\text{ Step 2:\quad $1\leq i\leq k-2$.}
	
	For $1\leq i\leq k-2$,
	
	%????????$(\ref{space})$,???~$D_{\bar \rho}^{i}\left(\frac{U(\Phi_\varepsilon)}{\sin (2\beta)}\right)$ ?~$L^\infty$???
	\begin{equation}\label{transest6}
		\begin{aligned}
			&\int D_{\bar \rho}^{i}\left(\frac{U(\Phi_\varepsilon)}{\sin (2\beta)}\right)D^{k-i} D_\beta\xi\cdot D_{\bar\rho}^k\xi\sin^{2-\eta} (2\beta) \bar \rho^2 d\beta d\bar\rho \\
			\leq&\left|D_{\bar \rho}^{i}\left(\frac{U(\Phi_\varepsilon)}{\sin (2\beta)}\right)\right|_{L^\infty_{\bar\rho}L^\infty_\beta}\int D^{k-i} D_\beta\xi \sin^{1-\eta/2} (2\beta) \bar\rho \cdot(D_{\bar\rho}^k\xi\sin^{1-\eta/2} (2\beta) \bar\rho) d\beta d\bar\rho\\
			\leq&\frac{C}{\alpha^{1/2}}\left|\left(\frac{U(\Phi_\varepsilon)}{\sin (2\beta)}\right)\right|_{\mathcal{H}^{i+2}}|\xi|^2_{\mathcal{W}^k_1}\\
			\leq&\frac{C}{\alpha^{3/2}}|\varepsilon|_{\mathcal{H}^{k}}|\xi|^2_{\mathcal{W}^k_1}
		\end{aligned}.
	\end{equation}

		\text{ Step 3:\quad $i=k-1$.}
		
	When~$i=k-1$, for $D_{\bar \rho}^{i}\left(\frac{U(\Phi_\varepsilon)}{\sin (2\beta)}\right)$ we consider $L_{\bar\rho}^2 L_{\beta}^\infty$ norm, by Cauchy-Schwarz inequality, we have
	%???$\bar\rho,\beta$???~Cauchy-Schwarz???????
	\begin{equation}\label{transestimate1}
		\begin{aligned}
			&\int D_{\bar \rho}^{i}\left(\frac{U(\Phi_\varepsilon)}{\sin (2\beta)}\right)D_{\bar{\rho}}^{k-i} D_\beta\xi\cdot D_{\bar\rho}^k\xi\bar\rho^2\sin^{2-\eta} (2\beta) d\beta d\bar\rho \\
			=&\int\left[ D_{\bar \rho}^{i}\left(\frac{U(\Phi_\varepsilon)}{\sin (2\beta)}\right)\bar\rho^{-1/2}\right](D_{\bar{\rho}}^{k-i} D_\beta\xi\bar{\rho}^{3/2}\sin^{1-\eta/2} (2\beta))\cdot D_{\bar\rho}^k\xi\bar\rho\sin^{1-\eta/2} (2\beta) d\beta d\bar\rho\\
			\leq&\int \left|D_{\bar \rho}^{i}\left(\frac{U(\Phi_\varepsilon)}{\sin (2\beta)}\right)\bar\rho^{-1/2}\right|_{L^\infty_\beta}\cdot|D_{\bar{\rho}}^{k-i} D_\beta\xi\bar\rho^{3/2}\sin^{1-\eta/2} (2\beta)|_{L^\infty_{\bar\rho}}\cdot (D_{\bar\rho}^k\xi\bar\rho\sin^{1-\eta/2} (2\beta) ) d\beta d\bar\rho\\
			\leq&\left|D_{\bar\rho}^{i}\left(\frac{U(\Phi_\varepsilon)}{\sin (2\beta)}\right)\bar\rho^{-1/2}\right|_{L^2_{\bar\rho} L^\infty_\beta}|D_{\bar{\rho}}^{k-i} D_\beta\xi\bar\rho^{3/2}\sin^{1-\eta/2} (2\beta) |_{L^\infty_\beta L^2_{\bar\rho}}\cdot |D_{\bar\rho}^k\xi\bar\rho\sin^{1-\eta/2} (2\beta)|_{L^2}.
		\end{aligned}
	\end{equation}
	%????????????~$L^2_{\bar\rho}$??????~$\mathcal {H}^k$ ?????
	%\begin{equation}\label{norm}
	%	\frac{(1+y)^4}{y^4}dy=\alpha\frac{(1+\rho^\alpha)^4}{\rho^{3\alpha}}
	%	\frac{d\rho}{\rho}\geq\alpha\frac{d\rho}{\rho}=\alpha\frac{d\bar\rho}{\bar\rho}
	%\end{equation}
	By ~(\ref{thm_biot}), we have
	%??,???~$\ref{thm_biot}$??????~$(\ref{decomposition})$?
	\begin{equation*}
		\begin{aligned}
			&\left|D_{\bar\rho}^{i}\left(\frac{U(\Phi_\varepsilon)}{\sin (2\beta)}\right)\bar\rho^{-1/2}\right|_{L^2_{\bar\rho} L^\infty_\beta}\leq\frac{1}{\sqrt{\alpha}}\left|D_{\bar\rho}^{i}\left(\frac{U(\Phi_\varepsilon)}{\sin (2\beta)}\right)\frac{(1+y)^2}{y^2}\right|_{L^2_{y} L^\infty_\beta}\\
			\leq& \frac{1}{\sqrt{\alpha}}\left|D_{\bar\rho}^{i}\left(\frac{U(\frac{1}{4\alpha}\sin(2\beta) L_{12}(\varepsilon))}{\sin (2\beta)}\right)\frac{(1+y)^2}{y^2}\right|_{L^2_{y} L^\infty_\beta}+\frac{1}{\sqrt{\alpha}}\left|D_{\bar\rho}^{i}\left(\frac{U(\bar{\Phi}_\varepsilon)}{\sin (2\beta)}\right)\frac{(1+y)^2}{y^2}\right|_{L^2_{y} L^\infty_\beta}\\
			\leq& \frac{C}{\alpha^{3/2}}|\varepsilon|_{\mathcal H^i}+\frac{C}{\alpha}\left|\frac{U(\bar{\Phi}_\varepsilon)}{\sin (2\beta)}\right|_{\mathcal H^{i+1}}\\
			\leq& \frac{C}{\alpha^{3/2}}|\varepsilon|_{\mathcal H^i}+\frac{C}{\alpha}|\varepsilon|_{\mathcal H^{i+1}}\leq \frac{C}{\alpha^{3/2}}|\varepsilon|_{\mathcal H^k}.
		\end{aligned}
	\end{equation*}
	As for the second term of $(\ref{transestimate1})$, by Lemma $\ref{lemma_inq}$, we have that
	%???~$(\ref{transestimate1})$??????$D^{k-i} D_\beta\xi$???????~$\beta$????~$\sin^{1-\eta/2}(2\beta) d\beta$????????~$\bar\rho$???????~$\ref{lemma_inq}$?
	%???? $(\ref{transest2})$?$D_\beta\xi$??$\bar\rho$???~$\bar\rho^{\frac{3}{2}}$,???
	\begin{equation}\label{transest3}
		\begin{aligned}
			&|D_{\bar\rho}^{k-i} D_\beta\xi\sin^{1-\eta/2}(2\beta)\bar\rho^{3/2}|^2_{L^2_\beta L^\infty_{\bar\rho}}\\
			= & \left|\int |D_{\bar\rho}^{k-i} D_\beta\xi|^2\sin^{2-\eta} (2\beta)\bar\rho^3 d\beta\right|_{L^\infty_{\bar\rho}} \\
			\leq & C\int (|D_{\bar\rho}D_{\bar\rho}^{k-i} D_\beta\xi|^2+|D_{\bar\rho}^{k-i} D_\beta\xi|^2)\sin^{2-\eta} (2\beta)\bar\rho^2 d\bar\rho d\beta \\
			\leq & C|\xi|_{\mathcal {W}^{k}_1}^2.
		\end{aligned}
	\end{equation}
	Then for $(\ref{transestimate1})$ we have
	\begin{equation*}
		\int D_{\bar \rho}^{i}\left(\frac{U(\Phi_\varepsilon)}{\sin (2\beta)}\right)D^{k-i} D_\beta\xi\cdot D_{\bar\rho}^k\xi\sin^{2-\eta} (2\beta) \bar\rho^2 d\beta d\bar\rho 	\leq\frac{C}{\alpha^{3/2}}|\varepsilon|_{\mathcal{H}^k}|\xi|_{\mathcal {W}^k_1}^2.
	\end{equation*}	

	\text{ Step 1:\quad $i=k$.}

	when $i=k$, we take $L^2$ norm for $D^k U(\Phi_\varepsilon)$
	%?~$U(\Phi_\varepsilon)$ ??k??????????~$L^2$ ????,??~$D_\beta \xi$????~$\beta,\bar\rho$??~$L^\infty$????~$\bar\rho$ ???????????
	\begin{equation}\label{transest2}
		\begin{aligned}
			&\int D_{\bar \rho}^{k}\left(\frac{U(\Phi_\varepsilon)}{\sin (2\beta)}\right) D_\beta\xi\cdot D_{\bar\rho}^k\xi\sin^{2-\eta} (2\beta) \bar\rho^2 d\beta d\bar\rho \\
			=&\int D_{\bar \rho}^{k}\left(\frac{U(\Phi_\varepsilon)}{\sin (2\beta)}\right)\bar\rho^{\frac{1}{2}}\sin^{-\eta/2}(2\beta)\cdot D_\beta\xi\bar\rho^{\frac{3}{2}}\sin (2\beta)\cdot (D_{\bar\rho}^k\xi\bar\rho\sin^{1-\eta/2} (2\beta) ) d\beta d\bar\rho\\
			\leq&\left|D_{\bar \rho}^{k}\left(\frac{U(\Phi_\varepsilon)}{\sin (2\beta)}\right)\bar\rho^{\frac{1}{2}}\sin^{-\eta/2}(2\beta) \right|_{L^2}\left|D_\beta\xi\bar\rho^{\frac{3}{2}}\sin (2\beta)\right|_{L^\infty_{\bar\rho} L^\infty_\beta}\cdot |D_{\bar\rho}^k\xi\bar\rho\sin^{1-\eta/2} (2\beta) |_{L^2}.
		\end{aligned}
	\end{equation}
	Similar to $(\ref{transest3})$, for $U(\Phi_\varepsilon)$ term we have
	%??~$(\ref{transest3})$??, ?????~$U(\Phi_\varepsilon)$??
	\begin{equation*}
		\begin{aligned}
			\left|D_{\bar \rho}^{k}\left(\frac{U(\Phi_\varepsilon)}{\sin (2\beta)}\right)\bar\rho^{\frac{-1}{2}}\sin^{-\eta/2}(2\beta) \right|_{L^2}^2=&\int \left| D_{\bar \rho}^{k}\left(\frac{U(\Phi_\varepsilon)}{\sin (2\beta)}\right)\right|^2 \sin^{-\eta}(2\beta)\frac{d\bar\rho}{\bar\rho}\\
			\leq &\frac{C}{\alpha}\int \left| D_{\bar \rho}^{k}\left(\frac{U(\Phi_\varepsilon)}{\sin (2\beta)}\right)\right|^2 \sin^{-\eta}(2\beta)\frac{(1+y)^4}{y^4}dy\\
			\leq& \frac{C}{\alpha^3}|\varepsilon|_{\mathcal H^k}^2.
		\end{aligned}
	\end{equation*}
	For $D_\beta\xi$ term in $(\ref{transest2})$, by Lemma $\ref{lemma_inq}$, we have
	
	%?? $(\ref{transest2})$?$D_\beta\xi$???$\bar\rho$ ??????~$(\ref{transest3})$ ????~$\beta$ ????????~$\ref{lemma_inq}$, ????~$\beta$???????
	\begin{equation*}
		\begin{aligned}
			\left|D_\beta\xi\sin(2\beta)\right|_{L^\infty_{\beta}}^2\leq &C\int (|D_\beta^2\xi|^2+|D_{\beta}\xi|^2)\sin (2\beta) d\beta\\
			\leq &C\int (|D_\beta^2\xi|^2+|D_\beta\xi|^2)\sin^{-\eta}(2\beta) d \beta.
		\end{aligned}
	\end{equation*}
	%?~$\mathcal{W}^k_1$??? $\beta$????????~$\left|D_\beta\xi\bar\rho^{\frac{3}{2}}\sin (2\beta)\right|_{L^\infty_{\bar\rho} L^\infty_\beta}\leq C|\xi|_{\mathcal W^2_1}$ \\
	%???????????
	Next
	\begin{equation}
		\begin{aligned}
			&\int D_{\bar \rho}^k\left(U(\Phi_\varepsilon)\partial_\beta\xi\right)\cdot D_{\bar\rho}^k\xi\rho^2\sin^{2-\eta} (2\beta) \bar d\beta d\bar\rho\\
			=&\sum_{i=0}^k\int D_{\bar \rho}^{i}\left(\frac{U(\Phi_\varepsilon)}{\sin (2\beta)}\right)D^{k-i} D_\beta\xi\cdot D_{\bar\rho}^k\xi\rho^2\sin^{2-\eta} (2\beta) \bar d\beta d\bar\rho\\
			\leq &\frac{C}{\alpha^{3/2}}|\varepsilon|_{\mathcal H^k}|\xi|_{\mathcal W^k_1}^2.
		\end{aligned}
	\end{equation}
	
	Then we consider the second term in~$(\ref{Uestimate})$
	\begin{equation}
		\begin{aligned}
			&\int D^{k-1}D_\beta\left(\frac{U(\Phi_\varepsilon)}{\sin (2\beta)}D_\beta\xi\right)\cdot D^{k-1}D_\beta\xi\sin^{-\eta} (2\beta) \bar \rho^2 d\beta d\bar\rho\\
			=&\sum_{i=0}^{k-1}c_i\int D^{i}D_\beta\left(\frac{U(\Phi_\varepsilon)}{\sin (2\beta)}\right)\cdot D^{k-i-1}D_\beta\xi\cdot D^{k-1}D_\beta\xi\sin^{-\eta} (2\beta) \bar \rho^2 d\beta d\bar\rho+\\
			&+\sum_{i=0}^{k-1}c_i\int D^{i}\left(\frac{U(\Phi_\varepsilon)}{\sin (2\beta)}\right)\cdot D^{k-1-i}D^2_\beta\xi\cdot D^{k-1}D_\beta\xi \sin^{-\eta} (2\beta) \bar \rho^2 d\beta d\bar\rho\\
			:=& \sum_{i=0}^{k-1}c_i A_i+\sum_{i=0}^{k-1}c_i B_i.
		\end{aligned}
	\end{equation}
	%???~$i\leq k-1$,~$B_i$???~$\frac{U(\Phi_\varepsilon)}{\sin(2\beta)}$???~$k-1$????????????~$(\ref{transest4})$???????
	%??????~$\beta$???~$-\eta$???~$A_i$????~$\frac{U(\Phi_\varepsilon)}{\sin(2\beta)}$?k??????~$\frac{U(\Phi_\varepsilon)}{\sin(2\beta)}$????~$\mathcal H^k$ ???????~$\beta$??????~$D_\beta U(\Phi_\varepsilon)$ ????~$\beta$??~$\sin^{-\gamma}(2\beta)$,??~$\gamma=1+\alpha/10>1$??????~$A_i$ ????
	%\begin{small}
	%\begin{equation}
	%	\begin{aligned}
	%  	&\int  D^{i}D_\beta\left(\frac{U(\Phi_\varepsilon)}{\sin (2\beta)}\right)\cdot D^{k-i-1}D_\beta\xi\cdot D^{k-1}D_\beta\xi\sin^{-\eta} (2\beta) \bar \rho^2 d\beta d\bar\rho\\
	%  	\leq&\int  \left|D^{i}D_\beta\left(\frac{U(\Phi_\varepsilon)}{\sin (2\beta)}\right)\sin^{-\gamma/2}(2\beta)\right|\cdot| D^{k-i-1}D_\beta\xi\sin^{\gamma/2-\eta/2}(2\beta)|\cdot |D^{k-1}D_\beta\xi\sin^{-\eta/2} (2\beta)| \bar \rho^2 d\beta d\bar\rho
	%	\end{aligned}
	%\end{equation}
	%\end{small}
	
	%???$i=k-1,k-2$????$(\ref{transestimate1}),(\ref{transest2})$??
		\text{ Step 2:\quad $i=k-1,k-2$.}
		
	For $i=k-1,k-2$, we have
	\begin{equation}
		\begin{aligned}
			&\int D^{k-2}D_\beta\left(\frac{U(\Phi_\varepsilon)}{\sin (2\beta)}\right)\cdot DD_\beta\xi\cdot D^{k-1}D_\beta\xi\sin^{-\eta} (2\beta) \bar \rho^2 d\beta d\bar\rho\\
			=&\int D^{k-2}D_\beta\left(\frac{U(\Phi_\varepsilon)}{\sin (2\beta)}\right)\sin^{-\gamma/2}(2\beta)\cdot DD_\beta\xi\bar\rho\sin^{\gamma/2-\eta/2} (2\beta)\cdot D^{k-1}D_\beta\xi\bar \rho\sin^{-\eta/2} (2\beta) d\beta d\bar\rho\\
			\leq& |D^{k-2}D_\beta\left(\frac{U(\Phi_\varepsilon)}{\sin (2\beta)}\right)\sin^{-\gamma/2}(2\beta)|_{L^\infty_{\bar\rho}L^2_\beta}\cdot| (DD_\beta\xi)\bar\rho\sin^{\gamma/2-\eta/2}(2\beta)|_{L^2_{\bar\rho}L^\infty_\beta}\cdot \\
			&\cdot|(D^{k-1}D_\beta\xi)\bar \rho\sin^{-\eta/2} (2\beta) |_{L^2_{\bar\rho}L^2_\beta}\\
			\leq&\frac{C}{\alpha^{3/2}}|\varepsilon|_{\mathcal H^k}|\xi|_{\mathcal{W}_1^k}^2,
		\end{aligned}
	\end{equation}
	\begin{equation}
		\begin{aligned}
			&\int D^{k-1}D_\beta\left(\frac{U(\Phi_\varepsilon)}{\sin (2\beta)}\right)\cdot D_\beta\xi\cdot D^{k-1}D_\beta\xi\sin^{-\eta} (2\beta) \bar \rho^2 d\beta d\bar\rho\\
			=&\int  D^{k-1}D_\beta\left(\frac{U(\Phi_\varepsilon)}{\sin (2\beta)}\right)\bar\rho^{\frac{-1}{2}}\sin^{-\gamma/2} (2\beta)\cdot D_\beta\xi\bar\rho^{\frac{3}{2}}\sin^{\gamma/2-\eta/2} (2\beta)\cdot D^{k-1}D_\beta\xi\sin^{-\eta/2} (2\beta) \bar \rho d\beta d\bar\rho\\
			\leq&|D^{k-1}D_\beta\left(\frac{U(\Phi_\varepsilon)}{\sin (2\beta)}\right)\bar\rho^{-1/2}\sin^{-\gamma/2}(2\beta)|_{L^2_{\bar\rho}L^2_\beta}\cdot| (D_\beta\xi)\bar\rho^{3/2}\sin^{\gamma/2-\eta/2}(2\beta)|_{L^\infty_{\bar\rho}L^\infty_\beta}\cdot\\
			&\cdot |(D^{k-1}D_\beta\xi)\bar \rho\sin^{-\eta/2} (2\beta) |_{L^2_{\bar\rho}L^2_\beta}\\
			\leq&\frac{C}{\alpha^{3/2}}|\varepsilon|_{\mathcal H^k}|\xi|_{\mathcal{W}_1^k}^2.\\
		\end{aligned}
	\end{equation}
	
			\text{ Step 3:\quad $0\leq i\leq k-3$.}
	
	As for $i<k-2$, by the same estimate $(\ref{transest6})$, we have $L^\infty$ estimate of $D^i U(\Phi_\varepsilon)$
	%??~$i<k-2$?????????~$(\ref{transest6})$??????~$D^i U(\Phi_\varepsilon)$ ?$L^\infty$ ??
	\begin{equation*}
		\left|D^{i}\left(\frac{ U(\Phi_\varepsilon)}{\sin (2\beta)}\right)\right|_{L^\infty}\leq C \left|\frac{ U(\Phi_\varepsilon)}{\sin (2\beta)}\right|_{\mathcal H^{i+2}}.
	\end{equation*}
	%?~$\mathcal W^k_1$???~$\beta$ ?????
	
	We obtain estimate of $(\ref{Uestimate})$:
	%????????~$(\ref{Uestimate})$???
	\begin{equation}\label{U_estimate}
		\left<U(\Phi_\varepsilon)\partial_\beta\xi,\xi\right>_{\mathcal{W}^k_1}\leq \frac{C}{\alpha^{3/2}}|\varepsilon|_{\mathcal H^k}|\xi|_{\mathcal{W}_1^k}^2.
	\end{equation}
	The case of $V(\Phi_\varepsilon)\alpha D_{\bar y}\xi,\Lambda_1(\Phi_\varepsilon)\xi$ is similar, for the forth term:
	%??~$(\ref{trans_M2})$???????$V(\Phi_\varepsilon)\alpha D_{\bar y}\xi,\Lambda_1(\Phi_\varepsilon)\xi$ ????????????????~$\phi$ ?????????
	\begin{equation}\label{lambda2}
		\begin{aligned}
			&\int D^k_{\bar{\rho}}(\Lambda_2(\Phi_\varepsilon)\phi)D^k_{\bar{\rho}}\xi\bar\rho^2\sin^{2-\eta}(2\beta) d\bar\rho d\beta+\\
			&+\int D^{k-1}D_\beta(\Lambda_2(\Phi_\varepsilon)\phi)D^{k-1}D_\beta\xi\bar\rho^2\sin^{-\eta}(2\beta) d\bar\rho d\beta \\
			=&\sum_{i=0}^{k}b_i\int  D^{i}_{\bar{\rho}}(\Lambda_2(\Phi_\varepsilon)\phi)D^{k-i}_{\bar{\rho}}\xi\bar\rho^2\sin^{2-\eta}(2\beta) d\bar\rho d\beta+ \\
			+&\sum_{i=0}^{k-1}c_i\int  D^{i}D_\beta\left(\Lambda_2(\Phi_\varepsilon)\right)\cdot D^{k-i-1}\phi\cdot D^{k-1}D_\beta\xi\sin^{-\eta} (2\beta) \bar \rho^2 d\beta d\bar\rho+\\
			+&\sum_{i=0}^{k-1}c_i\int  D^{i}\left(\Lambda_2(\Phi_\varepsilon)\right)\cdot D^{k-1-i}D_\beta\phi\cdot D^{k-1}D_\beta\xi \sin^{-\eta} (2\beta) \bar \rho^2 d\beta d\bar\rho\\
			:=&\sum_{i=0}^{k}b_i A_i+\sum_{i=0}^{k}c_iB_i+\sum_{i=0}^{k}c_iC_i,
		\end{aligned}
	\end{equation}
	where $b_i,c_i$ is the binomial coefficient.
	
	%???~$\mathcal W^k_1,\mathcal W^k_3$ ?????~$(\ref{normdef})$,???????????~$\beta$ ???????????~$\beta$ ????????~$(\ref{lambda2})$???~$B_i,C_i$??~$\beta$ ???????\\
	For $B_i$, we have
	\begin{equation}\label{lambda3}
		\begin{aligned}
			&\int  D^{i}D_\beta\left(\Lambda_2(\Phi_\varepsilon)\right)\cdot D^{k-i-1}\phi\cdot D^{k-1}D_\beta\xi\sin^{-\eta} (2\beta) \bar \rho^2 d\beta d\bar\rho\\
			=&\int  D^{i}D_\beta\left(\Lambda_2(\Phi_\varepsilon)\right)\sin^{-\gamma/2}(2\beta)\cdot (D^{k-i-1}\phi)\bar\rho\sin^{\gamma/2-\eta/2}(2\beta)\cdot D^{k-1}D_\beta\xi\bar\rho\sin^{-\eta} (2\beta)   d\beta d\bar\rho.
		\end{aligned}
	\end{equation}
	By Lemma $\ref{lemma_inq}$, we have

	\begin{equation}
		\begin{split}
			&|(D^{k-i-1}\phi)\bar\rho\sin^{\gamma/2-\eta/2}(2\beta)|_{L^\infty_\beta}^2\\
			\leq&C\int |\partial_\beta D^{k-i-1}\phi|^2\bar\rho^2\sin^{\gamma-\eta+1}(2\beta) d\beta+C\int | D^{k-i-1}\phi|^2\bar\rho^2\sin^{\gamma-\eta-1}(2\beta) d\beta\\
			\leq&C\int |\partial_\beta D^{k-i-1}\phi|^2\bar\rho^2\sin^{2-\eta}(2\beta) d\beta+C\int |\partial_\beta D^{k-i-1}\phi|^2\bar\rho^2\sin^{\gamma-\eta+1}(2\beta) d\beta\\
			\leq&C\int\int |\partial_\beta D^{k-i-1}\phi|^2\bar\rho^2\sin^{2-\eta}(2\beta) d\beta\\
			\leq&C\int\int |\partial_\beta D^{k-i-1}\phi|^2\bar\rho^2\cos^{2-\eta}\beta d\beta.
		\end{split}
	\end{equation}
	%??????~$\mathcal W_3^k$???~$\beta$?????~$C_i$???~$i\leq k-1$,?~$\Lambda_2$??~$L^\infty_\beta L^2_{\bar\rho}$????? ????????
	We rewrite as
	\begin{equation}
		\begin{aligned}
			&\int  D^{i}\left(\Lambda_2(\Phi_\varepsilon)\right)\cdot D^{k-1-i}D_\beta\phi\cdot D^{k-1}D_\beta\xi \sin^{-\eta} (2\beta) \bar \rho^2 d\beta d\bar\rho\\
			=&\int  D^{i}\left(\Lambda_2(\Phi_\varepsilon)\right)\bar\rho^{-1/2}\cdot D^{k-1-i}\partial_\beta\phi\bar\rho^{3/2}\sin^{-\eta/2+1}(2\beta)\cdot D^{k-1}D_\beta\xi \sin^{-\eta/2} (2\beta) \bar \rho d\beta d\bar\rho.
		\end{aligned}
	\end{equation}
	%??~$\partial_\beta\phi$?????~$\sin^{1-\eta/2}(2\beta)$??~$\mathcal W^k_3$?????~$\cos^{1-\eta/2}\beta$?????
	
			\text{ Step 1:\quad $i=k-1,k-2$.}
	
	for $i=k-2,k-1$, we have estimate of~$A_{k-1},A_{k},B_{k-1},B_{k}$:
	\begin{equation*}
		\begin{aligned}
			&\int  D^{k-2}D_\beta \Lambda_2(\Phi_\varepsilon) \cdot D\phi\cdot D^{k-1}D_\beta\xi\sin^{-\eta} (2\beta) \bar \rho^2 d\beta d\bar\rho\\
			\leq&  |D^{k-2}D_\beta\Lambda_2(\Phi_\varepsilon)\sin^{-\gamma/2}(2\beta)|_{L^\infty_{\bar\rho}L^2_\beta}\cdot| D\phi\bar\rho\sin^{\gamma/2-\eta/2}(2\beta)|_{L^2_{\bar\rho}L^\infty_\beta}\cdot \\
			&\cdot|D^{k-1}D_\beta\xi\bar \rho\sin^{-\eta/2} (2\beta) |_{L^2_{\bar\rho}L^2_\beta}\\
			\leq&\frac{C}{\alpha^{3/2}}|\varepsilon|_{\mathcal H^k}|\xi|_{\mathcal{W}_1^k}|\phi|_{\mathcal{W}_3^k}.\\
		\end{aligned}
	\end{equation*}
	\begin{equation*}
		\begin{aligned}
			&\int  D^{k-2} \Lambda_2(\Phi_\varepsilon) \cdot D_\beta
			D\phi\cdot D^{k-1}D_\beta\xi\sin^{-\eta} (2\beta) \bar \rho^2 d\beta d\bar\rho\\
			\leq&  |D^{k-2}D_\beta\Lambda_2(\Phi_\varepsilon)\sin^{-\gamma/2}(2\beta)|_{L^\infty_{\bar\rho}L^2_\beta}\cdot| DD_\beta\phi\bar\rho\sin^{\gamma/2-\eta/2}(2\beta)|_{L^2_{\bar\rho}L^\infty_\beta}\cdot \\
			&\cdot|D^{k-1}D_\beta\xi\bar \rho\sin^{-\eta/2} (2\beta) |_{L^2_{\bar\rho}L^2_\beta}\\
			\leq&\frac{C}{\alpha^{3/2}}|\varepsilon|_{\mathcal H^k}|\xi|_{\mathcal{W}_1^k}|\phi|_{\mathcal{W}_3^k}.
		\end{aligned}
	\end{equation*}
	\begin{equation*}
		\begin{aligned}
			&\int  D^{k-1}D_\beta\Lambda_2(\Phi_\varepsilon)\cdot \phi\cdot D^{k-1}D_\beta\xi\sin^{-\eta} (2\beta) \bar \rho^2 d\beta d\bar\rho\\
			\leq&|D^{k-1}D_\beta\Lambda_2(\Phi_\varepsilon)\bar{\rho}^{-1/2}\sin^{-\gamma/2}(2\beta)|_{L^2_{\bar\rho}L^2_\beta}\cdot| \phi\bar\rho^{3/2}\sin^{\gamma/2-\eta/2}(2\beta)|_{L^2_{\bar\rho}L^2_\beta}\cdot\\
			&\cdot |D^{k-1}D_\beta\xi\bar \rho\sin^{-\eta/2} (2\beta) |_{L^2_{\bar\rho}L^2_\beta}\\
			\leq&\frac{C}{\alpha^{3/2}}|\varepsilon|_{\mathcal H^k}|\xi|_{\mathcal{W}_1^k}|\phi|_{\mathcal{W}_3^k}.\\
			&\\
			&\int  D^{k-1}\Lambda_2(\Phi_\varepsilon)\cdot D_\beta\phi\cdot D^{k-1}D_\beta\xi\sin^{-\eta} (2\beta) \bar \rho^2 d\beta d\bar\rho\\
			\leq&|D^{k-1}\Lambda_2(\Phi_\varepsilon)\sin^{-\gamma/2}(2\beta)|_{L^\infty_{\bar\rho}L^2_\beta}\cdot| D_\beta\phi\bar\rho\sin^{\gamma/2-\eta/2}(2\beta)|_{L^2_{\bar\rho}L^\infty_\beta}\cdot\\
			&\cdot |D^{k-1}D_\beta\xi\bar \rho\sin^{-\eta/2} (2\beta) |_{L^2_{\bar\rho}L^2_\beta}\\
			\leq&\frac{C}{\alpha^{3/2}}|\varepsilon|_{\mathcal H^k}|\xi|_{\mathcal{W}_1^k}|\phi|_{\mathcal{W}_3^k}.		
		\end{aligned}
	\end{equation*}

			\text{ Step 2:\quad $0\leq i\leq k-3$.}

	As for~$i< k-2$, by embedding $(\ref{space})$ we have $L^\infty$ estimate of~$D^{i}\Lambda_2(\Phi_\varepsilon)$.
	\begin{equation*}
		|D^{i}\Lambda_2(\Phi_\varepsilon)|_{L^\infty}\leq \frac{C}{\alpha^{1/2}} |\Lambda(\Phi_\varepsilon)|_{\mathcal H^{i+2}}.
	\end{equation*}
	For $B_i,C_i$, we have:
	\begin{equation}
		\begin{aligned}
			&\int  D^{i}D_\beta \Lambda_2(\Phi_\varepsilon) \cdot D^{k-i-1}\phi\cdot D^{k-1}D_\beta\xi\sin^{-\eta} (2\beta) \bar \rho^2 d\beta d\bar\rho\\
			\leq&  |D^{i}D_\beta\Lambda_2(\Phi_\varepsilon)\sin^{-\gamma/2}(2\beta)|_{L^\infty_{\bar\rho}L^2_\beta}\cdot| D^{k-i-1}\phi\bar\rho\sin^{\gamma/2-\eta/2}(2\beta)|_{L^2_{\bar\rho}L^\infty_\beta}\cdot \\
			&\cdot|D^{k-1}D_\beta\xi\bar \rho\sin^{-\eta/2} (2\beta) |_{L^2_{\bar\rho}L^2_\beta}\\
			\leq&\frac{C}{\alpha^{3/2}}|\varepsilon|_{\mathcal H^k}|\xi|_{\mathcal{W}_1^k}|\phi|_{\mathcal{W}_3^k}.\\
		\end{aligned}
	\end{equation}
	
	\begin{equation}
		\begin{aligned}
			&\int  D^{i} \Lambda_2(\Phi_\varepsilon) \cdot D_\beta
			D^{k-i-1}\phi\cdot D^{k-1}D_\beta\xi\sin^{-\eta} (2\beta) \bar \rho^2 d\beta d\bar\rho\\
			\leq&  |D^{i}\Lambda_2(\Phi_\varepsilon)|_{L^\infty_{\bar\rho}\!L^\infty_\beta}\!\cdot|\!\partial_\beta D^{k\!-\!i\!-\!1}\phi\bar\rho\sin^{1-\eta/2}(2\beta)|_{L^2_{\bar\rho}\!L^2_\beta}\!\cdot\!|D^{k\!-\!1}D_\beta\xi\bar \rho\sin^{-\eta/2} (2\beta) |_{L^2_{\bar\rho}\!L^2_\beta}\\
			\leq&\frac{C}{\alpha^{3/2}}|\varepsilon|_{\mathcal H^k}|\xi|_{\mathcal{W}_1^k}|\phi|_{\mathcal{W}_3^k},\\
		\end{aligned}
	\end{equation}
	then we obtain that
	\begin{equation*}
		\left<\Lambda_2(\Phi_\varepsilon)\phi,\xi\right>_{\mathcal{W}^k_1}\leq\frac{C}{\alpha^{3/2}}|\varepsilon|_{\mathcal H^k}|\xi|_{\mathcal{W}_1^k}|\phi|_{\mathcal{W}_3^k},
	\end{equation*}
	combined with $(\ref{U_estimate})$, we have
	\begin{equation}\label{transestfinal1}
		\left<\mathcal{M}^{\varepsilon}_\xi,\xi\right>_{\mathcal{W}^k_1}\leq \frac{C}{\alpha^{3/2}}|\varepsilon|_{\mathcal H^k}|\xi|_{\mathcal{W}_1^k}(|\phi|_{\mathcal{W}_3^k}+|\xi|_{\mathcal{W}_1^k}).
	\end{equation}

	Next we consider estimate of~$\left<\mathcal{M}^{\varepsilon}_\xi\xi,\xi\right>_{\mathcal{W}^k_2}$, note that the difference between~ $\mathcal{W}^k_1,\mathcal{W}^k_2 $ is the order of~$\bar\rho$, we have
	%??~$\left<\mathcal{M}^{\varepsilon}_\xi,\xi\right>_{\mathcal{W}^k_2}$???~$\Lambda_2(\Phi_\varepsilon)\phi$ ?????~$\mathcal{W}^k_1$?????????????
	\begin{align*}
		&\left<U(\Phi_\varepsilon)\partial_\beta\xi,\xi\right>_{\mathcal{W}^k_2}+
		\left<V(\Phi_\varepsilon)\bar{y}\partial_{\bar{y}}\xi,\xi\right>_{\mathcal{W}^k_2}+
		\left<\Lambda_1(\Phi_\varepsilon)\xi,\xi\right>_{\mathcal{W}^k_2}\\
		\leq& \frac{C}{\alpha^{3/2}}|\varepsilon|_{\mathcal H^k}|\xi|_{\mathcal{W}_2^k}^2,
	\end{align*}
	%????$(\Lambda_2(\Phi_\varepsilon)\phi,\xi)_{\mathcal{W}_2^k}$ ????~$\bar\rho$ ?????
	next
	\begin{equation}
		\begin{aligned}
			&\int  D^{k-1}D_\beta( \Lambda_2(\Phi_\varepsilon)\phi)\cdot D^{k-1}D_\beta\xi\sin^{-\eta} (2\beta) \bar \rho^\eta d\beta d\bar\rho\\
			=&\int  D^{k-1}D_\beta( \Lambda_2(\Phi_\varepsilon)\phi)\bar \rho^{\eta/2}\sin^{-\eta/2} (2\beta)\cdot D^{k-1}D_\beta\xi\bar \rho^{\eta/2}\sin^{-\eta/2} (2\beta)  d\beta d\bar\rho,
		\end{aligned}
	\end{equation}
	where the order of $\bar\rho$ in the norm of~$\mathcal W_3^k$ is~$\bar\rho^2$,
	%~$\mathcal W_2^k$ ????????~$\bar\rho^{\eta}$?$0<\eta<1$
	for all~$\varepsilon_0>0$, there exists~$C(\varepsilon_0)$ such that
	
	\begin{equation}
		\begin{aligned}
			&\int |D^k\phi|^2\bar{\rho}^{\eta}\cos^{2-\eta}\beta  d\beta d\bar\rho\\
			\leq&\varepsilon_0\int|D^k\phi|^2\bar{\rho}^{2}\cos^{2-\eta}\beta  d\beta d\bar\rho+C(\varepsilon_0)\int|D^k\phi|^2\cos^{2-\eta}\beta  d\beta d\bar\rho.
		\end{aligned}	
	\end{equation}

	%??????~$\mathcal W_3^k$??????????~$\bar\rho$???~0??~Laplace ??$\bar\rho$??????????~Laplace ????
	%%??%
	
	Next is the estimate of~$\left<\mathcal M^\varepsilon_\phi(\xi,\phi),\phi\right>_{\mathcal W_3^k}$, recall the definition of~$\mathcal W_3^k$~$(\ref{normdef})$,
	the difficulty is the terms containing $\partial_\beta \phi$, More precisely we consider following terms:
	%????~$\beta$ ??????????~$D^{k-1}\partial_\beta\mathcal M^\varepsilon_\phi(\xi,\phi) $ ???????????????????????
	\begin{equation}
		\begin{aligned}
			D^{k-1}\partial_\beta M^\varepsilon_\phi(\xi,\phi)=&D^{k-1}\partial_\beta(U(\Phi_\varepsilon)\partial_\beta\phi)+D^{k-1}\partial_\beta(V(\Phi_\varepsilon)D_{\bar{y}})
			+
			\\	&D^{k-1}\partial_\beta(\Lambda_3(\Phi_\varepsilon)\xi)+
			D^{k-1}\partial_\beta(\Lambda_4(\Phi_\varepsilon)\phi)\\
			:=&I_1+I_2+I_3+I_4.
		\end{aligned}
	\end{equation}
	
				\text{ Step 1:\quad Estimate of $I_1$.}
	
	For~$I_1$, we rewrite as
	\begin{equation*}
		\begin{aligned}
			&D^{k-1}\partial_\beta(U(\Phi_\varepsilon)\partial_\beta\phi)\\
			=&D^{k-1}\partial_\beta(\frac{U(\Phi_\varepsilon)}{\sin2 \beta}\cdot\sin(2\beta)\cdot\partial_\beta\phi)\\
			=&\sum_{i=0}^{k-1}c_i D^iD_\beta (\frac{U(\Phi_\varepsilon)}{\sin2 \beta})D^{k-i-1}\partial_\beta\phi+\sum_{i=0}^{k-1}c_i D^i (\frac{U(\Phi_\varepsilon)}{\sin2 \beta})D^{k-i-1}D_\beta\partial_\beta\phi,
		\end{aligned}
	\end{equation*}
	where $c_i$ is a binary coefficient, and we have:
	%??????~$\frac{U(\Phi_\varepsilon)}{\sin2 \beta}$ ??~$(\ref{UV})$ ??????????~$(\ref{Uestimate})$ ??????????
	\begin{equation}\label{phi1}
		\begin{aligned}
			\int D^{k-1}\partial_\beta(U(\Phi_\varepsilon)\partial_\beta\phi)D^{k-1}\partial_\beta\phi\bar\rho^2\cos^{2-\eta}\beta d\beta d\bar{\rho}\leq& \frac{ C}{\alpha^{1/2}} \left |\frac{U(\Phi_\varepsilon)}{\sin2 \beta}\right|_{\mathcal H^k}|\phi|_{\mathcal W^k_3}^2\\
			\leq&\frac{C}{\alpha^{3/2}}|\varepsilon|_{\mathcal H^k}|\phi|_{\mathcal W^k_3}^2.
		\end{aligned}
	\end{equation}

				\text{ Step 2:\quad Estimate of $I_2$.}
	
	For~$I_2$,
	\begin{equation*}
		\begin{aligned}
			&D^{k-1}\partial_\beta(V(\Phi_\varepsilon)D_{\bar\rho}\phi)\\
			=&\sum_{i=0}^{k-1}c_i D^i V(\Phi_\varepsilon)\cdot D^{k-i-1}  D_{\bar\rho}\partial_\beta\phi+\sum_{i=0}^{k-1}c_i D^i \partial_\beta V(\Phi_\varepsilon)D^{k-i-1}D_{\bar\rho}\phi\\
			:=&\sum_{i=0}^{k-1}c_i A_i+\sum_{i=0}^{k-1}c_i B_i.
		\end{aligned}
	\end{equation*}
	%???~$(\ref{V})$???~$\partial_\beta V(\Phi_\varepsilon)$?~$\mathcal{H}^{k}$????,?~$i\leq k-2$???~$(\ref{space})$???~$D^iV(\Phi_\varepsilon),D^i\partial_\beta V(\Phi_\varepsilon)$?~$L^\infty$ ??
	By $\mathcal{H}^{k}$ estimate of $\partial_\beta V(\Phi_\varepsilon)$, when $i\leq k-2$, we have
	\begin{equation}
		\begin{aligned}
			|D^iV(\Phi_\varepsilon)|_{L^\infty_\beta L^\infty_{\bar\rho}}+|D^i\partial_\beta V(\Phi_\varepsilon)|_{L^\infty_\beta L^\infty_{\bar\rho}}
			\leq &\frac{C}{\alpha^{1/2}}(|D^{i}V(\Phi_\varepsilon)|_{\mathcal{H}^2}+|D^{i}\partial_\beta V(\Phi_\varepsilon)|_{\mathcal{H}^2})\\
			\leq &\frac{C}{\alpha^{1/2}}(|V(\Phi_\varepsilon)|_{\mathcal{H}^{i+2}}+|\partial_\beta V(\Phi_\varepsilon)|_{\mathcal{H}^{i+2}})\\
			\leq&\frac{C}{\alpha^{1/2}}(|V(\Phi_\varepsilon)|_{\mathcal{H}^{k}}+|\partial_\beta V(\Phi_\varepsilon)|_{\mathcal{H}^{k}})\\
			\leq &\frac{C}{\alpha^{3/2}}|\varepsilon|_{\mathcal{H}^{k}}.
		\end{aligned}
	\end{equation}
	Then for integral of~$A_i$ and $B_i$, by Cauchy-Schwarz inequality
	\begin{equation}\label{Vest1}
		\begin{aligned}
			&\int A_i\cdot D^{k-1}\partial_\beta\phi\bar\rho^2\cos^{2-\eta}\beta d\bar\rho d\beta+\int B_i\cdot D^{k-1}\partial_\beta\phi\bar\rho^2\cos^{2-\eta}\beta d\bar\rho d\beta\\
			\leq& (|D^iV(\Phi_\varepsilon)|_{L^\infty_\beta L^\infty_{\bar\rho}}+|D^i\partial_\beta V(\Phi_\varepsilon)|_{L^\infty_\beta L^\infty_{\bar\rho}})\cdot\int (|D^{k-i-1}  D_{\bar\rho}\partial_\beta\phi|+\\
			&+|D^{k-i-1}D_{\bar\rho}\phi|)\cdot D^{k-1}\partial_\beta\phi\bar\rho^2\cos^{2-\eta}\beta d\bar\rho d\beta\\
			\leq& \frac{C}{\alpha^{3/2}}|\varepsilon|_{\mathcal{H}^{k}}|\phi|^2_{\mathcal{W}_3^k}.
		\end{aligned}
	\end{equation}
	When~$i=k-1$, we have
	%?~$A_{k-1}$??~$V(\Phi_\varepsilon)$ ??~k-1?????~$D^{k-1}V(\Phi_\varepsilon)$ ???????~$L^2_{\beta\rho}L^\infty_\beta$??? ??~$\beta$ ???~1 ??~Sobolev??,?~$\bar\rho$????~$(\ref{space})$ ??~$L^\infty$ ??
	\begin{equation}
		\begin{aligned}
			A_{k-1}\leq& |D^{k-1}V(\Phi_\varepsilon)|_{L^\infty_\beta L^\infty_{\bar\rho}}|\partial_\beta D_{\bar\rho}\phi|\\
			\leq &(|\partial_\beta D^{k-1}V(\Phi_\varepsilon)|_{L^2_\beta L^\infty_{\bar\rho}}+|D^{k-1}V(\Phi_\varepsilon)|_{L^2_\beta L^\infty_{\bar\rho}})|\partial_\beta D_{\bar\rho}\phi|\\
			\leq &C(|\partial_\beta V(\Phi_\varepsilon)|_{\mathcal{H}^k}+|V(\Phi_\varepsilon)|_{\mathcal{H}^k})|\partial_\beta D_{\bar\rho}\phi|\\
			\leq &\frac{C}{\alpha}|\varepsilon|_{\mathcal{H}^k}|\partial_\beta D_{\bar\rho}\phi|.
		\end{aligned}
	\end{equation}
	It implies
	\begin{equation}\label{Vest2}
		\begin{aligned}
			&\int A_{k-1}\cdot D^{k-1}\partial_\beta\phi\bar\rho^2\cos^{2-\eta}\beta d\bar\rho d\beta\\
			=&\int D^{k-1}V(\Phi_\varepsilon)\bar\rho^{-1/2}\cdot D_{\bar\rho}\partial_\beta\phi\bar\rho^{3/2}\cos^{1-\eta/2}\beta\cdot D^{k-1}\partial_\beta\phi\bar\rho\cos^{1-\eta/2}\beta d\bar\rho d\beta\\
			\leq& |D^{k-1}V(\Phi_\varepsilon)\bar\rho^{-1/2}|_{L^2_{\bar\rho}L^\infty_\beta}\cdot |D_{\bar\rho}\partial_\beta\phi\bar\rho^{3/2}\cos^{1-\eta/2}\beta|_{L^2_\beta L^\infty_{\bar\rho}}\cdot |D^{k-1}\partial_\beta\phi\bar\rho\cos^{1-\eta/2}\beta|_{L^2_\beta L^2_{\bar\rho}}\\
			\leq&\frac{C}{\alpha^{3/2}}|\varepsilon|_{\mathcal{H}^k}|\phi|^2_{\mathcal{W}^k_3}.
		\end{aligned}
	\end{equation}
	%??~$B_{k-1}$?,??~$\partial_\beta V$ ?~k-1??????????????~$(\ref{transest2})$???,??~$L_{\beta}^\infty L_{\bar\rho}^2$ ????????~Cauchy-Schwarz ???
	As for $B_{k-1}$, by Cauchy-Schwarz inequality we have
	\begin{equation}\label{Vest3}
		\begin{aligned}
			&\int B_{k-1}\cdot D^{k-1}\partial_\beta\phi\bar\rho^2\cos^{2-\eta}\beta d\bar\rho d\beta\\
			=&\int D^{k-1}\partial_\beta V(\Phi_\varepsilon)D_{\bar\rho}\phi\cdot D^{k-1}\partial_\beta\phi\bar\rho^2\cos^{2-\eta}\beta d\bar\rho d\beta\\
			=& \int D^{k-1}\partial_\beta V(\Phi_\varepsilon)\bar\rho^{-1/2}\cdot D_{\bar\rho}\phi\bar\rho^{3/2}\cos^{1-\eta/2}\beta\cdot D^{k-1}\partial_\beta\phi\bar\rho\cos^{1-\eta/2}\beta d\bar\rho d\beta\\
			\leq &|D^{k-1}\partial_\beta V(\Phi_\varepsilon)\bar\rho^{-1/2}|_{L^\infty_\beta L^2_{\bar\rho}}|D_{\bar\rho}\phi\bar\rho^{3/2}\cos^{1-\eta/2}\beta|_{L^2_{\beta}L^\infty_{\bar\rho}}|\phi|_{\mathcal{W}^k_3}\\
			\leq &\frac{C}{\alpha^{1/2}}|D_\beta D^{k-1}\partial_\beta V(\Phi_\varepsilon)\bar\rho^{-1/2}\sin^{-\gamma/2}(2\beta)|_{L^2_{\beta}L^2_{\bar\rho}}|D_{\bar\rho}\phi\bar\rho\cos^{1-\eta/2}\beta|_{L^2_{\beta}L^2_{\bar\rho}}|\phi|_{\mathcal{W}^k_3}\\
			\leq &\frac{C}{\alpha^{1/2}}|\partial_\beta V(\Phi_\varepsilon)|_{\mathcal{H}^k}|\phi|_{\mathcal{W}^k_3}^2\\
			\leq&\frac{C}{\alpha^{3/2}}|\varepsilon|_{\mathcal{H}^k}|\phi|_{\mathcal{W}^k_3}^2,
		\end{aligned}	
	\end{equation}
	combined $(\ref{Vest1},\ref{Vest2},\ref{Vest3})$, we obtain
	\begin{equation}\label{phiestV}
		\left<V(\Phi_\varepsilon)D_{\bar\rho}\phi,\phi\right>_{\mathcal W_{3}^k}\leq\frac{C}{\alpha^{3/2}}|\varepsilon|_{\mathcal{H}^k}|\phi|_{\mathcal{W}^k_3}^2.
	\end{equation}
	
					\text{ Step 3:\quad Estimate of $I_3$.}
	
	For~$\Lambda_3(\Phi_\varepsilon)$ we need consider integral
	\begin{equation}\label{lambda xi}
		\begin{aligned}
			&\int D^{k-1}\partial_\beta(\Lambda_3(\Phi_\varepsilon)\xi)\cdot D^{k-1}\partial_\beta\phi\cdot\bar\rho^2\cos^{2-\eta}\beta d\beta d\bar{\rho}\\
			=&\sum_{i=0}^{k-1}c_i\int D^{i}\Lambda_3(\Phi_\varepsilon)D^{k-i-1}\partial_\beta\xi\cdot D^{k-1}\partial_\beta\phi\cdot\bar\rho^2\cos^{2-\eta}\beta d\beta d\bar{\rho}\\
			&+\sum_{i=0}^{k-1}c_i\int D^{i}\partial_\beta\Lambda_3(\Phi_\varepsilon)D^{k-i-1}\xi\cdot D^{k-1}\partial_\beta\phi\cdot\bar\rho^2\cos^{2-\eta}\beta d\beta d\bar{\rho}.
		\end{aligned}	
	\end{equation}
	We need a more precise estimate of $\Lambda_3(\Phi_\varepsilon),\partial_\beta\Lambda_3(\Phi_\varepsilon)$,
	%????~$\Lambda_3(\Phi_\varepsilon),\partial_\beta\Lambda_3(\Phi_\varepsilon)$??????,?
	\begin{equation}
		%	F_1=&\cos\beta\partial_{\beta\beta\beta}\Phi_\varepsilon \quad \quad F_2=\partial_\beta\Lambda_3(\Phi_\varepsilon)-\cos\beta F_1\\
		%	=&\frac{1}{2\sin\beta}D_\beta\partial_{\beta\beta}\Phi_\varepsilon\\
		G_1=(\alpha D_{ y})^2\Phi_\varepsilon, \quad\quad G_2=\Lambda_3(\Phi_\varepsilon)-sin^2\beta G_1.
	\end{equation}
	
	By~$(\ref{UV})$
	\begin{equation}
		\begin{aligned}
			\Lambda_3(\Phi_\varepsilon)=&sin^2\beta (\alpha D_{ y})^2\Phi_\varepsilon+\cos^2\beta\partial_{\beta\beta}\Phi_\varepsilon+\sin(2\beta) \alpha D_{ y}\partial_{\beta}\Phi_\varepsilon+l.o.t\\
			:=&\sin^2\beta G_1+G_2,
			%		\partial_\beta\Lambda_3(\Phi_\varepsilon)=&\cos^2\beta)\partial_{\beta\beta\beta}\Phi_\varepsilon+sin^2\beta) (\alpha D_{ y})^2\partial_\beta\Phi_\varepsilon+\sin(2\beta) \alpha D_{ y}\partial_{\beta\beta}\Phi_\varepsilon+l.o.t\\
			%		:=&\cos\beta F_1+F_2
		\end{aligned}
	\end{equation}
	where~l.o.t=lower order term.
	
By~$(\ref{thm_biot})$
		\begin{equation*}
		%	|F_1\sin\beta|_{\mathcal{H}^{k-1}}=&|\cos\beta D_\beta\partial_{\beta\beta}\Phi_\varepsilon|_{\mathcal{H}^{k-1}}\leq C|\partial_{\beta\beta}\Phi_\varepsilon|_{\mathcal{H}^{k}}\leq \frac{C}{\alpha}|\varepsilon|_{\mathcal{H}^k}\\
		| G_1|_{\mathcal{H}^{k}}= |(\alpha D_{ y})^2\Phi_\varepsilon |_{\mathcal{H}^{k}}\leq\frac{C}{\alpha}|\varepsilon|_{\mathcal{H}^k},
	\end{equation*}
	
	%??~$G_2$,%?????~$G_2$????????~$\cos\beta$??,
	we have
	\begin{align*}
		%	&D^{k-1}F_2=sin^2\beta) D^{k-1}(\alpha D_{ y})^2\partial_\beta\Phi_\varepsilon+\sin(2\beta) \alpha D^{k-1} D_{ y}\partial_{\beta\beta}\Phi_\varepsilon+l.o.t\\
		&D^{k-1}(\frac{G_2}{\cos\beta})=2\sin\beta D^{k-1}\alpha D_{ y}\partial_{\beta}\Phi_\varepsilon+\cos\beta D^{k-1}\partial_{\beta\beta}\Phi_\varepsilon+\\
		&+D^{k-1}(\frac{1+2\sin^2\beta)}{\cos\beta}\alpha D_{ y}\Phi_\varepsilon)+l.o.t.
	\end{align*}
	%???k???k-1%
	By~$(\ref{thm_biot})$, %???~$\partial_{\beta\beta}\Phi_\varepsilon$?????~$\frac{G_2}{\cos\beta}$??????
	we have
	\begin{equation*}
		%	&|F_2|_{\mathcal{H}^{k-1}}\leq C|\alpha D_{ y}\partial_\beta\Phi_\varepsilon|_{\mathcal{H}^k}+|\partial_{\beta\beta}\Phi_\varepsilon|_{\mathcal{H}^k}\leq \frac{C}{\alpha}|\varepsilon|_{\mathcal{H}{k}}\\
		|\frac{G_2}{\cos\beta}|_{\mathcal{H}^{k-1}}\leq C |\alpha D_{ y}\partial_\beta\Phi_\varepsilon|_{\mathcal{H}^k}+|\partial_{\beta\beta}\Phi_\varepsilon|_{\mathcal{H}^k}\leq \frac{C}{\alpha}|\varepsilon|_{\mathcal{H}{k}}.
	\end{equation*}

	For~$(\ref{lambda xi})$, we rewrite as
	%?????????????????????,?????~$\Lambda_3(\Phi_\varepsilon)$,???~$G_1,G_2$???
	\begin{align}\label{G_1G_2}
		&\int D^{i}\Lambda_3(\Phi_\varepsilon)D^{k-i-1}\partial_\beta\xi\cdot D^{k-1}\partial_\beta\phi\cdot\bar\rho^2\cos^{2-\eta}\beta d\beta d\bar{\rho}\\
		=&\int D^{i}(\sin^2\beta) G_1+G_2)D^{k-i-1}\partial_\beta\xi\cdot D^{k-1}\partial_\beta\phi\cdot\bar\rho^2\cos^{2-\eta}\beta d\beta d\bar{\rho},
	\end{align}
	the difficulty is the integral of $\beta$,
	%???????~$\frac{1}{\sin(2\beta)}$?????????~$G_1$?????????
	for~$i\leq k-2$, by Cauchy-Schwarz inequality
	\begin{equation}\label{G_1}
		\begin{aligned}
			&\int D^{i}(\sin^2\beta) G_1)D^{k-i-1}\partial_\beta\xi\cdot D^{k-1}\partial_\beta\phi\bar\rho^2\cos^{2-\eta}\beta d\beta d\bar\rho\\
			=&\int( D^{i} G_1\cdot (D^{k-i-1}D_\beta\xi)\bar\rho\sin\beta\cos^{-\eta/2}\beta)\cdot D^{k-1}\partial_\beta\phi\bar\rho\cos^{1-\eta/2}\beta d\beta\\
			\leq&| D^{i} G_1|_{L^\infty}|(D^{k-i-1}D_\beta\xi)\bar\rho\sin\beta\cos^{-\eta/2}\beta|_{L^2}|D^{k-1}\partial_\beta\phi\bar\rho\cos^{1-\eta/2}\beta|_{L^2}\\
			\leq &\frac{C}{\alpha^{1/2}}|G_1|_{\mathcal{H}^{i+2}}|\xi|_{\mathcal{W}_1^k}|\phi|_{\mathcal{W}_3^k}\\
			\leq &\frac{C}{\alpha^{3/2}}|\varepsilon|_{\mathcal{H}^{k}}|\xi|_{\mathcal{W}_1^k}|\phi|_{\mathcal{W}_3^k}.
		\end{aligned}
	\end{equation}
	For $i=k-1$, %?~$D^{k-1}G_1$?~$D_\beta\xi$?~Cauchy-Schwarz???????\\~$L^2_{\beta}L^\infty_{\bar{\rho}},L^\infty_{\beta}L^2_{\bar{\rho}}$??,???~$(\ref{norm})$???~$(\ref{lemma_inq})$???
	\begin{align*}
		&\int (D^{k-1} G_1\bar\rho^{-1/2}\cdot D_\beta\xi\bar\rho^{3/2}\sin\beta\cos^{-\eta/2}\beta)\cdot D^{k-1}\partial_\beta\phi\bar\rho\cos^{1-\eta/2}\beta d\beta\\
		\leq&|D^{k-1} G_1\frac{1}{\bar\rho^{-1/2}}|_{L^\infty_\beta L^2_{\bar\rho}}|D_\beta\xi\bar\rho^{3/2}\sin\beta\cos^{-\eta/2}\beta|_{L^2_\beta L^\infty_{\bar{\rho}}}|D^{k-1}\partial_\beta\phi\bar\rho\cos^{1-\eta/2}\beta|_{L^2}\\
		\leq &\frac{C}{\alpha^{1/2}}|G_1|_{\mathcal{H}^{k}}|\xi|_{\mathcal{W}_1^k}|\phi|_{\mathcal{W}_3^k}\\
		\leq &\frac{C}{\alpha^{3/2}}|\varepsilon|_{\mathcal{H}^{k}}|\xi|_{\mathcal{W}_1^k}|\phi|_{\mathcal{W}_3^k},
	\end{align*}
	%????~$G_1$???????
	we obtain
	\begin{align}\label{G_1est}
		&\int D^{i}(\sin^2\beta) G_1)D^{k-i-1}\partial_\beta\xi\cdot D^{k-1}\partial_\beta\phi\cdot\bar\rho^2\cos^{2-\eta}\beta d\beta d\bar{\rho}\leq \frac{C}{\alpha^{3/2}}|\varepsilon|_{\mathcal{H}^{k}}|\xi|_{\mathcal{W}_1^k}|\phi|_{\mathcal{W}_3^k}.
	\end{align}
	For all $0\leq i\leq k-1$.
	%???~$0\leq i\leq k-1$???
	Next, consider integral containing $G_2$ in~$(\ref{G_1G_2})$,
	%??~$(\ref{G_1G_2})$???~$G_2$????????????
	\begin{equation}\label{G_2}
		\begin{aligned}
			&\int D^{i}G_2 D^{k-i-1}\partial_\beta\xi\cdot D^{k-1}\partial_\beta\phi\cdot\bar\rho^2\cos^{2-\eta}\beta d\beta d\bar{\rho}\\
			=&\int D^{i}(\frac{G_2}{\cos\beta}\cos\beta)D^{k-i-1}\partial_\beta\xi\cdot D^{k-1}\partial_\beta\phi\cdot\bar\rho^2\cos^{2-\eta}\beta d\beta d\bar{\rho}\\
			=&\int D^{i}(\frac{G_2}{\cos\beta})D^{k-i-1}\partial_\beta\xi\bar\rho\cos^{2-\eta/2}\cdot D^{k-1}\partial_\beta\phi\cdot\bar\rho\cos^{1-\eta/2}\beta d\beta d\bar{\rho}\\
			=&\int D^{i}(\frac{G_2}{\cos\beta})D^{k-i-1}(\cos\beta\partial_\beta\xi)\bar\rho\cos^{1-\eta/2}\cdot D^{k-1}\partial_\beta\phi\cdot\bar\rho\cos^{1-\eta/2}\beta d\beta d\bar{\rho}.\\
		\end{aligned}
	\end{equation}
	%???,~$D=D_\beta$?$D_{\bar{y}}$?~$\cos\beta$??????????,?????????

	%??????????~$D^{k-i-1}\partial_\beta\xi$,????????????~$(\ref{exchangeorder})$????
	By compatibility condition $(\ref{exchangeorder})$
	\begin{equation}\label{inq}
		|\cos\beta\partial_\beta\xi|\leq   |\sin\beta D_{\bar\rho}\xi|+|\cos\beta D_{\bar\rho}\phi|+|\sin\beta\partial_\beta\phi|.
	\end{equation}
	Combined with~$(\ref{G_2})$
	\begin{align*}
		&\int D^{i}(\frac{G_2}{\cos\beta})D^{k-i-1}(\cos\beta\partial_\beta\xi)\bar\rho\cos^{1-\eta/2}\cdot D^{k-1}\partial_\beta\phi\cdot\bar\rho\cos^{1-\eta/2}\beta d\beta d\bar{\rho}\\
		\leq& \int |D^{i}(\frac{G_2}{\cos\beta})|\Big(|D^{k-i-1}(\sin\beta D_{\bar\rho}\xi)|+|D^{k-i-1}(\cos\beta D_{\bar\rho}\phi)|+\\
		&+|D^{k-i-1}(\sin\beta\partial_\beta\phi)|\Big)\bar\rho\cos^{1-\eta/2}\beta\cdot |D^{k-1}\partial_\beta\phi\cdot\bar\rho\cos^{1-\eta/2}\beta| d\beta d\bar{\rho}\\
		\leq& \int |D^{i}(\frac{G_2}{\cos\beta})|\Big(|D^{k-i-1} D_{\bar\rho}\xi\sin\beta\cos^{1-\eta/2}\beta|+|D^{k-i-1} D_{\bar\rho}\phi\cos^{2-\eta/2}\beta|+\\
		&+|D^{k-i-1}\partial_\beta\phi\sin\beta\cos^{1-\eta/2}\beta|\Big)\bar\rho\cdot |D^{k-1}\partial_\beta\phi\cdot\bar\rho\cos^{1-\eta/2}\beta| d\beta d\bar{\rho},
	\end{align*}
	%?????~$(\ref{inq})$???????~$D_{\bar\rho}\xi,D_{\bar\rho}\phi,\partial_\beta\phi$,?????????~$\beta$?????????,??~$\beta\in(0,\pi/2)$
	we have
	\begin{align*}
		&|D^{k-i-1} D_{\bar\rho}\xi\sin\beta\cos^{1-\eta/2}\beta|\leq |D^{k-i-1} D_{\bar\rho}\xi\sin^{1-\eta/2}(2\beta)|\\
		&|D^{k-i-1} D_{\bar\rho}\phi\cos^{2-\eta/2}\beta|\leq |D^{k-i-1} D_{\bar\rho}\phi\cos^{1-\eta/2}\beta|\\
		&|D^{k-i-1}\partial_\beta\phi\sin\beta\cos^{1-\eta/2}\beta| \leq |D^{k-i-1}\partial_\beta\phi\cos^{1-\eta/2}\beta|.
	\end{align*}
	%?????~$\mathcal{W}^k_1,\mathcal{W}^k_3$???~$\beta$???????~$G_1$?~$(\ref{G_1})$ ???????????
	It implies
	\begin{equation}
		\begin{aligned}
			&\int D^{i}(\frac{G_2}{\cos\beta})D^{k-i-1}(\cos\beta\partial_\beta\xi)\bar\rho\cos^{1-\eta/2}\cdot D^{k-1}\partial_\beta\phi\cdot\bar\rho\cos^{1-\eta/2}\beta d\beta d\bar{\rho}\\
			\leq & \frac{C}{\alpha^{3/2}}|\varepsilon|_{\mathcal{H}^{k}}|\xi|_{\mathcal{W}_1^k}|\phi|_{\mathcal{W}_3^k}.
		\end{aligned}
	\end{equation}
	Combined with $(\ref{G_1est})$, we obtain
	%?????~$(\ref{G_1est})$??????
	\begin{equation}\label{Lambda3est1}
		\int D^{i}\Lambda_3(\Phi_\varepsilon)D^{k-i-1}\partial_\beta\xi\cdot D^{k-1}\partial_\beta\phi\cdot\bar\rho^2\cos^{2-\eta}\beta d\beta d\bar{\rho}\leq  \frac{C}{\alpha^{3/2}}|\varepsilon|_{\mathcal{H}^{k}}|\xi|_{\mathcal{W}_1^k}|\phi|_{\mathcal{W}_3^k}.
	\end{equation}
	For the second term of~$(\ref{lambda xi})$, we have
	\begin{equation*}
		\begin{aligned}
			&\int D^{i}\partial_\beta\Lambda_3(\Phi_\varepsilon)D^{k-i-1}\xi\cdot D^{k-1}\partial_\beta\phi\cdot\bar\rho^2\cos^{2-\eta}\beta d\beta d\bar{\rho}\\
			=&\int\!D^{i}\!D_\beta\!(\sin^2\!\beta G_1\!+\!G_2)\sin^{-\gamma/2}\!(2\beta) D^{k-i-1}\xi\sin^{\gamma/2-1}\!(2\beta)\!\cdot\!D^{k-1}\partial_\beta\phi\!\cdot\!\bar\rho^2\!\cos^{2-\eta}\!\beta d\beta d\bar{\rho}\\
			:=&J_1+J_2.
		\end{aligned}
	\end{equation*}
	We rewrite $J_1,J_2$ as following
	\begin{small}
		\begin{equation*}
			\begin{aligned}
				J_1=&\int D^{i}D_\beta(\sin^2\beta) G_1)\sin^{-\gamma/2}(2\beta) \cdot D^{k-i-1}\xi\sin^{\gamma/2-1}(2\beta)\cdot D^{k-1}\partial_\beta\phi\cdot\bar\rho^2\cos^{2-\eta}\beta d\beta d\bar{\rho}\\
				=&\int D^{i}D_\beta G_1\sin^{-\gamma/2}(2\beta) \cdot(D^{k-i-1}\xi\bar\rho\sin^{\gamma/2-1}(2\beta)\sin^2\beta)\cos^{1-\eta/2}\beta)\cdot D^{k-1}\partial_\beta\phi\cdot\bar\rho\cos^{1-\eta/2}\beta d\beta d\bar{\rho},\\
				&\\
				J_2=&\int D^{i}D_\beta(\frac{ G_2}{\cos\beta}\cos\beta)\sin^{-\gamma/2}(2\beta) \cdot D^{k-i-1}\xi\sin^{\gamma/2-1}(2\beta)\cdot D^{k-1}\partial_\beta\phi\cdot\bar\rho^2\cos^{2-\eta}\beta d\beta d\bar{\rho}\\
				=&\int D^{i}D_\beta(\frac{ G_2}{\cos\beta})\sin^{-\gamma/2}(2\beta) \cdot D^{k-i-1}\xi\bar\rho\cos^{2-\eta/2}\beta\sin^{\gamma/2-1}(2\beta)\cdot D^{k-1}\partial_\beta\phi\cdot\bar\rho^2\cos^{1-\eta/2}\beta d\beta d\bar{\rho}.
			\end{aligned}
		\end{equation*}
	\end{small}
	The difficulty is the order of $\beta$, by Lemma $\ref{lemma_inq}$, we have
	%???????~$\beta$??????????????????~$\beta$ ???????~$\ref{lemma_inq}$??????~$D^{k-i-1}\xi$???~$\beta$?~$L^\infty_\beta$ ??????
	\begin{equation*}
		\begin{aligned} &|D^{k-i-1}\xi\sin^{\gamma/2-1}(2\beta)\sin^2\beta)\cos^{1-\eta/2}\beta|_{L^\infty_\beta}^2=|D^{k-i-1}\xi\sin^{\gamma/2+1}\beta\cos^{(\gamma-\eta)/2}\beta|_{L^\infty_\beta}^2\\
			\leq&C\int (D_\beta D^{k-i-1}\xi)^2\sin^{\gamma+1}\beta\cos^{\gamma-\eta-1}\beta d\beta\\
			\leq& C\int (D_\beta D^{k-i-1}\xi)^2\sin^{-\eta}\beta\cos^{-\eta}\beta d\beta\\		&|D^{k-i-1}\xi\cos^{2-\eta/2}\beta\sin^{\gamma/2-1}(2\beta)|_{L^\infty_\beta}^2=|D^{k-i-1}\xi\cos^{\gamma/2-\eta/2+1}\beta\sin^{\gamma/2-1}\beta|_{L^\infty_\beta}^2\\
			\leq &C\int (D_\beta D^{k-i-1}\xi)^2\sin^{\gamma-3}\beta\cos^{\gamma-\eta+1}\beta d\beta.
		\end{aligned}
	\end{equation*}
	Where the first term controlled by~$\mathcal{W}^k_1$ norm, as for second term, by compatibility condition~$(\ref{exchangeorder})$, we have%for~$D_\beta\xi$?
	\begin{small}
		\begin{align*}
			&\int  (D^{k-i-1}(D_\beta\xi))^2\sin^{\gamma-3}\beta\cos^{\gamma-\eta+1}\beta d\beta\\
			\leq&\int (D^{k-i-1}(\sin\beta(|\sin\beta D_{\bar{\rho}}\xi|+|\cos\beta D_{\bar\rho}\phi|+|\sin\beta \partial_\beta\phi|)))^2\sin^{\gamma-3}\beta\cos^{\gamma-\eta+1}\beta d\beta\\
			\leq &\int (D^{k-i-1}D_{\bar{\rho}}\xi)^2\sin^{\gamma+1}\beta\cos^{\gamma-\eta+1}\beta d\beta+\int (D^{k-i-1} D_{\bar{\rho}}\phi)^2\sin^{\gamma-1}\beta\cos^{\gamma-\eta+3}\beta d\beta+\\
			+&\int (D^{k-i-1} \partial_\beta\phi)^2\sin^{\gamma+1}\beta\cos^{\gamma-\eta+1}\beta d\beta\\
			\leq& \int (D^{k-i-1}D_{\bar{\rho}}\xi)^2\sin^{2-\eta}(2\beta) d\beta+\int (D^{k-i-1}D_{\bar{\rho}}\phi)^2\cos^{2-\eta}\beta d\beta+\int (D^{k-i-1}\partial_\beta\phi)^2\cos^{2-\eta}\beta d\beta.\\
		\end{align*}
	\end{small}	
	%??????~$\beta$????~$\mathcal{W}^k_1.\mathcal{W}^k_3$ ???~$\beta$?????~$i\leq k-2$ ????
	For $i\leq k-2$, we have
	\begin{small}
		\begin{equation*}
			\begin{aligned}
				J_1\!=\!&\int D^{i}\!D_\beta G_1\!\sin^{-\gamma/2}(2\beta)\!\cdot\!(D^{k\!-\!i\!-\!1}\xi\bar\rho\sin^{\gamma/2\!-\!1}\!(2\beta)\sin^2\!\beta\cos^{1\!-\!\eta/2}\!\beta)\!\cdot\! D^{k\!-\!1}\partial_\beta\phi\cdot\bar\rho\cos^{1\!-\!\eta/2}\!\beta d\beta d\bar{\rho}\\
				\leq&|D^{i}D_\beta G_1\sin^{-\gamma/2}(2\beta)|_{L_{\bar\rho}^\infty L^2_{\beta}}\cdot|(D^{k-i-1}\xi\bar\rho\sin^{\gamma/2-1}(2\beta)\sin^2\beta)\cos^{1-\eta/2}\beta)|_{L_{\bar\rho}^2 L^\infty_{\beta}}|\phi|_{\mathcal{W}^k_3}\\
				\leq& \frac{C}{\alpha^{1/2}} |G_1|_{\mathcal{H}^{i+2}}|\xi|_{\mathcal{W}^{k-i}}|\phi|_{\mathcal{W}^k_3}\\
				\leq&\frac{C}{\alpha^{3/2}}|\varepsilon|_{\mathcal{H}^{k}}|\xi|_{\mathcal{W}^{k}}|\phi|_{\mathcal{W}^k_3},
			\end{aligned}
		\end{equation*}
	\end{small}
	\begin{small}
		\begin{equation*}
			\begin{aligned}
				J_2\!=\!&\int\!D^{i}\!D_\beta(\frac{ G_2}{\cos\beta}\!\cos\beta)\!\sin^{-\!\gamma\!/\!2}(2\beta)\!\cdot\!(D^{k\!-\!i\!-\!1}\!\xi\bar\rho\cos^{1\!-\!\eta/2}\!\beta\sin^{\gamma/2\!-\!1}\!(2\beta))\!\cdot\! D^{k\!-\!1}\partial_\beta\phi\!\cdot\!\bar\rho^2\!\cos^{1\!-\!\eta/2}\!\beta d\beta d\bar{\rho}\\
				\leq&|D^{i}D_\beta(\frac{ G_2}{\cos\beta})\sin^{-\gamma/2}(2\beta)|_{L_{\bar\rho}^\infty L^2_{\beta}} \cdot |D^{k-i-1}\xi\bar\rho\cos^{2-\eta/2}\beta\sin^{\gamma/2-1}(2\beta)|_{L_{\bar\rho}^2 L^\infty_{\beta}}\cdot |\phi|_{\mathcal{W}^k_3}\\
				\leq& \frac{C}{\alpha^{1/2}}|\frac{ G_2}{\cos\beta}|_{\mathcal{H}^{i+2}}(|\xi|_{\mathcal{W}_1^{k-i}}+|\phi|_{\mathcal{W}_3^{k-i}})\cdot |\phi|_{\mathcal{W}^k_3}\\
				\leq& \frac{C}{\alpha^{3/2}}|\varepsilon|_{\mathcal{H}^{k}}(|\xi|_{\mathcal{W}_1^{k}}+|\phi|_{\mathcal{W}_3^{k}})\cdot |\phi|_{\mathcal{W}^k_3}.
			\end{aligned}
		\end{equation*}
	\end{small}
	For $i=k-1$, we have
	%?~$D^{k-i-1}\xi$ ??~$\bar{\rho}$??~$L^\infty$ ?????~$(\ref{transest2})$??? ????????
	\begin{equation*}
		J_2\leq\frac{C}{\alpha^{3/2}}|\varepsilon|_{\mathcal{H}^{k}}(|\xi|_{\mathcal{W}_1^{k}}+|\phi|_{\mathcal{W}_3^{k}})\cdot |\phi|_{\mathcal{W}^k_3}
	\end{equation*}
	Combined with $(\ref{Lambda3est1})$, we have
	\begin{equation*}
		\int D^{k-1}\partial_\beta(\Lambda_3(\Phi_\varepsilon)\xi)\cdot D^{k-1}\partial_\beta\phi\cdot\bar\rho^2\cos^{2-\eta}\beta d\beta d\bar{\rho}\leq \frac{C}{\alpha^{3/2}}|\varepsilon|_{\mathcal{H}^{k}}(|\xi|_{\mathcal{W}_1^{k}}+|\phi|_{\mathcal{W}_3^{k}})\cdot |\phi|_{\mathcal{W}^k_3}.
	\end{equation*}
	It implies
	\begin{equation}\label{lambda3est2}
		\left<\Lambda_3(\Phi_\varepsilon)\xi,\phi\right>_{\mathcal{W}^k_3}\leq \frac{C}{\alpha}|\varepsilon|_{\mathcal{H}^{k}}(|\xi|_{\mathcal{W}_1^{k}}+|\phi|_{\mathcal{W}_3^{k}})\cdot |\phi|_{\mathcal{W}^k_3}.
	\end{equation}
	
\text{ Step 4:\quad Estimate of $I_4$.}
	
	For~$\Lambda_4(\Phi_\varepsilon)$, we consider
	\begin{equation}\label{lambda4}
		\begin{aligned}
			&\int D^{k-1}\partial_\beta(\Lambda_4(\Phi_\varepsilon)\phi)\cdot D^{k-1}\partial_\beta\phi\cdot\bar\rho^2\cos^{2-\eta}\beta d\beta d\bar{\rho}\\
			=&\sum_{i=0}^{k-1}c_i\int D^{i}\Lambda_4(\Phi_\varepsilon)D^{k-i-1}\partial_\beta\phi\cdot D^{k-1}\partial_\beta\phi\cdot\bar\rho^2\cos^{2-\eta}\beta d\beta d\bar{\rho}+\\
			&+\sum_{i=0}^{k-1}c_i\int D^{i}\partial_\beta\Lambda_4(\Phi_\varepsilon)D^{k-i-1}\phi\cdot D^{k-1}\partial_\beta\phi\cdot\bar\rho^2\cos^{2-\eta}\beta d\beta d\bar{\rho}.
		\end{aligned}	
	\end{equation}
	The first term controlled by $\mathcal{W}_3^k$ norm, the difference is the second term
	%??????~$D^{k-i-1}\partial_\beta\phi$?~$\mathcal{W}_3^k$???? ?????~$\beta$ ????????????~$\partial_\beta$ ??~$\Lambda_4(\Phi_\varepsilon)$??????~$\partial_\beta\Lambda_4(\Phi_\varepsilon)$ ???
	\begin{align*}
		\partial_\beta\Lambda_4(\Phi_\varepsilon)=&-\sin\beta\cos\beta D_{\bar\rho}^2\partial_\beta\Phi_\varepsilon-(\cos^2\beta)-\sin^2\beta))D_{\bar{\rho}}\partial_{\beta\beta}\Phi_\varepsilon+\\
		&+\sin\beta\cos\beta\partial_{\beta\beta\beta}\Phi_\varepsilon-D_{\bar{\rho}}\partial_\beta(\tan\beta\Phi_\varepsilon)+l.o.t\\
		=&D_{\bar\rho}\left(-\sin\beta\cos\beta D_{\bar\rho}\partial_\beta\Phi_\varepsilon-(\cos^2\beta)-\sin^2\beta))\partial_{\beta\beta}\Phi_\varepsilon -\partial_\beta(\tan\beta\Phi_\varepsilon)\right)+\\
		&+\frac{1}{2}D_\beta\partial_{\beta\beta}\Phi_\varepsilon+l.o.t.
	\end{align*}
	By~$(\ref{thm_biot})$
	\begin{equation*}
		|\partial_\beta\Lambda_4(\Phi_\varepsilon)|_{\mathcal{H}^{k-1}}\leq \frac{C}{\alpha}|\varepsilon|_{\mathcal{H}^k},
	\end{equation*}
	it implies
	%??~$\partial_\beta\Lambda_4(\Phi_\varepsilon)$??????~$\beta$ ????~$D_\beta\Lambda_4(\Phi_\varepsilon)$,?????
	\begin{equation*}
		\left<\Lambda_4(\Phi_\varepsilon)\phi,\phi\right>_{\mathcal{W}_3^k}\leq  \frac{C}{\alpha^{3/2}}|\varepsilon|_{\mathcal{H}^{k}}|\phi|_{\mathcal{W}_3^{k}}^2.
	\end{equation*}
	Combining $(\ref{lambda3est2})$ with~$(\ref{phiestV})$, we have
	\begin{equation}
		\left<\mathcal M^{\varepsilon}_\phi(\xi,\phi),\phi\right>_{\mathcal{W}^k_3}\leq\frac{C}{\alpha^{3/2}}|\varepsilon|_{\mathcal {H}^k}|\phi|{\mathcal{W}^k_1}(|\xi|_{\mathcal{W}^k_1}+|\phi|_{\mathcal{W}^k_3}).
	\end{equation}
	We obtain all estimates in Proposition $\ref{Mest2}$.
	%??????????~$\ref{prop_est}$?????
\end{proof}

\section{Proof of main theorem}
Putting together what we gained from the preceding sections, we give the proof of Theorem~\ref{thm2} in this section.

Recall the definition of $X(s),Y(s)$ in~$(\ref{XY})$
\begin{equation}
	\begin{aligned}
		X(s)=&|\xi|_{\mathcal{W}^k_1}^2+|\xi|_{\mathcal{W}^k_2}^2+|\phi|_{\mathcal{W}^k_3}^2,\\
		Y(s)=&|\frac{1}{\rho}\partial_\beta\xi|_{\mathcal{W}^{k}_1}^{2}+|\frac{1}{\rho}D_\rho\xi|_{\mathcal{W}^{k}_1}^{2}+
		|\frac{1}{\rho}\partial_\beta\xi|_{\mathcal{W}^{k}_2}^{2}+|\frac{1}{\rho}D_\rho\xi|_{\mathcal{W}^{k}_2}^{2}+
		|\frac{1}{\rho}\partial_\beta\phi|_{\mathcal{W}^{k}_3}^{2}+|\frac{1}{\rho}D_\rho\phi|_{\mathcal{W}^{k}_3}^{2}.
	\end{aligned}
\end{equation}
We need obtain the energy estimate of~$ |\varepsilon|_{\mathcal {H}^k}$ and $X(s)$. First combine three estimates in~$(\ref{energy})$ we obtain
\begin{equation}\label{energyfinal}
	\begin{aligned}
		&\frac {d}{ds}X(s)\leq -\frac{\mu_s}{\mu}A_1-\big(\frac{\lambda_s}{\lambda}+1\big)((1+\delta)A_1+2A_2)-(\big(2+\frac{l'_{1}}{l_1}\big)A_2+\big(1+\delta+\frac{l'_{2}}{l_2}\big)A_1)\\
		&-\left<\mathcal{M}_\xi(\xi,\phi),\xi\right>_{\mathcal{W}_1^k}-\left<\mathcal{M}_\xi(\xi,\phi),\xi\right>_{\mathcal{W}_2^k}-\left<\mathcal{M}_\phi(\xi,\phi),\phi\right>_{\mathcal{W}_3^k}-C Y(s),
	\end{aligned}
\end{equation}

where
\begin{equation*}
	\begin{aligned} A_1=&\left<\bar{y}\partial_{\bar{y}}\xi,\xi\right>_{\mathcal{W}^k_1}+\left<\bar{y}\partial_{\bar{y}}\xi,\xi\right>_{\mathcal{W}_2^k}+\left<\bar{y}\partial_{\bar{y}}\phi,\phi\right>_{\mathcal{W}_3^k},\\
		A_2=&\left<\xi,\xi\right>_{\mathcal{W}_1^k}+\left<\xi,\xi\right>_{\mathcal{W}_2^k}+\left<\phi,\phi\right>_{\mathcal{W}_3^k}.\\
	\end{aligned}
\end{equation*}
By Proposition~$\ref{ode}$, we obtain$$\big|\frac{\mu_s}{\mu}\big|+\big|\frac{\lambda_s}{\lambda}+1\big|\leq \frac{C}{\alpha}|\varepsilon|_{\mathcal{H}^k}+CY(s),$$
Corollary~$\ref{l1l2cor}$ gives the estimate of time derivative part: $$-\frac{\mu_s}{\mu}A_1-\big(\frac{\lambda_s}{\lambda}+1\big)((1+\delta)A_1+2A_2)-(\big(2+\frac{l'_{1}}{l_1}\big)A_2+\big(1+\delta+\frac{l'_{2}}{l_2}\big)A_1)\leq-cX(s)+CY(s)X(s),$$ Proposition $\ref{Mest1},\ref{Mest2}$ give the estimate of $\mathcal{M}(\xi,\phi)$.
Thus we get
\begin{equation}\label{final1}
	\frac {d}{ds}X(s)\leq -c X(s)+\frac{C}{\alpha^{3/2}}|\varepsilon(s)|_{\mathcal{H}^k}X(s)+CY(s)X(s)-C_1Y(s).
\end{equation}
Combine the energy estimate of $\varepsilon$~$(\ref{varepsilonenergy})$ with the estimate of $\xi$ $(\ref{xinorm})$ together, we obtain
\begin{equation}\label{final2}
	\frac {d}{ds}|\varepsilon(s)|_{\mathcal{H}^k}^2\leq -c|\varepsilon(s)|_{\mathcal{H}^k}^2+\frac{C}{\alpha^{3/2}}|\varepsilon(s)|_{\mathcal{H}^k}^3+CY(s)|\varepsilon(s)|_{\mathcal{H}^k}^2+C_2\alpha^{-2k+1} Y(s).
\end{equation}

Putting~$(\ref{final1})$ and~$(\ref{final2})$ together, we can eliminate the positive term of $Y(s)$ in $(\ref{final2})$. Thus setting constant $C_3=\frac{C_2+1}{C_1}>0$ and
$$\mathcal{E}(s)=|\varepsilon(s)|_{\mathcal{H}^k}^2+C_3\alpha^{-2k+1}X(s),$$ 
we get
\begin{equation}\label{final}
	\begin{aligned}
		\frac {d}{ds}\mathcal{E}(s)\leq &-c\mathcal{E}(s)+\frac{C}{\alpha^{3/2}}\mathcal{E}(s)^{3/2}+CY(s)\mathcal{E}(s)+(C_2-C_1 C_3)\alpha^{-2k+1}Y(s)\\
		=&-c\mathcal{E}(s)+\frac{C}{\alpha^{3/2}}\mathcal{E}(s)^{3/2}+(C\mathcal{E}(s)-\alpha^{-2k+1})Y(s).
	\end{aligned}
\end{equation}
We take $\mathcal{E}(0)\leq \delta_0\alpha^3$, by~Gronwall inequality, we get~$(\ref{finalest})$. For~$Y(s)$, we integrate above inequality, and have
$$\int_0^\infty Y(s)ds\leq \alpha^{2k-1}\mathcal{E}(0).$$
Combine with the estimate of $|\mu_s|+|\frac{\lambda_s}{\lambda}+1|$~$(\ref{lambda_mu})$, we obtain~$(\ref{finalest2})$, this ends the proof of Theorem~$\ref{thm2}$.

\noindent \textbf{Acknowledgments:} L. Zhang is partially supported by NSFC under grant
12031012 and 11631008.

\end{document}